\documentclass{article}
\usepackage[english]{babel}
\usepackage[latin1]{inputenc}
\usepackage[T1]{fontenc}
\usepackage{amsmath,amssymb,stmaryrd,enumerate,bbm,latexsym,theorem}
\usepackage[top=2.5cm,bottom=2.5cm,right=2cm,left=2cm]{geometry}
\usepackage{authblk}
\usepackage{xcolor}
\usepackage{epsfig}
\numberwithin{equation}{section}

\def\ii{\infty}
\def\R{\mathbbm R}
\def\C{\mathbbm C}
\def\N{\mathbbm N}
\def\Z{\mathbbm Z}
\def\BU{\mathcal{U}}
\def\BC{\mathcal{C}}
\def\BE{\mathcal{E}}
\def\e{\varepsilon}
\def\ep{\epsilon}
\def\vp{\varphi}
\def\ex{{\sf e}}
\def\p{\partial}
\def\gs{\geqslant }
\def\ls{\leqslant }
\def\ddef{\overset{\rm def}{=}}

\catcode`\<=\active \def<{
\fontencoding{T1}\selectfont\symbol{60}\fontencoding{\encodingdefault}}
\catcode`\>=\active \def>{
\fontencoding{T1}\selectfont\symbol{62}\fontencoding{\encodingdefault}}
\newcommand{\assign}{:=}

\newcommand{\nobracket}{}

\newcommand{\tmem}[1]{{\em #1\/}}

\newcommand{\tmop}[1]{\ensuremath{\operatorname{#1}}}

\newenvironment{proof}{\noindent\textbf{Proof.\ }}{\hspace*{\fill}$\Box$\medskip}

\newenvironment{tmparmod}[3]{\begin{list}{}{\setlength{\topsep}{0pt}\setlength{\leftmargin}{#1}\setlength{\rightmargin}{#2}\setlength{\parindent}{#3}\setlength{\listparindent}{\parindent}\setlength{\itemindent}{\parindent}\setlength{\parsep}{\parskip}} \item[]}{\end{list}}
\newtheorem{theorem}{Theorem}[section]
\newtheorem{corollary}[theorem]{Corollary}
\newtheorem{lemma}[theorem]{Lemma}
\newtheorem{proposition}[theorem]{Proposition}
\newtheorem{remark}[theorem]{Remark}

\usepackage[hidelinks]{hyperref}

\date{}
\author{David Chiron} 
\affil{Universit{\'e} C\^{o}te d'Azur, \\ CNRS, LJAD, France \\ e-mail: david.chiron@univ-cotedazur.fr }
\author{Eliot Pacherie}
\affil{NYUAD Research Institute, New York University Abu Dhabi, \\ PO Box 129188, Abu Dhabi, UAE \\ e-mail: ep2699@nyu.edu}
\begin{document}

\title{A uniqueness result for the two vortex travelling wave in the Nonlinear Schr\"odinger equation}

\maketitle

\begin{abstract}
  For the Nonlinear Schr\"odinger equation in dimension 2, the existence of a global
  minimizer of the energy at fixed momentum has been established by
  Bethuel-Gravejat-Saut {\cite{Bet_Gra_Saut}} (see also \cite{Chi_Mar}). 
  This minimizer is a travelling wave for the Nonlinear Schr\"odinger equation. 
  For large momentums, the propagation speed is small and the minimizer behaves like 
  two well separated vortices. In that limit, we show the uniqueness of this
  minimizer, up to the invariances of the problem, hence proving the orbital stability of 
  this travelling wave. This work is a follow up to 
  two previous papers {\cite{Chi_Pac_1}}, {\cite{Chi_Pac_2}}, where we constructed and
  studied a particular travelling wave of the equation. We show a uniqueness
  result on this travelling wave in a class of functions that contains in
  particular all possible minimizers of the energy.
\end{abstract}

\section{Introduction and statement of the results}

We consider the Nonlinear Schr\"odinger equation
\begin{equation}
\tag{NLS}
 i \partial_t \Psi + \Delta \Psi - (|\Psi |^2 - 1) \Psi = 0
\end{equation}
in dimension 2 for $\Psi: \mathbbm{R}_t \times \mathbbm{R}_x^2
\rightarrow \mathbbm{C}$, also called the Gross-Pitaevskii equation without potential. 
The Nonlinear Schr\"odinger 
equation is a physical model for Bose-Einstein condensate (see \cite{Ginz-Pit_58a}, \cite{Neu_90}, \cite{RB}, 
\cite{AbidHueMetNoreBra}), superfluidity (\cite{Pismen}) and nonlinear Optics (see \cite{KL}). 
The condition at infinity for $(\tmop{NLS})$ will be
\[ 
\lvert \Psi \rvert \rightarrow 1 \quad \tmop{as} \quad | x | \rightarrow + \infty . 
\]
The $(\tmop{NLS})$ equation is associated with the Ginzburg-Landau energy
\[ E (v) \assign \frac{1}{2} \int_{\mathbbm{R}^2} | \nabla v |^2 + \frac{1}{4}
   \int_{\mathbbm{R}^2} (1 - | v |^2)^2 , \]
which is formally conserved by the $(\tmop{NLS})$ flow. 
We denote by $\mathcal{E}$ the set of functions with finite energy, that is
\[ \mathcal{E} \assign \{ u \in H^1_{\tmop{loc}} (\R^2, \C), E (u) < + \infty \} . \]

\begin{remark} 
The Cauchy problem for (NLS) is globally well-posed in the energy space: see \cite{Gerard_Cauchy1}, 
\cite{Gerard_Cauchy2}, \cite{Gallo}.
\end{remark}

Besides the energy, the momentum is another quantity formally conserved 
by the $(\tmop{NLS})$ flow, and is associated with the invariance by translation of $(\tmop{NLS})$. 
Formally, the momentum of $u$ is $ \frac{1}{2} \int_{\mathbbm{R}^2} \mathfrak{Re} ( i \nabla u \bar{u} ) \in \mathbbm{R}^2 $, 
but its precise definition requires some care in the energy space due to the condition at infinity 
(see \cite{Maris_exist} in dimension larger than two and \cite{Chi_Mar} in dimension two). 
If $ u \in 1 + \BC^\ii_c (\R^2 ) $ for instance, or if $ u$ is a travelling wave tending to $1$ at infinity, 
then the expression of the momentum reduces to
\[ 
 \vec{P} (u) 
 = ( P_1( u) , P_2( u) ) 
 = { \frac{1}{2} } \int_{\R^2} \mathfrak{Re} \big( i \nabla u ( \bar{u} - 1 ) \big) .
\]
In addition to the translation invariance, the $(\tmop{NLS})$ equation is also phase shift 
invariant, that is invariant by multiplication by a complex of modulus one, and 
rotation invariant.

\subsection{Travelling waves for $(\tmop{NLS})$}

Following the works in the physical literature of Jones and Roberts 
(see {\cite{Jon_Rob}}, {\cite{Jon_Put_Rob}}), there has been a large amount of
mathematical works on the question of existence and properties of travelling
wave solutions in the $(\tmop{NLS})$ equation, that are solutions of
\[ 0 = (\tmop{TW}_c) (u) \assign - i c \partial_{x_2} u - \Delta u - (1 - | u |^2) u \]
for some $ c > 0 $, corresponding to particular solutions of $(\tmop{NLS})$ of the form 
$ \Psi (t, x) = u ( x_1 , x_2 + ct ) $ (due to the rotational invariance, we may always 
assume that the traveling wave moves along the direction $ - \overrightarrow{e_2} $). 
We refer to {\cite{Bet_Gra_Saut_TW}} for an overview on 
these problems in several dimensions. A natural approach is to look at the
minimizing problem for $ \mathfrak{p} > 0$
\[ E_{\min} ( \mathfrak{p} ) \assign \inf_{u \in \mathcal{E}} \{ E (u), P_2 (u) = \mathfrak{p} \} .
\]
It was shown by Bethuel-Gravejat-Saut in {\cite{Bet_Gra_Saut}} that there exists a
minimizer to this problem.

\begin{theorem}[{\cite{Bet_Gra_Saut}}] 
\label{th_BGS}
  For any $ \mathfrak{p} > 0$, there exists a non constant function $ u_\mathfrak{p} \in \mathcal{E}$
  and $ c (u_\mathfrak{p}) > 0 $ such that $ P_2(u_\mathfrak{p}) = \mathfrak{p}$, 
  $u_\mathfrak{p}$ is a solution to $(\tmop{TW}_{c(u_\mathfrak{p})}) (u_\mathfrak{p}) = 0$ and
  \[ E (u_\mathfrak{p}) = E_{\min} (\mathfrak{p}) . \]
  Furthermore, any minimizer for $ E_{\min} (\mathfrak{p}) $ is, up to a translation in $ x_1 $, even in $ x_1 $.
\end{theorem}

The strategy is to look at the corresponding minimization problem on tori (this avoids the problems 
with the definition of the momentum) larger and larger, and then pass to the limit. 
For the minimizing problem $ E_{\min} (\mathfrak{p}) $, the compactness of minimizing sequences 
has been shown later on in \cite{Chi_Mar} for the natural semi-distance on $ \mathcal{E} $
\[
 D_0 ( u, v) \assign \| \nabla u - \nabla v \|_{L^2(\R^2)} + \| \lvert u \rvert - \lvert v \rvert \|_{L^2(\R^2)} .
\]

\begin{theorem}[\cite{Chi_Mar}]
\label{th_CM}
  For any $ \mathfrak{p} > 0 $, and any minimizing sequence $ (u_n)_{n \in \N} $ for $ E_{\min} (\mathfrak{p}) $, 
  there exists a subsequence $ (u_{n_j})_{j \in \N} $, a sequence of translations $ ( y_j )_{j \in \N} $ and a non constant 
  function $ u_\mathfrak{p} \in \mathcal{E} $ such that $ D_0 ( u_{n_j} , u_\mathfrak{p} ) \to 0 $, 
  $ P_2(u_{n_j}) \to P_2(u_\mathfrak{p}) = \mathfrak{p} $ and $ E(u_{n_j}) \to E(u_\mathfrak{p}) = E_{\min} (\mathfrak{p}) $ 
  as $ j \to + \ii $. In particular, there exists $ c (u_\mathfrak{p}) > 0 $ such that $ P_2( u_\mathfrak{p}) = \mathfrak{p}$, 
  $u_\mathfrak{p}$ is a solution to $(\tmop{TW}_{c(u_\mathfrak{p})}) (u_\mathfrak{p}) = 0$ and
  \[ E (u_\mathfrak{p}) = E_{\min} (\mathfrak{p}) . \]
   Furthermore, the set $ \mathcal{S}_\mathfrak{p} \assign \{ v \in \mathcal{E} , \, P_2 ( v ) = \mathfrak{p} \ and \ E( v) = E_{\min} (\mathfrak{p}) \} $ 
  of minimizers for $ E_{\min} (\mathfrak{p}) $ is orbitally stable for the semi-distance $ D_0 $.
\end{theorem}

An open and difficult question is to show, up to the invariances of the
problem, the uniqueness of the energy minimizer at fixed momentum. In other words, the problem is to 
determine if $ \mathcal{S}_\mathfrak{p} $ consists of a single orbit under phase shift and space translation, that 
is: do we have, for some minimizer $ U_\mathfrak{p} $,
\[
 \mathcal{S}_\mathfrak{p} = \{ U_\mathfrak{p} (. - X) \ex^{i \gamma} , \, \gamma \in \R , \, X \in \R^2 \}?
\]
The main consequence of our work is to solve this open problem of uniqueness for large momentum.

\begin{theorem}
  \label{th2}There exists $ \mathfrak{p}_0 > 0$ such that, for any $ \mathfrak{p} > \mathfrak{p}_0 $, 
  if $u, v \in \mathcal{E}$ with $ P_2(u) = P_2(v) = \mathfrak{p} $ satisfy
  \[ E (u) = E (v) = E_{\min} (\mathfrak{p}), \]
  then, there exist $X \in \mathbb{R}^2 $ and $ \gamma \in \mathbb{R}$ such that
  \[ u = v (. - X) \ex^{i \gamma} . \]
\end{theorem}

In fact, we will be able to show slightly stronger results than Theorem \ref{th2}, 
see Theorem \ref{th8} below. 

Even though we focus on the Ginzburg-Landau nonlinearity, it is plausible that our results 
hold true (still for large momentum) for more general nonlinearities, provided vortices exist. 
For the Ginzburg-Landau (cubic) nonlinearity, it is also possible that uniqueness of minimizers 
holds true for $ E_{\min} (\mathfrak{p}) $ for any $ \mathfrak{p} > 0 $. However, the numerical 
results given in \cite{Chi_Sch1} suggest that this may no longer be the case for more 
general nonlinearities.

\bigbreak
In the analysis of the minimization problem in \cite{Bet_Gra_Saut} (and also \cite{Chi_Mar}), 
the following properties of $ E_{\min} $ play a key role.

\begin{proposition} [\cite{Bet_Gra_Saut}]
\label{propriete_Emin} 
The function $ E_{\min } : \R_+ \to \R $ is concave, nondecreasing and $ \sqrt{2} $-Lipschitz continuous. 
In addition, there exists $K\gs 0 $ such that, for any $ \mathfrak{p} \gs 1 $, we have
\begin{equation}
\label{upper_bound_mini}
 E_{\min}(\mathfrak{p} ) 
 \ls 2 \pi \ln \mathfrak{p} + K .
\end{equation}
\end{proposition}

\subsection{A smooth branch of travelling waves for large momentum}

There have been several ways of constructing travelling waves of the
$(\tmop{NLS})$ equation, with different approaches. For instance, we may use variational methods, 
such as a mountain pass argument in {\cite{Bet_Saut}} and in \cite{Bel_Rui}, or by minimizing the 
energy at fixed kinetic energy (\cite{Bet_Gra_Saut}, \cite{Chi_Mar}). 
Also, we have constructed in {\cite{Chi_Pac_1}} a travelling
wave by perturbative methods, taking for ansatz a pair of vortices, by following the 
Lyapounov-Schmidt reduction method as initiated in \cite{delP_Kow_Mus}. Vortices
are stationary solutions of $(\tmop{NLS})$ of degrees $n \in
\mathbb{Z}^{\ast}$ (see \cite{Ginz-Pit_58a}, \cite{Neu_90}, \cite{Weinstein_vortex}, 
\cite{HerHer_94}, \cite{ChenElliottQi}):
\[ V_n (x) = \rho_n (r) \ex^{i n \theta}, \]
where $x = r \ex^{i \theta}$, solving
\[ \left\{\begin{array}{l}
     \Delta V_n - (| V_n |^2 - 1) V_n = 0\\
     | V_n | \rightarrow 1 \tmop{as} | x | \rightarrow \infty .
   \end{array}\right. \]
In the previous paper {\cite{Chi_Pac_1}}, we constructed solutions of
$(\tmop{TW}_c)$ for small values of $c > 0$ as a perturbation of two
well-separated vortices (the distance between their centers is large when $c$
is small). We have shown the following result.

\begin{theorem}[{\cite{Chi_Pac_1}}, Theorem 1.1 and \cite{Chi_Pac_2}, Proposition 1.2]
  \label{th1}There exists $c_0 > 0$ a small constant such that for any $0 < c
  \leqslant c_0$, there exists a solution of $(\tmop{TW}_c)$ of the form
  \[ Q_c = V_1 (. - d_c \overrightarrow{e_1}) V_{- 1} (. + d_c
     \overrightarrow{e_1}) + \Gamma_c, \]
  where $d_c = \frac{1 + o_{c \rightarrow 0} (1)}{c}$ is a $C^1$ function
  of c. This solution has finite energy, that is $ Q_c \in \mathcal{E} $,  
  and $Q_c \rightarrow 1$ at infinity.
  
  Furthermore, for all $ 2 < p \leqslant + \infty $, there exists $c_0 (p) > 0$
  such that, if $0 < c \leqslant c_0 (p)$, for the norm
  \[ \| h \|_p \assign \| h \|_{L^p (\mathbb{R}^2)} + \| \nabla h \|_{L^{p -
     1} (\mathbb{R}^2)} \]
  and the space $X_p \assign \{ f \in L^p (\mathbb{R}^2), \nabla f \in L^{p -
  1} (\mathbb{R}^2) \}$, one has
  \[ \| \Gamma_c \|_p = o_{c \rightarrow 0} (1) . \]
  In addition,
  \[ c \mapsto Q_c - 1 \in C^1 (] 0, c_0 (p) [, X_p), \]
  with the estimate
  \[ \left\| \partial_c Q_c + \left( \frac{1 + o_{c \rightarrow 0} (1)}{c^2}
     \right) \partial_d (V_1 (. - d \overrightarrow{e_1}) V_{- 1} (. + d
     \overrightarrow{e_1}))_{| d = d_c \nobracket} \right\|_p = o_{c
     \rightarrow 0} \left( \frac{1}{c^2} \right) . \]
  Finally, we have
\[
 \frac{d}{dc} \big( P_2 (Q_c) \big) = \frac{- 2 \pi + o_{c \rightarrow 0} (1)}{c^2} < 0 ,
\]
hence the $ C^1 $ mapping 
\[
 \mathcal{P} : ] 0, c_0 ] \ni c \mapsto P_2 (Q_c) \in \mathbbm{R} 
\]
is a strictly decreasing diffeomorphism from $ ]0 ,c_0] $ onto $ [ P_2( Q_{c_0} ) , +\ii [ $.
\end{theorem}

\begin{remark} With the same kind of approach, \cite{Liu_Wei} also provides 
an existence result of travelling waves for $(\tmop{NLS}) $, including some cases with 
more than two vortices. Our result has the advantage of showing the smoothness of the branch with respect to the speed. 
In particular, with the last part of Theorem \ref{th1}, we see that we may also parametrize the 
branch $ c \mapsto Q_c $ by its momentum $ \mathcal{P} $.
\end{remark}

It is conjectured that all these constructions yield the same branch of
travelling waves (for large momentum) when they are all defined, and that they are the solutions 
numerically observed in \cite{Jon_Rob} and \cite{Chi_Sch1} for more general 
nonlinearities (see also \cite{Chi_Sch2}). 
We will show here that the construction of Theorem \ref{th1} are the unique, up to the natural 
translation and phase invariances, constrained energy minimizers.

\subsection{A uniqueness result for symmetric functions}

We have shown in {\cite{Chi_Pac_2}} several coercivity results for the travelling waves constructed in 
Theorem \ref{th1}. This will allow us to show the following uniqueness result for symmetric functions 
close to the branch constructed in Theorem \ref{th1}. There, for $d \in \mathbb{R}$, we use the notation $\tilde{r}_d = \tmop{min} (|.- d \overrightarrow{e_1}|,|.+d \overrightarrow{e_1}|)$.

\begin{proposition}
  \label{locC1} 
  There exists $ \lambda_* > 1 $ such that, for any $ \lambda \gs \lambda_* $, there exists 
  $ \varepsilon ( \lambda) > 0 $ such that 
  if a function $u \in \mathcal{E}$ satisfies
  \begin{enumerate}
    \item $\forall (x_1, x_2) \in \mathbb{R}^2, u (x_1, x_2) = u (- x_1, x_2)$,
    
    \item $u = V_1 (x - d \overrightarrow{e_1}) V_{- 1} (x + d
    \overrightarrow{e_1}) + \Gamma$, with $d > \frac{1}{\varepsilon ( \lambda)  }$, $\|
    \Gamma \|_{L^{\infty} (\{ \tilde{r}_d \leqslant 2 \lambda \})} \leqslant \varepsilon ( \lambda) $,
    
    \item $\| | u | - 1 \|_{L^{\infty} (\{ \tilde{r}_d \geqslant \lambda \})} \leqslant \frac{1}{ \lambda_*} $,
    
    \item $(\tmop{TW}_c) (u) = 0$ for some $c > 0$ such that $| d c - 1 | \leqslant \varepsilon ( \lambda)  $,
  \end{enumerate}
  then, there exist $X \in \mathbb{R} $ and $ \gamma \in \mathbb{R}$ such that 
  $u = Q_c (. - X \vec{e}_2) \ex^{i \gamma}$, where $Q_c$ is defined in Theorem
  \ref{th1}.
\end{proposition}

\begin{remark} In view of the symmetry assumption, we may replace the second hypothesis by
\[
 \| u - V_1 (\cdot - d \overrightarrow{e_1}) \| _{L^\infty ( B ( d \overrightarrow{e_1} , 2 \lambda )) } 
 \ls \varepsilon (\lambda ) .
\]
\end{remark}

We will discuss the main arguments of the proof of Proposition \ref{locC1} in
the next section. This result can be seen as a local uniqueness result, but
the uniqueness turns out to be in a rather large class of function. Indeed, two functions that
satisfies the hypothesis of Proposition \ref{locC1} can be very far from each
other, for two main reasons. First, in condition {\it 2.}, the vortices that compose one of them
have no reason to be close to the ones composing the other function since $d$ depends on $u$: 
their centers $\pm d \overrightarrow{e_1}$ only need to satisfy $| d c - 1 |
\leqslant \varepsilon (\lambda)$, but for instance both $d = \frac{1}{c}$ and $d =
\frac{1}{c} + \frac{1}{\sqrt{c}}$ satisfy these conditions for $c > 0$ small enough at fixed $\lambda$. 
Secondly, we only have that far from the vortices, the modulus is close to $1$ from condition {\it 3.}, but we have no information 
on the phase. The proof of Proposition \ref{locC1} will rely on methods used in \cite{Chi_Pac_2} 
in order to prove some coercivity, and we shall need to be very precise 
to take into account all these cases.

A way to see that Proposition \ref{locC1} is a strong unicity result is that it implies the following local
uniqueness result in $L^{\infty}$ for even functions in $x_1$.

\begin{corollary}
  \label{CP3cor}There exist $c_0, \varepsilon > 0$ such that, for $0 < c < c_0$, 
  if a function $u \in \mathcal{E}$ satisfies 
  \begin{enumerate}
  \item 
  $\forall (x_1, x_2) \in \mathbb{R}^2, u (x_1, x_2) = u (- x_1, x_2)$,
  \item $(\tmop{TW}_c) (u) = 0$ in the distributional sense,
  \item $ \| u - Q_c \|_{L^{\infty} (\mathbb{R}^2)} \leqslant \varepsilon $, 
  \end{enumerate}
  then, there exist $X \in \mathbb{R}$ and $ \gamma \in \mathbb{R}$ such that 
  $u = Q_c (. - X \vec{e}_2) \ex^{i \gamma}$.
\end{corollary}

We may now state our main result. It establishes that any travelling wave solution 
which is even in $ x_1 $ and within $\mathcal{O} (1) $ of the minimizing energy must be, 
for large momentum, the $ Q_c $ travelling wave constructed in Theorem \ref{th1}, 
up to the natural translation and phase invariances.

\begin{theorem}
  \label{th8}
  For any $ \Lambda_0 > 0$ there exists $ \mathfrak{p}_0 (\Lambda_0) > 0$ such that, 
  if $u \in \mathcal{E}$ satisfies
 \begin{enumerate}
 \item $ \forall (x_1, x_2) \in \mathbb{R}^2, u (x_1, x_2) = u (- x_1, x_2)$,
 \item $ (\tmop{TW}_c) (u) = 0$ for some $c > 0$,
 \item $ P_2 (u) \gs \mathfrak{p}_0 (\Lambda_0) $,
 \item $ E (u) \leqslant 2 \pi \ln P_2 (u) + \Lambda_0 $,
 \end{enumerate} 
  then, there exist $X \in \mathbb{R} $ and $ \gamma \in \mathbb{R} $ such that
  \[ u = Q_c (. - X \vec{e}_2) \ex^{i \gamma}, \]
  where $Q_c$ is defined in Theorem \ref{th1}. 
  In particular, $ P_2( u) = \mathcal{P} ( c ) $ (where $\mathcal{P}$ is defined in Theorem \ref{th1}). 
\end{theorem}

Section \ref{sec3} is devoted to the proof of this result. We show there that a function satisfying the hypothesis of Theorem \ref{th8} also satisfies the hypothesis of Proposition \ref{locC1}.
Our result applies in particular to the constraint minimizers for the problem $ E_{\min} (\mathfrak{p}) $, 
for large $ \mathfrak{p} $. 

\begin{corollary} 
\label{coro_bore}
There exist $ \mathfrak{p}_0 > 0 $ such that, for any $ \mathfrak{p} \gs \mathfrak{p}_0 $, any minimizer $ U $ 
for $ E_{\min} (\mathfrak{p}) $, there exists $ \gamma \in \R $ and $ X \in \R^2 $ such that, 
with $ c = \mathcal{P}^{-1} (\mathfrak{p} ) $,
\[
 U = Q_c ( \cdot - X ) \ex^{ i \gamma } .
\]
Moreover, $ (\tmop{TW}_c) ( U ) = 0 $.
\end{corollary}

\begin{proof}
By a first translation in $ x_1 $, we may assume, by Theorem \ref{th_BGS}, that this 
minimizer $U$ is even in $ x_1 $. By Proposition \ref{propriete_Emin}, the last hypothesis 4 of 
Theorem \ref{th8} is satisfied hence we may translate in $x_2 $ and use phase shift and 
get that this minimizer $U$ is $ Q_c $. Necessarily, $P_2 ( U) = \mathfrak{p} = P_2 (Q_c ) $, 
thus $ c = \mathcal{P}^{-1} (\mathfrak{p} ) $.
\end{proof}

Theorem \ref{th2} is a direct consequence of this corollary.
This allows to derive several interesting consequences on the function $ E_{\min} $. 
This also shows that the branch of Theorem \ref{th1} coincides with the global
energy minimizer at fixed momentum (up to translation and phase shift).

\begin{theorem}
  \label{th9}
  There exists $ c_* > 0 $ such that, for $ 0 < c \ls c_* $, $ Q_c $ is a minimizer for 
  $ E_{\min} ( P_2 ( Q_c ) ) $. Moreover, there exists $ \mathfrak{p}_0 > 0 $ such that the following statements hold.

 \begin{enumerate}
    \item The function $ E_{\min} $ is of class $ C^2 $ in $ [ \mathfrak{p}_0 , +\ii [ $ and 
    \[ 
    0 > E_{\min}'' ( \mathfrak{p} ) \sim - \frac{2\pi}{\mathfrak{p}^2 } , 
    \quad \quad 
    0 < E_{\min}' ( \mathfrak{p} ) \sim \frac{2\pi}{\mathfrak{p} } , 
    \quad \quad 
    E_{\min} ( \mathfrak{p} ) = 2\pi \ln \mathfrak{p} + \mathcal{O}(1)
    . \]
    \item For $ \mathfrak{p} \gs \mathfrak{p}_0 $
    $ \mathcal{S}_\mathfrak{p} = \{ Q_{\mathcal{P}^{-1} (\mathfrak{p})} (. - X) \ex^{i \gamma} , \, \gamma \in \R , \, X \in \R^2 \} $, 
    hence, for any $ \mathfrak{p} \gs \mathfrak{p}_0 $, $ E_{\min}' ( \mathfrak{p} ) $ is the speed of any 
    minimizer for $ E_{\min} ( \mathfrak{p} ) $.
    \item For any $ \mathfrak{p} \gs \mathfrak{p}_0 $, $ Q_{ \mathcal{P}^{-1} ( \mathfrak{p} ) } $ is orbitally stable 
    for the semi-distance $ D_0 $ (or, equivalently, for $ 0 < c \ls c_* $, $ Q_c $ is orbitally stable 
    for the semi-distance $ D_0 $).
    \item For $ \mathfrak{p} \gs \mathfrak{p}_0 $ and any minimizer $u$ for $ E_{\min} (\mathfrak{p} ) $, then, 
    up to a space translation and a phase shift, $u$ enjoys the symmetry
    \[
    \forall (x_1, x_2 ) \in \R^2 , \quad \quad 
    u ( x_1, -x_2 ) = \bar{u} (x_1, x_2 ) 
    \]
    in addition to the symmetry in $ x_1 $.
    
    \item For any $ \Lambda > 0$, there exists $ \mathfrak{p}_0 (\Lambda) > 0$ such that, 
    if $u \in \mathcal{E}$ satisfies $(\tmop{TW}_c) (u) = 0$ for some $c > 0$,
    $ P_2 (u) \geqslant \mathfrak{p}_0 (\Lambda )$ and $u$ is even in $x_1$, then either
    $ E (u) = E_{\min} ( P_2 (u) ) $, or $E (u) \geqslant E_{\min} (P_2 (u) ) + \Lambda $.
     \end{enumerate}
\end{theorem}

\begin{proof} By Theorems \ref{th_BGS} and \ref{th_CM}, we have existence of at least one minimizer 
$ U_\mathfrak{p} $ for $ E_{\min} (\mathfrak{p} ) $, whatever is $ \mathfrak{p} > 0 $. 
For large $\mathfrak{p} $, by applying Corollary \ref{coro_bore}, we have $ U_\mathfrak{p} = Q_c ( \cdot - X ) \ex^{i\gamma} $ 
for some $ X \in \R^2 $ and $ \gamma \in \R $, thus proving that $ Q_c $ is a minimizer for $ E_{\min} (\mathfrak{p} ) $ 
and that $ P_2 ( Q_c ) = \mathcal{P} ( c) = \mathfrak{p} $. 

For {\it 1.}, it suffices to notice that, in view of Corollary \ref{coro_bore} applied 
to any minimizer (we have existence by Theorems \ref{th_BGS} and \ref{th_CM}) 
$ E_{\min} (\mathfrak{p} ) = E ( Q_{ \mathcal{P}^{-1} (\mathfrak{p} ) } ) $. 
We then conclude by using that $ \mathcal{P} $ is a $ C^1 $ diffeomorphism and that 
$ c \mapsto E (Q_c) $ is also of class $ C^1 $ (see \cite{Chi_Pac_2}, Proposition 1.2) 
that $ E_{\min} $ is of class $ C^1 $ in $ [\mathfrak{p}_0 , +\ii [ $ and that
\[
 E_{\min}' (\mathfrak{p} ) 
 = \frac{d}{ dc} E ( Q_{ c } )_{|c = \mathcal{P}^{-1}(\mathfrak{p})} \times \frac{1}{ \mathcal{P}' ( \mathcal{P}^{-1} (\mathfrak{p} )) } 
 = \mathcal{P}^{-1}(\mathfrak{p}) ,
\]
in view of the Hamilton like relation (formally shown in \cite{Jon_Rob} and rigorously proved for the 
branch constructed in Theorem \ref{th1} in \cite{Chi_Pac_2})
\[
 \frac{d}{ dc} E ( Q_{ c } ) = c \frac{d}{ dc} P_2 ( Q_{ c } ) .
\]

Since $ \mathcal{P} $ is a $ C^1 $ diffeomorphism, we deduce that $ E_{\min}' $ is of class $ C^1 $. 
The asymptotics for $ E_{\min}' $ and $ E_{\min}'' $ then follow from Proposition 1.2 in \cite{Chi_Pac_2}. 
Integration would yield $ E_{\min} ( \mathfrak{p} ) \sim 2\pi \ln \mathfrak{p} $, but we may slightly improve 
this estimate. Indeed, Proposition \ref{propriete_Emin} gives $ E_{\min} ( \mathfrak{p} ) \ls 2\pi \ln \mathfrak{p} + \mathcal{O} (1) $, 
and the lower bound is a straightforward consequence of Theorem \ref{Th_Sandier_lower} $ (i) $ and 
the study in subsection \ref{sec:lower}.

Statement {\it 2.} is a rephrasing of Corollary \ref{coro_bore} combined with the existence of at least  
one constrained minimizer. Statement {\it 3.} is then a direct consequence of Theorem \ref{th_CM}. 
Statement {\it 4.} simply follows from the fact that $ Q_c $ enjoys by construction this symmetry 
(see \cite{Chi_Pac_1}). Finally, statement {\it 5.} is also a rephrasing of Theorem \ref{th8}. 
\end{proof}

\begin{remark} 
Concerning the stability stated in statement {\it 3.} in the above theorem, we quote the work \cite{Lin_Zeng}, 
where a linear "spectral" stability result is proved (through {\it ad hoc} hypotheses that have been checked in \cite{Chi_Pac_2}), 
namely that the linearized equation $ i \p_t v = L_{Q_c} (v) $ does not have exponentially growing solutions 
(in $ \dot{H}_ 1(\R^2; \C) $, say). Statement {\it 3.} in the above theorem does not rely on the result in \cite{Lin_Zeng}, 
and is for the nonlinear (orbital) stability (following the Cazenave-Lions approach).
\end{remark}

Let us conclude this section with several comments on our result. 
First, let us explain the relevance of the symmetry hypothesis, namely 
that we restrict to mappings even in $x_1 $. This symmetry is used in the 
coercivity of the branch of Theorem \ref{th1}, along the following arguments. The
quadratic form around the travelling wave $Q_c$ is decomposed in three areas,
close to the two vortices, and far from them. In the latter region, the
coercivity can be shown without any orthogonality condition. Close to the
vortices, the quadratic form is close to the one of a single vortex, that has been
studied in {\cite{delP_Fel_Kow}}. Its coercivity requires three orthogonality
conditions, two for the translation, and one for the phase. Therefore, we can
show the coercivity of the full quadratic form with six orthogonality
conditions, three for each vortex. However, the family of travelling waves
of Theorem \ref{th1} has only five parameters (two for the speed, two for the
translation, and one for the phase). The symmetry is then used to reduce the
problem to three orthogonality conditions into a family with three parameters.
With this symmetry, both orthogonality conditions on the phase for the two
vortices become the same condition. It is possible to prove a coercivity result with
only five orthogonality conditions without symmetry (see \cite{Chi_Pac_2}), 
but then the coercivity constant goes to $0$ when $ c \rightarrow 0 $. 
This would pose a problem for the uniqueness result. The last statement in Theorem 
\ref{th9} shows that, when restricting ourselves to symmetric travelling waves, 
there is an energy threshold under which there is no travelling wave 
except the $Q_c $ branch. 

Secondly, the proof of the fact that $ Q_c $ is a minimizer of the energy for fixed 
momentum relies on the existence of such minimizers. In particular, we have not 
been able to use our coercivity results in \cite{Chi_Pac_2} in order to prove 
directly that $ Q_c $ is orbitally stable (for small $ c$). 

Thirdly, the symmetry in $ x_2 $ for the minimizers (statement {\it 4)} is 
established as a consequence of the uniqueness result and not in itself. 
Notice that the numerical studies in \cite{Jon_Rob}, \cite{Chi_Sch1} and 
\cite{Chi_Sch2} assume the two symmetries.

\subsection{The travelling wave $ Q_c $ and two other variational characterizations}

Before providing other variational characterizations of $ Q_c $, we have to define a 
distance on the energy space $ \mathcal{E} $. One can use (see \cite{Gerard_Cauchy2}) 
\[
 D_\mathcal{E} (\psi_1 , \psi_2 ) \assign 
 \| \psi_1 - \psi_2 \|_{L^2(\R^2) + L^\ii(\R^2) } 
 + \| \nabla \psi_1 - \nabla \psi_2 \|_{L^2(\R^2) } + \| \lvert \psi_1 \rvert - \lvert \psi_2 \rvert \|_{L^2(\R^2)} ,
\]
which is adapted to the Cauchy problem. Actually, we may also use the pseudo-distance\footnote{$D_0 ( \psi_1 , \psi_2 ) $ 
is zero if and only if $ \psi_2 - \psi_1 $ is constant with $ | \psi_1 | -1 = | \psi_2 | -1 \in L^2 (\R^2 ) $.}
\[
 D_0 ( \psi_1 , \psi_2 ) \assign 
 \| \nabla \psi_1 - \nabla \psi_2 \|_{L^2(\R^2) } + \| \lvert \psi_1 \rvert - \lvert \psi_2 \rvert \|_{L^2(\R^2)} ,
\]
Is it shown in \cite{Chi_Mar}, Corollary 4.13 there, that both the energy $E$ and the momentum 
$ P_2 $ are continuous for the distance $ D_\mathcal{E} $, and actually even for the pseudo-distance $ D_0 $.

\bigbreak 

{\bf The travelling wave $ Q_c $ as a mountain pass solution.} 
Thanks to the results in Theorem \ref{th9}, it is easy to show that we have locally, 
near $ Q_c $, a mountain-pass geometry. Indeed, let $ c_* > 0 $ be small, and define
\[
 \Upsilon_{c_*} 
 \assign \{ \upsilon : [-1 , +1 ] \to \BE \ {\rm continuous}, \, v(-1) = Q_{3c_*/2} , \, v(+1) = Q_{c_*/2} \}
\]
the set of continuous paths from $ Q_{3c_*/2} $ to $ Q_{ c_*/2} $ in $ \BE $. Then, we claim that
\begin{equation}
\label{critic_val}
 \inf_{ \upsilon \in \Upsilon_{c_*} } \max_{ t \in [-1,+1] } ( E - c_* P_2 ) ( \upsilon(t) ) 
 = ( E - c_* P_2 ) ( Q_{c_*} ) .
\end{equation}
Indeed, let $ \upsilon \in \Upsilon_{c_*} $. By the intermediate value theorem, there 
exists $ t_* \in [ -1, + 1 ] $ such that $ P_2 ( \upsilon(t) ) = P_2 (Q_{c_*} ) $ ($c \mapsto P_2(Q_c)$ is a $C^1$ function, see 
Proposition 1.2 in \cite{Chi_Pac_2}). Since 
$ Q_{c_*} $ is a minimizer for $ E_{\rm min} ( Q_{c_*} ) $, we infer
\[
 \max_{ t \in [-1,+1] } ( E - c_* P_2 ) ( \upsilon (t) ) 
 \gs E ( v(t_*) ) - c_* P_2 ( Q_{c_*} ) 
 \gs E ( Q_{c_*} ) - c_* P_2 ( Q_{c_*} ) .
\]
Moreover, considering the particular $ \BC^1 $ path $ \upsilon_* : [-1, +1 ] \to \BE $ 
defined by $ \upsilon ( t) \assign Q_{c_* - t c_*/2 } $, we see that
\[
 \frac{d}{dt} ( E - c_* P_2 ) ( \upsilon_* (t) ) 
 = - \frac{c_*}{2} \Big( \frac{d}{dc} E(Q_c) - c_* \frac{d}{dc} P_2(Q_c) \Big)_{|c = c_*- t c_*/2 } 
 = \frac{c_*^2 t }{4} \Big( \frac{d}{dc} P_2(Q_c) \Big)_{|c = c_*- t c_*/2 } 
\]
in view of the Hamilton group relation $ \frac{d}{dc} E(Q_c) = c \frac{d}{dc} P_2(Q_c) $ (see 
Proposition 1.2 in \cite{Chi_Pac_2}). Since $ \frac{d}{dc} P_2(Q_c) < 0 $, we deduce that 
$ ( E - c_* P_2 ) ( \upsilon_*(t) ) $ increases in $ [-1 , 0 ] $ and decreases in $ [0,+1 ] $, hence 
has maximal value $ E ( Q_{c_*} ) - c_* P_2 ( Q_{c_*} ) $, as wished.

Furthermore, by the asymptotics given in \cite{Chi_Pac_2} and the above mentioned 
Hamilton group relation $ \frac{d}{dc} E(Q_c) = c \frac{d}{dc} P_2(Q_c) $, we have
\[
 ( E - c_* P_2 ) ( Q_{c_*} ) - ( E - c_* P_2 ) ( Q_{c_*/2} ) 
 = \int_{c_*/2}^{c_*} (c - c_*) \frac{d}{dc} P_2(Q_c) \, dc
 > 0
\]
since $ c - c_* < 0 $ and $ \frac{d}{dc} P_2(Q_c) < 0 $. Similarly, we prove that 
$ ( E - c_* P_2 ) ( Q_{c_*} ) - ( E - c_* P_2 ) ( Q_{3c_*/2} ) < 0 $.

We now claim that if $ u \in \mathcal{E} $ is such that $ (\tmop{TW}_{c_*}) (u) = 0 $ and
\begin{equation}
\label{critic_val_atteinte}
 ( E - c_* P_2 ) ( u ) 
 = \inf_{ \upsilon \in \Upsilon_{c_*} } \max_{ t \in [-1,+1] } ( E - c_* P_2 ) ( \upsilon(t) ) 
 = ( E - c_* P_2 ) ( Q_{c_*} ) ,
\end{equation}
by \eqref{critic_val}, that is if $u $ is a critical point of $ E - c_* P_2 $ at the good critical value, 
then we must have $ P_2 ( u) = P_2 ( Q_{c_*} ) $. Indeed, by the Pohozaev identity \eqref{Pohoz}, 
we have
\[
 c_* P_2( u) = \frac{1}{2} \int_{ \R^2 } ( 1- \lvert u \rvert^2 )^2 \, dx 
 \gs 0 ,
\]
hence $ P_2 (u) \gs 0 $. Furthermore, we know that $ E_{\min} $ is concave in $ \R_+ $ 
(Proposition \ref{propriete_Emin}), and that $ E_{\min} $ is of class $ C^1 $ and strictly 
concave on $ [ \mathfrak{p}_0, +\ii [ $ (by statement {\it 1.} of Theorem \ref{th9}). 
Therefore, if $ P_2 ( u) \neq P_2 ( Q_{c_*} ) $, then
\begin{align*}
 E(u) \gs E_{\min}(P_2(u)) 
 & > E_{\min} ( P_2 ( Q_{c_*} )) + E_{\min} '( P_2 ( Q_{c_*} ) ) \big( P_2 ( u) - P_2 ( Q_{c_*} ) \big) 
 \\ & = E ( Q_{c_*} ) + c_* \big( P_2 ( u) - P_2 ( Q_{c_*} ) \big) ,
\end{align*}
in contradiction with \eqref{critic_val_atteinte}. 

As a consequence, we have
\[
 E( u) = E ( Q_{c_*} ) = E_{\min} ( P_2 ( u) ) = E_{\min} ( P_2 (Q_{c_*} ) ) ,
\]
implying that $ u $ is a minimizer for $ E_{\min} ( P_2 (Q_{c_*} ) ) $, hence there exist 
$ \gamma \in \R $ and $ X \in \R^2 $ such that $ u = Q_{c_*} ( \cdot - X ) \ex^{i \gamma } $, 
hence proving a uniqueness result for mountain pass type travelling wave solutions. 
However, stating rigorously a useful uniqueness result for this kind of variational solution is not 
so easy: in \cite{Bet_Saut}, the mountain pass is implemented in the space $ 1 + H^1 (\R^2 ) $ 
whereas we know (by the result in \cite{Gra_asy}) that the nontrivial traveling wave 
do not belong to this affine space; in \cite{Bel_Rui}, the solution is constructed by 
working first on $ [-N,+N ] \times \R $ and then passing to the limit, and it is then not 
immediate to compute the functional $ E - c P $ on the solution; in addition, the method 
does not provide easily some explicit bounds on the energy or the momentum. 
We shall then not go further in this discussion even though the previous arguments 
indicate that mountain pass solutions are (at least for small $c$) only the orbit of $Q_c$.

\bigbreak

{\bf The travelling wave $ Q_c $ as a minimizer of $ E - c P_2 $ for fixed kinetic energy.} 
In \cite{Chi_Mar}, for $ \kappa \gs 0 $, the following variational problem is investigated: 
\[
 I_{\min} ( \kappa ) = \inf \Big\{ \frac{1}{4} \int_{\R^2} ( 1 - \lvert v \rvert^2 )^2 \, dx - P_2 (v), \, 
 v \in \mathcal{E} \textrm{ s.t. } \frac{1}{2} \int_{\R^2} | \nabla v |^2 \, d x = \kappa \Big\} .
\]
Any minimizer $v $ for $ I_{\min} ( \kappa ) $ is such that there exists $c >0 $ satisfying 
$ (\tmop{TW}_{c}) ( v ( \cdot / c ) ) = 0 $. In 2d and for the Ginzburg-Landau nonlinearity, 
existence of minimizers for $ \kappa > 0 $ is established in Theorem 1.2 there. 
Furthermore, it is shown in \cite{Chi_Mar} (see Proposition 8.4 there) that if $ \mathfrak{p} > 0 $ 
and if $ U $ is a minimizer for 
$ E_{\min} ( \mathfrak{p}) $ with speed $ c $, then $ U ( c\, \cdot ) $ is a minimizer 
for $ I_{\min} ( \kappa ) $ with $ \kappa = \frac{1}{2} \int_{\R^2} | \nabla U |^2 \, d x $ (this last quantity 
is scale-invariant in 2d) and $ I_{\min} $ is differentiable at this $ \kappa $, with 
$ I_{\min} ' ( \kappa ) = - 1 / c^2 $. Since $ Q_c $ is a minimizer for 
$ E_{\min} ( P_2 (Q_c) ) $, if we prove that $ c \mapsto \frac{1}{2} \int_{\R^2} | \nabla Q_c |^2 \, d x $ 
is a decreasing $ C^1 $-diffeomorphism from $ ]0, c_0 ] $, for some small $ c_0 $, 
onto $ [ \kappa_0 , +\ii [ $, with $ \kappa_0 \assign \frac{1}{2} \int_{\R^2} | \nabla Q_{c_0} |^2 \, d x $, 
then we shall conclude that $ I_{\min} $ is of class $ C^1 $ on $ [ \kappa_0 , +\ii [ $, 
and that (by the arguments in \cite{Chi_Mar}), the only minimizer for 
$ \kappa = \frac{1}{2} \int_{\R^2} | \nabla Q_c |^2 \, d x $ (for some suitable $ c \in ]0, c_0 ] $) is 
$ Q_c ( c \, \cdot ) $ up to the natural translation and phase invariances and, in 
addition, $ I_{\min} ' ( \kappa ) = - 1 / c^2 $. In order to prove that statement, it suffices to 
use the Pohozaev identity \eqref{Pohoz} and deduce
\[
  \frac{1}{2} \int_{\R^2} | \nabla Q_c |^2 \, d x 
  = E ( Q_c) - \frac{1}{4} \int_{\R^2} ( 1 - \lvert Q_c \rvert^2 )^2 \, dx 
  = E ( Q_c) - \frac{ c P_2 (Q_c) }{2 } .
\]
Therefore, by using the Hamilton like relation $ \frac{d}{ dc} E ( Q_{ c } ) = c \frac{d}{ dc} P_2 ( Q_{ c } ) $ 
and then the asymptotics of $c \mapsto P_2(Q_c)$ obtained in \cite{Chi_Pac_2}, we arrive at
\[
  \frac{d}{2 dc} 
  \int_{\R^2} | \nabla Q_c |^2 \, d x 
  = \frac{d}{dc} ( E ( Q_c) ) - \frac{ c }{2} \frac{d}{dc} P_2 (Q_c) - \frac{1}{2} P_2 (Q_c)
  = \frac{ c }{2} \frac{d}{dc} P_2 (Q_c) - \frac{1}{2} P_2 (Q_c) 
  \sim - \frac{2\pi}{c} < 0 
\]
and this concludes.

\bigbreak 
The paper is organized as follows. In section \ref{pf:propC1}, we give the proof 
of the uniqueness result given in Proposition \ref{locC1}. Section \ref{sec:quasimin} 
is devoted to the vortex analysis of travelling waves with energy 
$ E_{\min} (\mathfrak{p}) + \mathcal{O} (1) $ and even in $x_2 $, in order to show that 
they satisfy the hypotheses of Proposition \ref{locC1}. 
Subsection \ref{subsec:nonsym} contains a few remarks on the nonsymmetrical case. 
Finally, in subsection \ref{subsec:decaylightly}, we provide some decay estimates 
slightly away from the vortices. For the Ginzburg-Landau (stationary) model, 
such estimates have been first given in \cite{Miro_explicit} for minimizing solutions 
and later generalized in \cite{Com_Mir} to non-minimizing solutions. 
They improve some estimates in \cite{Chi_Pac_1} 
and are not specific to the way we construct the solutions.

\bigbreak
\noindent {\bf Acknowledgement:} E.P. is supported by Tamkeen under the NYU Abu Dhabi 
Research Institute grant CG002. We would like to thank the referee for her/his careful reading 
of the manuscript and for suggestions that have helped and clarified the presentation.

%%%%%%%%%%%%%%%%%%%%%%%%%%%%%%%%%%%%%%%%%%%%
%%%%%%%%%%%%%%%%%%%%%%%%%%%%%%%%%%%%%%%%%%%%
\section{Proof of the local uniqueness result (Proposition \ref{locC1})}
\label{pf:propC1}

This section is devoted to the proof of Proposition \ref{locC1} and Corollary
\ref{CP3cor}. The proof of Proposition \ref{locC1} uses arguments from the
proof of Theorem 1.14 of {\cite{Chi_Pac_2}}, another local uniqueness result for this
problem, but in different spaces. We explain here the core ideas of the proof.

Let us explain schematically the proof of Proposition \ref{locC1}. We first pick $c' $, $X $, $ \gamma' $ 
in such a way that $ Q = Q_c'(.-X) \ex^{i \gamma}$ has the same vortices as $u$. This is possible because $c \rightarrow d_c$, 
the position of the vortices, is smooth. We then decompose $u=Q \ex^{\psi} $, where $\psi $ is the error term. 
This can not be done near the zeros of $Q$, but we focus here on the domain far from the vortices.

The equation satisfied by $\psi$ is then $ (\tmop{TW}_c) (u) = 0 =(\tmop{TW}_c)(Q)+\it{L}(\psi)+\it{NL}(\psi)$, 
where we regroup the linear terms in $\it{L}$ and the nonlinear terms in $\it{NL}$, and $(\tmop{TW}_c)(Q) \neq 0$ 
because $c \neq c'$. We then take the scalar product of this equation with $\psi$, and we get 
$ 0 = \langle (\tmop{TW}_c)(Q), \psi \rangle + B_Q(\psi) + \langle \it{NL}(\psi) , \psi \rangle $. 
Now, the coercivity of $B_Q$ has been studied in \cite{Chi_Pac_2}. It holds (for even functions in $x_1$) 
up to three orthogonality conditions, that can be satisfied by changing slightly the modulation parameters 
$c',X,\gamma$. We deduce that $B_Q(\psi) \geqslant K \| \psi \|^2_1$ for some norm $ \| \cdot \|_1$. 

There are two main difficulties at this point. First, since the hypothesis on $u$ in Proposition \ref{locC1} 
are weak, we simply have $||\psi||_1 < + \infty$, but not the fact that it is small. Therefore, an estimate 
of the form $ \lvert \langle \it{NL}(\psi) , \psi \rangle \rvert \leqslant K \| \psi \|^3_1$ would not be enough to conclude. 
Secondly, the norm $ \| \cdot \|_1$ is rather weak, and in fact $\langle \it{NL}(\psi) , \psi \rangle$ can not be controlled 
by powers of $ \| \psi \|_1$. 

Concerning the term $\langle (TW_c)(Q), \psi \rangle $, we may show that we always have $|c-c'| \leqslant o(1) \| \psi \|_1$, 
thus $ \lvert \langle (TW_c)(Q), \psi \rangle \rvert \leqslant o(1) \| \psi \|^2_1 $. Therefore, we are led to 
$ (K/2) \| \psi \|^2_1 \ls \langle (\tmop{TW}_c)(Q), \psi \rangle + B_Q(\psi) = - \langle \it{NL}(\psi) , \psi \rangle $. 
Then, even if $\| \psi \|_1$ is not small, by the hypothesis of Proposition \ref{locC1}, $\psi$ will be small in other 
(non equivalent) norms. Let us write one of them $ \| \cdot \|_2$. Our goal is then to show an estimate of the 
form $ \lvert \langle \it{NL}(\psi) , \psi \rangle \rvert \leqslant K \| \psi \|_2 \| \psi \|^2_1$, which would conclude. 
This is possible, except for one nonlinear term, which contains two derivatives. We then perform some 
integrations by parts on it. When both derivatives fall on the same term, we get a term containing $\Delta \psi$, 
which also appears in the equation $0= (\tmop{TW}_c)(Q)+\it{L}(\psi)+\it{NL}(\psi)$ (in $\it{L}(\psi)$). 
We thus replace it using the equation, which leads to another term containing two derivatives (from $\it{NL}(\psi)$), 
and other terms that can be successfully estimated. After $n$ such integration by parts, we have an estimate 
of the form $ \lvert \langle \it{NL}(\psi) , \psi \rangle \rvert \leqslant K\| \psi \|_2 \| \psi \|^2_1 + \| \psi \|_3 \| \psi \|_2^n \| \psi \|_1^2 $, 
where $\|\cdot \|_3$ is another (semi-)norm in which $\psi$ is not necessarily small. Now, taking $n$ large enough 
(depending on $\psi$), since $ \| \psi \|_2 \ll 1$, we get $| \langle \it{NL}(\psi) , \psi \rangle | \leqslant o(1) \| \psi \|^2_1$, 
concluding the proof. 

The problem is a lot simpler near the vortices. There, we write $u = Q +\phi$ and the coercivity norm is equivalent 
to the $H^1$ norm, and the hypothesis of Proposition \ref{locC1} gives us $ \| \phi \|_{L^\infty} = o(1)$. 
The estimate of the nonlinear terms then becomes trivial.

As stated in the introduction, the symmetry condition is necessary to have a
coercivity result where the coercivity constant is uniform, see Corollary
\ref{CP3coerc} below. This is the only place where the symmetry is used in a crucial way.

%%%%%%%%%%%%%%%%%%%%%%%%%%%%%%%%%%%%%%%%%%%%
\subsection{Some properties of the branch of travelling waves from Theorem \ref{th1}}

We recall here properties on the branch $c \mapsto Q_c$ from Theorem
\ref{th1}, coming mainly from {\cite{Chi_Pac_2}} and {\cite{Chi_Pac_1}}. In this section, we will use the notation
\[ \langle f,g \rangle \assign \int_{\mathbb{R}^2} \mathfrak{Re} (f \bar{g}) . \]

\subsubsection{Properties of vortices}\label{CP3vor}

We start with some estimates on vortices, that compose the travelling wave
(see Theorem \ref{th1}).

\begin{lemma}[{\cite{ChenElliottQi}} and {\cite{HerHer_94}}]
  \label{lemme3new}A vortex centered around $0$, $V_1 (x) = \rho_1 (r) \ex^{i
  \theta}$, verifies $V_1 (0) = 0$, $E (V_1) = + \infty$ and there exist
  constants $K, \kappa > 0$ such that
  \[ \forall r > 0, \ 0 < \rho_1 (r) < 1; 
  \quad \quad 
  \rho_1 (r) \sim_{r \rightarrow 0} \kappa r ;
  \quad \quad 
  \rho'_1 (r) \sim_{r \rightarrow 0} \kappa ; \]
  \[ \forall r > 0, \ \rho_1' (r) > 0 ; 
  \quad \quad 
  \rho_1' (r) = O_{r \rightarrow \infty} \left( \frac{1}{r^3} \right); 
  \quad \quad | \rho_1'' (r) | + | \rho_1''' (r) | \leqslant K ; 
  \]
  \[ 1 - | V_1 (x) | = \frac{1}{2 r^2} + O_{r \rightarrow \infty} \left(
     \frac{1}{r^3} \right) ; \]
  \[ | \nabla V_1 | \leqslant \frac{K}{1 + r}; 
  \quad \quad | \nabla^2 V_1 | \leqslant
     \frac{K}{(1 + r)^2} \]
  and
  \[ \nabla V_1 (x) = i V_1 (x) \frac{x^{\bot}}{r^2} + O_{r \rightarrow
     \infty} \left( \frac{1}{r^3} \right), \]
  where $x^{\perp} \assign (- x_2, x_1)$, $x = r \ex^{i \theta} \in
  \mathbb{R}^2$. Furthermore, similar properties hold for $V_{- 1}$, since
  \[ V_{- 1} (x) = \overline{V_1 (x)_{}} . \]
\end{lemma}

%%%%%
\subsubsection{Toolbox}

We list in this section some results useful for the analysis of travelling waves for not necessarily small speeds.

\begin{theorem}[Uniform $ L^\ii $ bound - \cite{Farina}]
\label{Linf}
Assume that $ U \in L^3_{\rm loc} (\R^d) $ solves
\[
 \Delta U + i c \p_2 U + U ( 1 - \lvert U \rvert^2 ) = 0 .
\]
Then,
\[
 \| U \|_{L^\ii(\R^d)} \ls 1+ \frac{c^2}{4} .
\]
\end{theorem}

\begin{corollary} 
\label{Gradinf}
There exists $ K > 0 $ such that for any $ c \in [ - \sqrt{2} , + \sqrt{2} ] $ and 
any $ U \in L^3_{\rm loc} (\R^d) $ satisfying $ (\textrm{TW}_c)(U)=0  $, we have
\begin{equation}
 \label{basik}
 \| \nabla U \|_{L^\ii(\R^d)} + \| \nabla^2 U \|_{L^\ii(\R^d)} \ls K .
\end{equation}
\end{corollary}

The following Pohozaev identity (see \cite{Bet_Gra_Saut} for instance) will be useful in our analysis. If 
$ c \in \R $ and $ U \in \BE $ satisfies $ (\textrm{TW}_c) $, then
\begin{equation}
\label{Pohoz}
 \frac{1}{2} \int_{\R^2} (1 - \lvert U \rvert^2 )^2 \, dx 
 = c P_2 (U) .
\end{equation}

We shall also make use of the algebraic decay of the travelling waves conjectured in 
\cite{Jon_Rob} and shown in \cite{Gra}.

\begin{theorem}[Algebraic decay of the travelling waves - \cite{Gra}]
\label{decay}
Let $ c \in [0, \sqrt{2} [ $. Assume that $ U \in \BE $ is a solution of $ (\textrm{TW}_c)(U)=0  $. Up to a phase shift, 
we may assume $ U( x) \to 1 $ for $|x| \to +\ii $. Then, there exists 
$ M $, depending on $ U$ and $c$ such that, for $ x \in \R^2 $,
\[
  \lvert U (x) - 1 \rvert \ls \frac{M}{ 1+ |x|} , 
 \quad \quad 
 \lvert \nabla U (x) \rvert \ls \frac{M}{( 1+ |x|)^2} , 
 \quad \quad 
 \big\lvert \lvert U (x) \rvert - 1 \big\rvert \ls \frac{M}{( 1+ |x|)^2} .
\] 
\end{theorem}

\subsubsection{Symmetries of the travelling waves from Theorem
\ref{th1}}\label{CP3sym}

We recall from {\cite{Chi_Pac_1}} that the travelling wave $Q_c$ constructed in
Theorem \ref{th1} satisfies for all $x = (x_1, x_2) \in \mathbb{R}^2$,
\[ Q_c (x_1, x_2) = Q_c (- x_1, x_2) = \overline{Q_c (x_1, - x_2)} . \]
This implies that for all $x = (x_1, x_2) \in \mathbb{R}^2$,
\[ \partial_c Q_c (x_1, x_2) = \partial_c Q_c (- x_1, x_2) =
   \overline{\partial_c Q_c (x_1, - x_2)}, \]
\[ \partial_{x_1} Q_c (x_1, x_2) = - \partial_{x_1} Q_c (- x_1, x_2) =
   \overline{\partial_{x_1} Q_c (x_1, - x_2)}, \]
\[ \partial_{x_2} Q_c (x_1, x_2) = \partial_{x_2} Q_c (- x_1, x_2) = -
   \overline{\partial_{x_2} Q_c (x_1, - x_2)} \]
and
\[ \partial_{c^{\bot}} Q_c (x_1, x_2) = - \partial_{c^{\bot}} Q_c (- x_1, x_2)
   = - \overline{\partial_{c^{\bot}} Q_c (x_1, - x_2)}, \]
where $\partial_{c^{\bot}} Q_c \assign x^{\bot} . \nabla Q_c$, see subsection
2.2 of {\cite{Chi_Pac_2}}. Remark that these quantities all have different symmetries.

\subsubsection{A coercivity result}

From Proposition 1.2 of {\cite{Chi_Pac_2}}, we recall that $Q_c$ defined in Theorem
\ref{th1} has two zeros, at $\pm \tilde{d}_c \vec{e}_1$, with
\begin{equation}
  d_c - \tilde{d}_c = o_{c \rightarrow 0} (1) . \label{31gne}
\end{equation}
We define (as in {\cite{Chi_Pac_2}}) the symmetric expended energy space by
\[ H_{Q_c}^{\exp, s} \assign \left\{ \varphi \in H^1_{\tmop{loc}}
   (\mathbb{R}^2, \mathbb{C}), \| \varphi \|_{H_{Q_c}^{\exp}} < + \infty,
   \forall (x_1, x_2) \in \mathbb{R}^2, \varphi (- x_1, x_2) = \varphi (x_1,
   x_2) \right\}, \]
where, with $\varphi = Q_c \psi$, $\tilde{r} = \tilde{r}_{\tilde{d}_c} = \min
(\tilde{r}_1, \tilde{r}_{- 1})$, $\tilde{r}_{\pm 1}$ being the distances to
the zeros of $Q_c$ (we use $\tilde{r}$ instead of $\tilde{r}_{\tilde{d}_c}$ to
simplify the notations here), we define
\[ \| \varphi \|^2_{H_{Q_c}^{\exp}} \assign \| \varphi \|^2_{H^1 (\{ \tilde{r}
   \leqslant 10 \})} + \int_{\{ \tilde{r} \geqslant 5 \}} | \nabla \psi |^2
   +\mathfrak{Re}^2 (\psi) + \frac{| \psi |^2}{\tilde{r}^2 \ln^2 \tilde{r} } . \]
By using \eqref{basik}, we deduce, for any $R > 0$, $\| \varphi \|_{H^1 (\{ \tilde{r} \leqslant R \})}
\leqslant K (R) \| \varphi \|_{H_{Q_c}^{\exp}}$. The linearized operator
around $Q_c$ is
\[ L_{Q_c} (\varphi) \assign - \Delta \varphi - i c \partial_{x_2} \varphi -
   (1 - | Q_c |^2) \varphi + 2\mathfrak{Re} (\overline{Q_c}
   \varphi) Q_c . \]
We take a smooth cutoff function $\tilde{\eta}$ such that $\tilde{\eta} (x) =
0$ on $B (\pm \widetilde{d_c} \overrightarrow{e_1}, 2 R)$, $\tilde{\eta} (x) =
1$ on $\mathbb{R}^2 \backslash B (\pm \widetilde{d_c} \overrightarrow{e_1}, 2
R + 1)$, where $\pm \widetilde{d_c} \overrightarrow{e_1}$ are the zeros of
$Q_c$ and $R > 0$ will be defined later on (it will be a universal constant,
independent of any parameters of the problem). We define the quadratic form
(as in {\cite{Chi_Pac_2}})
\begin{eqnarray}
  B^{\exp}_{Q_c} (\varphi) & \assign & \int_{\mathbb{R}^2} (1 - \tilde{\eta})
  (| \nabla \varphi |^2 -\mathfrak{Re} (i c \partial_{x_2} \varphi
  \bar{\varphi}) - (1 - | Q_c |^2) | \varphi |^2 + 2\mathfrak{Re}^2
  (\overline{Q_c} \varphi)) \nonumber\\
  & - & \int_{\mathbb{R}^2} \nabla \tilde{\eta} . (\mathfrak{Re}
  (\nabla Q_c \overline{Q_c}) | \psi |^2 - 2\mathfrak{I}\mathfrak{m} (\nabla
  Q_c \overline{Q_c}) \mathfrak{Re} (\psi) \mathfrak{I}\mathfrak{m}
  (\psi)) \nonumber\\
  & + & \int_{\mathbb{R}^2} c \partial_{x_2} \tilde{\eta}
  \mathfrak{Re} (\psi) \mathfrak{I}\mathfrak{m} (\psi) | Q_c |^2
  \nonumber\\
  & + & \int_{\mathbb{R}^2} \tilde{\eta} (| \nabla \psi |^2 | Q_c |^2 +
  2\mathfrak{Re}^2 (\psi) | Q_c |^4) \nonumber\\
  & + & \int_{\mathbb{R}^2} \tilde{\eta} (4\mathfrak{I}\mathfrak{m} (\nabla
  Q_c \overline{Q_c}) \mathfrak{I}\mathfrak{m} (\nabla \psi)
  \mathfrak{Re} (\psi) + 2 c | Q_c |^2 \mathfrak{I}\mathfrak{m}
  (\partial_{x_2} \psi) \mathfrak{Re} (\psi)) .  \label{CP2Btilda}
\end{eqnarray}

We recall from {\cite{Chi_Pac_2}} (or by integration by parts) that for $\varphi \in
C^{\infty}_c (\mathbb{R}^2, \mathbb{C})$, we have $B^{\exp}_{Q_c} (\varphi) = \langle
L_{Q_c} (\varphi), \varphi \rangle$, and that $B^{\exp}_{Q_c} (\varphi)$ is
well defined for $\varphi \in H_{Q_c}^{\exp, s}$. This last point is the
reason why we write the quadratic form as (\ref{CP2Btilda}), which is equal,
up to some integration by parts, to the more natural definition
\[ \int_{\mathbb{R}^2} | \nabla \varphi |^2 - (1 - | Q_c |^2) | \varphi |^2 +
   2\mathfrak{Re}^2 (\overline{Q_c} \varphi)
   -\mathfrak{Re} (i c \partial_{x_2} \varphi \bar{\varphi}), \]
but this integral is not well defined for $\varphi \in H_{Q_c}^{\exp, s}$. See
{\cite{Chi_Pac_2}} for more details on this point. We now quote a coercivity
result from {\cite{Chi_Pac_2}}.

\begin{theorem}[{\cite{Chi_Pac_2}}, Theorem 1.13]
  \label{CP3th22}There exists $R, K, c_0 > 0$ such that, for $0 < c \leqslant
  c_0$, $Q_c$ defined in Theorem \ref{th1}, if a function $\varphi \in
  H_{Q_c}^{\exp, s}$ satisfies the three orthogonality conditions:
  \[ \mathfrak{Re} \int_{B (\tilde{d}_c \vec{e}_1, R) \cup B (-
     \tilde{d}_c \vec{e}_1, R)} \partial_c Q_c \bar{\varphi}
     =\mathfrak{Re} \int_{B (\tilde{d}_c \vec{e}_1, R) \cup B (-
     \tilde{d}_c \vec{e}_1, R)} \partial_{x_2} Q_c \bar{\varphi} = 0, \]
  \[ \mathfrak{Re} \int_{B (\tilde{d}_c \vec{e}_1, R) \cup B (-
     \tilde{d}_c \vec{e}_1, R)} i Q_c \bar{\varphi} = 0, \]
  then
  \[ \frac{1}{K} \| \varphi \|^2_{H_{Q_c}^{\exp}} \geqslant B_{Q_c}^{\exp}
     (\varphi) \geqslant K \| \varphi \|_{H_{Q_c}^{\exp}}^2 . \]
\end{theorem}

We will use a slight variation of this result, given in the next corollary.

\begin{corollary}
  \label{CP3coerc}There exists $R, K, c_0 > 0$ such that, for $0 < c \leqslant
  c_0$, $Q_c$ defined in Theorem \ref{th1}, if a function $\varphi \in
  H_{Q_c}^{\exp, s}$ satisfies the three orthogonality conditions:
  \[ \mathfrak{Re} \int_{B (d_c \vec{e}_1, R) \cup B (- d_c
     \vec{e}_1, R)} \partial_d (V_1 (. - d \vec{e}_1) V_{- 1} (. + d
     \vec{e}_1))_{| d = d_c \nobracket} \bar{\varphi}
     =\mathfrak{Re} \int_{B (d_c \vec{e}_1, R) \cup B (- d_c
     \vec{e}_1, R)} \partial_{x_2} Q_c \bar{\varphi} = 0, \]
  \[ \mathfrak{Re} \int_{B (d_c \vec{e}_1, R) \cup B (- d_c
     \vec{e}_1, R)} i Q_c \bar{\varphi} = 0, \]
  then
  \[ \frac{1}{K} \| \varphi \|^2_{H_{Q_c}^{\exp}} \geqslant B_{Q_c}^{\exp}
     (\varphi) \geqslant K \| \varphi \|_{H_{Q_c}^{\exp}}^2 . \]
\end{corollary}

Remark, with Theorem \ref{th1} (for $p = + \infty$), that $- \frac{1}{c^2}
\partial_d (V_1 (. - d \vec{e}_1) V_{- 1} (. + d \vec{e}_1))_{| d = d_c
\nobracket}$ is the first order of $\partial_c Q_c$ when $c \rightarrow 0$ in
$L^{\infty} (\mathbb{R}^2, \mathbb{C})$, and that (with Lemma \ref{lemme3new})
they both have the same symmetries. We need to change the quantity
$\mathfrak{Re} \int_{B (\tilde{d}_c \vec{e}_1, R) \cup B (-
\tilde{d}_c \vec{e}_1, R)} \partial_c Q_c \bar{\varphi}$ in the orthogonality
conditions because we will differentiate it with respect to $c$, and
\[ c \mapsto \partial_d (V_1 (. - d \vec{e}_1) V_{- 1} (. + d \vec{e}_1))_{| d
   = d_c \nobracket} = - \partial_{x_1} V_1 (. - d_c \vec{e}_1) V_{- 1} (. +
   d_c \vec{e}_1) + \partial_{x_1} V_{- 1} (. + d_c \vec{e}_1) V_1 (. - d_c
   \vec{e}_1) \]
is a $C^1$ function ($c \mapsto d_c \in C^1 (] 0, c_0 [, \mathbb{R})$ for $c_0
> 0$ a small constant, see subsection 4.6 of {\cite{Chi_Pac_1}}), but it is not
clear that $c \mapsto \partial_c Q_c$ can be differentiated with respect to
$c$. Precise estimates on $\partial_d (V_1 (. - d \vec{e}_1) V_{- 1} (. + d
\vec{e}_1))_{| d = d_c \nobracket}$ can be found in Lemma 2.6 of {\cite{Chi_Pac_1}}.
Furthermore, we changed, in the area of the integrals, $\tilde{d}_c$ by $d_c$
(they are close when $c \rightarrow 0$, see (\ref{31gne})).

\begin{proof}
  {\tmem{Step 1: changing the integrand but not the integration domain.}}
  
  \bigbreak
  
  Take a function $\varphi \in H_{Q_c}^{\exp, s}$ satisfying
  \[ \mathfrak{Re} \int_{B (\tilde{d}_c \vec{e}_1, R) \cup B (-
     \tilde{d}_c \vec{e}_1, R)} \partial_d (V_1 (. - d \vec{e}_1) V_{- 1} (. +
     d \vec{e}_1))_{| d = d_c \nobracket} \bar{\varphi}
     =\mathfrak{Re} \int_{B (\tilde{d}_c \vec{e}_1, R) \cup B (-
     \tilde{d}_c \vec{e}_1, R)} \partial_{x_2} Q_c \bar{\varphi} = 0, \]
  \[ \mathfrak{Re} \int_{B (\tilde{d}_c \vec{e}_1, R) \cup B (-
     \tilde{d}_c \vec{e}_1, R)} i Q_c \bar{\varphi} = 0. \]
  Let us show that it satisfies $\frac{1}{K} \| \varphi \|^2_{H_{Q_c}^{\exp}}
  \geqslant B_{Q_c}^{\exp} (\varphi) \geqslant K \| \varphi
  \|_{H_{Q_c}^{\exp}}$. For $\mu \in \mathbb{R}$, we define
  \[ \varphi^{\ast} = \varphi + c^2 \mu \partial_c Q_c . \]
  We have that $\partial_c Q_c \in H_{Q_c}^{\exp, s}$. We want to choose $\mu
  \in \mathbb{R}$ such that $\varphi^{\ast}$ satisfied the hypothesis of Theorem \ref{CP3th22}.
  By the symmetries of subsection \ref{CP3sym} and the hypotheses on
  $\varphi$, we have that
  \[ \mathfrak{Re} \int_{B (\tilde{d}_c \vec{e}_1, R) \cup B (-
     \tilde{d}_c \vec{e}_1, R)} i Q_c \overline{\varphi^{\ast}}
     =\mathfrak{Re} \int_{B (\tilde{d}_c \vec{e}_1, R) \cup B (-
     \tilde{d}_c \vec{e}_1, R)} \partial_{x_2} Q_c \overline{\varphi^{\ast}} =
     0, \]
  and we compute, using $\mathfrak{Re} \int_{B (\tilde{d}_c
  \vec{e}_1, R) \cup B (- \tilde{d}_c \vec{e}_1, R)} \partial_d (V_1 (. - d
  \vec{e}_1) V_{- 1} (. + d \vec{e}_1))_{| d = d_c \nobracket} \bar{\varphi} =
  0$, that
  \begin{eqnarray*}
    &  & \mathfrak{Re} \int_{B (\tilde{d}_c \vec{e}_1, R) \cup B
    (- \tilde{d}_c \vec{e}_1, R)} c^2 \partial_c Q_c
    \overline{\varphi^{\ast}}\\
    & = & \mathfrak{Re} \int_{B (\tilde{d}_c \vec{e}_1, R) \cup B
    (- \tilde{d}_c \vec{e}_1, R)} c^2 \partial_c Q_c \bar{\varphi} + \mu
    \mathfrak{Re} \int_{B (\tilde{d}_c \vec{e}_1, R) \cup B (-
    \tilde{d}_c \vec{e}_1, R)} c^4 | \partial_c Q_c |^2\\
    & = & \mathfrak{Re} \int_{B (\tilde{d}_c \vec{e}_1, R) \cup B
    (- \tilde{d}_c \vec{e}_1, R)} (c^2 \partial_c Q_c - \partial_d (V_1 (. - d
    \vec{e}_1) V_{- 1} (. + d \vec{e}_1))_{| d = d_c \nobracket})
    \bar{\varphi}\\
    & + & \mu \mathfrak{Re} \int_{B (\tilde{d}_c \vec{e}_1, R)
    \cup B (- \tilde{d}_c \vec{e}_1, R)} c^4 | \partial_c Q_c |^2 .
  \end{eqnarray*}
  By Theorem \ref{th1} (for $p = + \infty$) and Lemma 2.6 of {\cite{Chi_Pac_1}}, we
  have
  \[ \| c^2 \partial_c Q_c - \partial_d (V_1 (. - d \vec{e}_1) V_{- 1} (. + d
     \vec{e}_1))_{| d = d_c \nobracket} \|_{L^{\infty} (\mathbb{R}^2)} = o_{c
     \rightarrow 0} (1), \]
  and also that there exists a universal constant $K > 0$ (we recall that $R >
  0$ is a universal constant) such that
  \[ \frac{1}{K} \leqslant \mathfrak{Re} \int_{B (\tilde{d}_c
     \vec{e}_1, R) \cup B (- \tilde{d}_c \vec{e}_1, R)} c^4 | \partial_c Q_c
     |^2 \leqslant K. \]
  Now, taking
  \[ \mu = \frac{-\mathfrak{Re} \int_{B (\tilde{d}_c \vec{e}_1, R)
     \cup B (- \tilde{d}_c \vec{e}_1, R)} (c^2 \partial_c Q_c - \partial_d
     (V_1 (. - d \vec{e}_1) V_{- 1} (. + d \vec{e}_1))_{| d = d_c \nobracket})
     \bar{\varphi}}{\mathfrak{Re} \int_{B (\tilde{d}_c \vec{e}_1,
     R) \cup B (- \tilde{d}_c \vec{e}_1, R)} c^4 | \partial_c Q_c |^2}, \]
  we have
  \[ \mathfrak{Re} \int_{B (\tilde{d}_c \vec{e}_1, R) \cup B (-
     \tilde{d}_c \vec{e}_1, R)} c^2 \partial_c Q_c \overline{\varphi^{\ast}} =
     0, \]
  with
  \[ | \mu | \leqslant o_{c \rightarrow 0} (1) \| \varphi \|_{L^2 (B
     (\tilde{d}_c \vec{e}_1, R) \cup B (- \tilde{d}_c \vec{e}_1, R))}
     \leqslant o_{c \rightarrow 0} (1) \| \varphi \|_{H_{Q_c}^{\exp}} . \]

  Since $\partial_c Q_c \in H_{Q_c}^{\exp, s}$ by Lemma 2.8 of {\cite{Chi_Pac_2}},
  we deduce that $\varphi^{\ast}$ satisfies all the hypotheses of Theorem
  \ref{CP3th22}, therefore
  \[ \frac{1}{K} \| \varphi^{\ast} \|^2_{H_{Q_c}^{\exp}} \geqslant
     B_{Q_c}^{\exp} (\varphi^{\ast}) \geqslant K \| \varphi^{\ast}
     \|^2_{H_{Q_c}^{\exp}} . \]
  Now, from Lemma 6.3 of {\cite{Chi_Pac_2}}, we have $\frac{1}{K} \leqslant \| c^2
  \partial_c Q_c \|_{H^{\exp}_{Q_c}} \leqslant K$ for a universal constant $K
  > 0$. With $| \mu | \leqslant o_{c \rightarrow 0} (1) \| \varphi
  \|_{H_{Q_c}^{\exp}}$, we deduce that, taking $c > 0$ small enough,
  \[ \frac{1}{K} \| \varphi \|^2_{H_{Q_c}^{\exp}} \geqslant B_{Q_c}^{\exp}
     (\varphi^{\ast}) \geqslant K \| \varphi \|^2_{H_{Q_c}^{\exp}} \]
  for some universal constant $K > 0$. Now, we decompose (using Lemmas 6.2 and
  6.3 of {\cite{Chi_Pac_2}})
  \begin{eqnarray*}
    B_{Q_c}^{\exp} (\varphi^{\ast}) & = & B_{Q_c}^{\exp} (\varphi + c^2 \mu
    \partial_c Q_c)\\
    & = & B_{Q_c}^{\exp} (\varphi) + 2 c^2 \mu \langle L_{Q_c} (\partial_c
    Q_c), \varphi \rangle + c^4 \mu^2 B_{Q_c}^{\exp} (\partial_c Q_c),
  \end{eqnarray*}
  and by Lemmas 2.8, 5.4 and 6.1 of {\cite{Chi_Pac_2}},
  \[ | \langle L_{Q_c} (\partial_c Q_c), \varphi \rangle | = | \langle i
     \partial_{x_2} Q_c, \varphi \rangle | \leqslant K \ln \left( \frac{1}{c}
     \right) \| \varphi \|_{H_{Q_c}^{\exp}}, \]
  hence
  \[ | 2 c^2 \mu \langle L_{Q_c} (\partial_c Q_c), \varphi \rangle | \leqslant
     K c^2 \ln \left( \frac{1}{c} \right) | \mu | \| \varphi
     \|_{H_{Q_c}^{\exp}} \leqslant o_{c \rightarrow 0} (1) \| \varphi
     \|_{H_{Q_c}^{\exp}}^2 . \]
  By Proposition 1.2 of {\cite{Chi_Pac_2}}, $B_{Q_c}^{\exp} (\partial_c Q_c) =
  \frac{2 \pi + o_{c \rightarrow 0} (1)}{c^2}$, thus
  \[ | c^4 \mu^2 B_{Q_c}^{\exp} (\partial_c Q_c) | \leqslant o_{c \rightarrow
     0} (1) \| \varphi \|_{H_{Q_c}^{\exp}}^2, \]
  which concludes the proof of $\frac{1}{K} \| \varphi \|^2_{H_{Q_c}^{\exp}}
  \geqslant B_{Q_c}^{\exp} (\varphi) \geqslant K \| \varphi
  \|_{H_{Q_c}^{\exp}}^2 $ by taking $c > 0$ small enough.
  
  \
  
  {\tmem{Step 2: changing the integration domain.}}
  
  \
  
  To change the conditions
  \[ \mathfrak{Re} \int_{B (\tilde{d}_c \vec{e}_1, R) \cup B (-
     \tilde{d}_c \vec{e}_1, R)} \partial_d (V_1 (. - d \vec{e}_1) V_{- 1} (. +
     d \vec{e}_1))_{| d = d_c \nobracket} \bar{\varphi}
     =\mathfrak{Re} \int_{B (\tilde{d}_c \vec{e}_1, R) \cup B (-
     \tilde{d}_c \vec{e}_1, R)} \partial_{x_2} Q_c \bar{\varphi} = 0, \]
  \[ \mathfrak{Re} \int_{B (\tilde{d}_c \vec{e}_1, R) \cup B (-
     \tilde{d}_c \vec{e}_1, R)} i Q_c \bar{\varphi} = 0 \]
  to
  \[ \mathfrak{Re} \int_{B (d_c \vec{e}_1, R) \cup B (- d_c
     \vec{e}_1, R)} \partial_d (V_1 (. - d \vec{e}_1) V_{- 1} (. + d
     \vec{e}_1))_{| d = d_c \nobracket} \bar{\varphi}
     =\mathfrak{Re} \int_{B (d_c \vec{e}_1, R) \cup B (- d_c
     \vec{e}_1, R)} \partial_{x_2} Q_c \bar{\varphi} = 0, \]
  \[ \mathfrak{Re} \int_{B (d_c \vec{e}_1, R) \cup B (- d_c
     \vec{e}_1, R)} i Q_c \bar{\varphi} = 0, \]
  we use similar arguments, using $| d_c - \tilde{d}_c | = o_{c \rightarrow 0} (1)$ by (\ref{31gne}). We
  check for instance that
  \[ \left| \mathfrak{Re} \int_{B (\tilde{d}_c \vec{e}_1, R) \cup B
     (- \tilde{d}_c \vec{e}_1, R)} \partial_{x_2} Q_c \bar{\varphi}
     -\mathfrak{Re} \int_{B (d_c \vec{e}_1, R) \cup B (- d_c
     \vec{e}_1, R)} \partial_{x_2} Q_c \bar{\varphi} \right| \leqslant K (R) |
     d_c - \tilde{d}_c |  \| \varphi \|_{H_{Q_c}^{\exp}}, \]
  and $| d_c - \tilde{d}_c | = o_{c \rightarrow 0} (1)$. 
  
Notice that the integration domain remains symmetric with respect to the $x_2$-axis.
\end{proof}

\subsection{Proof of Proposition \ref{locC1}}

In this subsection, we take $\nu \in ] 0, 1 [$ a small but universal constant, 
that will be fixed at the end of the proof. We take \[ \lambda_* = \text{max}( 3 R + 1, \frac{1}{\nu^2}) \] in the statement of 
Proposition \ref{locC1} (where $R > 0$ is defined in
 Corollary \ref{CP3coerc}). Then, for any $\lambda \geqslant \lambda_*$, we take
 \[ \varepsilon (\lambda) = \min \Big(\nu , \frac{1}{10 \lambda^2 + 100} \Big)  \] 
 in the statement of Proposition \ref{locC1}. The condition $ \varepsilon (\lambda) \leqslant \frac{1}{10 \lambda^2 + 100} $ is required only to make sure that the two balls $ B ( d \overrightarrow{e_1} , 2 \lambda ) $ and $ B ( -d \overrightarrow{e_1} , 2 \lambda ) $ are disjoint and at a distance at least $1$ from one another. This will be used only in the proof of Lemma \ref{CP3L27}.

We take $u$ a function satisfying the hypotheses of
Proposition \ref{locC1} for these values of $\lambda_* ,\lambda$ and $\varepsilon(\lambda)$. In
the rest of the subsection, $K, K' > 0$ denote universal constants,
independent of any parameters of the problem (in particular, $\lambda, \lambda_*, \varepsilon(\lambda)$ and $\nu$).

\subsubsection{Modulation on the parameters of the branch}

From Theorem 1.1 and the end of section 4.6 of {\cite{Chi_Pac_1}}, we have that $Q_c
= V_1 (. - d_c \vec{e}_1) V_{- 1} (. + d_c \vec{e}_1) + \Gamma_c$, with $d_c =
\frac{1 + o_{c \rightarrow 0} (1)}{c}$, $ \| \Gamma_c \|_{L^\ii} \to 0 $, and
\[ c \mapsto d_c \in C^1 (] 0, c_0 [, \mathbb{R}) , \]
with $ \p_c d_c \sim - 1/c^2 $ for $c \to 0 $ (see section 4.6 of \cite{Chi_Pac_1}). 
In particular, $ c \mapsto d_c $ is a smooth decreasing diffeomorphism from 
$ ]0,c_0 ] $ onto $ [ d_0, +\infty [ $, and thus, given $d > \frac{1}{\nu} > d_0 $ (for $\nu$ small enough), there exists a 
unique $c' > 0$ such that $ d_{c'} = d $. In addition, $ c' \sim_{d\to \ii} 1/d \leqslant K \nu $.  Furthermore,
\begin{eqnarray*}
  u (x) - Q_{c'} (x) & = & V_1 (x - d \vec{e}_1) V_{- 1} (x + d \vec{e}_1) +
  \Gamma (x) - V_1 (x - d_{c'} \vec{e}_1) V_{- 1} (x + d_{c'} \vec{e}_1) -
  \Gamma_{c'} (x)\\
  & = & \Gamma (x) - \Gamma_{c'} (x) .
\end{eqnarray*}
From the hypotheses on $\Gamma$, and the fact that $\| \Gamma_{c'}
\|_{L^{\infty} (\mathbb{R}^2)} \leqslant 2 \nu$ (since $c' \leqslant \frac{2}{d}\leqslant 2 \nu$,
we deduce that (we denote $\tilde{r} =
\tilde{r}_d = \tilde{r}_{d_{c'}}$ to simplify the notations)
\[ 
 \| u - Q_{c'} \|_{L^{\infty} (\{ \tilde{r} \leqslant 2 \lambda \})} 
 \leqslant K \nu. \]
Since $\frac{1 + o_{c' \rightarrow 0} (1)}{c'} = d_{c'} = d$ by Theorem
\ref{th1}, and $| d c - 1 | \leqslant \nu $, we have
\begin{equation}
  d | c - c' | \leqslant K \nu .
  \label{CP3ccp}
\end{equation}
We now claim that, for a universal constant $K > 0$,
\begin{equation}
  \| u - Q_{c'} \|_{C^1 (\{ \tilde{r} \leqslant \lambda \})} 
  \leqslant K \nu
  \label{22CP3} .
\end{equation}
That is, $u$ is close to $Q_{c'}$ near the vortices (in the region $\{
\tilde{r} \leqslant \lambda \}$) in the $C^1$ norm and not only in $L^\infty$. In order to show this, we use the 
elliptic equation satisfied by $u - Q_{c'}$, that is
\[ \Delta (u - Q_{c'}) = - i c \partial_{x_2} (u - Q_{c'}) - (u - Q_{c'}) (1 -
   | u |^2) + (| u |^2 - | Q_{c'} |^2) Q_{c'}. \]
Let us fix $x \in  \{ \tilde{r} \leqslant \lambda \}$. 
We have $ \| u - Q_{c'} \|_{L^{\infty} (\{ \tilde{r} \leqslant 2 \lambda \})}\leqslant K' \nu $
by hypothesis, thus the right-hand side of the equation is small in $ H^{-1} ( B(x, 4) ) $. 
By a standard $H^1-H^{-1}$ estimate, we deduce
\[
 \| u - Q_{c'} \|_{H^1 ( B( x, 3 ))} \leqslant K' \nu . \]
Then, the right-hand side is small in $ L^2 $, and standard $L^2 $ elliptic regularity yields first
\[
 \| u - Q_{c'} \|_{H^2 ( B( x, 2 ))} \leqslant K' \nu 
\]
and then
\[
 \| u - Q_{c'} \|_{H^3 ( B( x, 1 ) )}\leqslant K' \nu ,
\]
and we conclude by Sobolev imbedding.

Outside of this domain, $u$ and $Q_{c'}$ are close only
in modulus. Indeed, by equation (2.6) of {\cite{Chi_Pac_2}} (for $\sigma = 1 / 2$)
and the hypotheses on $u$, we have for a universal constant $K > 0 $ that on $\{ \tilde{r} \geqslant \lambda \}$,
\[ | | u | - | Q_{c'} | | \leqslant | | u | - 1 | + | | Q_{c'} | - 1 |
   \leqslant \nu + \frac{K}{\lambda^{3 / 2}} 
   \leqslant K' \nu    . \]
Now, we modulate on the parameters of the family of travelling waves to get
the orthogonality conditions of Corollary \ref{CP3coerc}. For $c'' > 0$ close
enough to $c'$ and $X, \gamma \in \mathbb{R}$, we define
\begin{equation}
  Q \assign Q_{c''} (. - X \vec{e}_2) \ex^{i \gamma} . \label{Qdef}
\end{equation}
\begin{lemma}
  \label{CP3L26} 
 There exists $K  > 0, \nu_0  > 0$ universal constants such that, for $u$ satisfying the hypotheses of Proposition \ref{locC1} for values of $ \lambda_*,  \lambda, \varepsilon( \lambda), \nu$ described above, if $\nu \leqslant \nu_0$, then there exists
 $c'' > 0$, $X, \gamma \in \mathbb{R}$ such that, for $R > 0$ defined
in Corollary \ref{CP3coerc}, and $\vec{d}_{\pm} \assign \pm d_{c''}
\vec{e}_1 + X \vec{e}_2$;
\begin{eqnarray*}
  &  & \mathfrak{Re} \int_{B (\vec{d}_+, R) \cup B (\vec{d}_-,
   R)} \partial_d (V_1 (. - d \vec{e}_1 - X \vec{e}_2) V_{- 1} (. + d
    \vec{e}_1 - X \vec{e}_2) \ex^{i \gamma})_{| d = d_{c''} \nobracket}
    \overline{(u - Q)}\\
    & = & \mathfrak{Re} \int_{B (\vec{d}_+, R) \cup B (\vec{d}_-,
    R)} \partial_{x_2} Q \overline{(u - Q)}\\
    & = & \mathfrak{Re} \int_{B (\vec{d}_+, R) \cup B (\vec{d}_-,
    R)} i Q \overline{(u - Q)}\\
    & = & 0.
  \end{eqnarray*}
  Furthermore,
  \[ \frac{| c'' - c' |}{c^{\prime 2}} + | X | + | \gamma | \leqslant K \nu .
    \]
\end{lemma}

\begin{proof} 
To simplify the notations, in this proof, we define
\[ \partial_d V \assign \partial_d (V_1 (. - d \vec{e}_1 - X \vec{e}_2) V_{-
   1} (. + d \vec{e}_1 + X \vec{e}_2) \ex^{i \gamma})_{| d = d_{c''} \nobracket}
   . \]
We will keep the notation $\tilde{r}$ for the minimum of the distance to the
zeros of $Q$.

  First, from equation (7.5) of {\cite{Chi_Pac_2}}, there exists a universal
  constant $K > 0$ such that, for $c'' < c_0$, $c' / 2 \leqslant c'' \leqslant
  2 c'$,
  \begin{equation}
    \| Q - Q_{c'} \|_{L^{\infty} (\mathbb{R}^2)} \leqslant K \left( | X | +
    \frac{| c'' - c' |}{c^{\prime 2}} + | \gamma | \right) . \label{CP3eq21}
  \end{equation}
  Now, we follow closely the proof of Lemma 7.6 of {\cite{Chi_Pac_2}}, which is done
  in Appendix C.3 there. We define
  \[ G \left(\begin{array}{c}
       X\\
       c''\\
       \gamma
     \end{array}\right) \assign \left(\begin{array}{c}
       \mathfrak{Re} \int_{B (\vec{d}_+, R) \cup B (\vec{d}_-, R)}
       \partial_{x_2} Q \overline{(u - Q)}\\
       \mathfrak{Re} \int_{B (\vec{d}_+, R) \cup B (\vec{d}_-, R)}
       \partial_d V \overline{(u - Q)}\\
       \mathfrak{Re} \int_{B (\vec{d}_+, R) \cup B (\vec{d}_-, R)}
       i Q \overline{(u - Q)}
     \end{array}\right) . \]
  Remark that $Q, \partial_d V$ and $\vec{d}_{\pm}$ all depend on $X$ and
  $c''$, and $Q$ depends also on $\gamma$. From equation (\ref{22CP3}) and the 
  fact that $\lambda \geqslant \lambda_* > 2 R$, we have $\| u - Q_{c'} \|_{L^{\infty} (\{ \tilde{r}
  \leqslant R \})} \leqslant K \nu$, 
  and from Theorem \ref{th1} with $p = +
  \infty$ as well as Lemma 2.6 of {\cite{Chi_Pac_1}},
  \begin{equation}
    \| \partial_{x_2} Q_{c'} \|_{L^{\infty} (\mathbb{R}^2)} + \| \partial_d V
    \|_{L^{\infty} (\mathbb{R}^2)} + \| i Q_{c'} \|_{L^{\infty}
    (\mathbb{R}^2)} \leqslant K \label{CP3A}
  \end{equation}
  for some universal constant $K > 0$. Therefore, since $Q = Q_{c'}$ for $X =
  \gamma = 0$, $c'' = c'$, we obtain
  \[ \left| G \left(\begin{array}{c}
       0\\
       c'\\
       0
     \end{array}\right) \right| 
     \leqslant K \| u - Q_{c'} \|_{L^{\infty} (\{ \tilde{r} \leqslant \lambda \})} 
     \leqslant K \nu
     . \]
  We want to show that $G$ is invertible in a vicinity of
  $\left(\begin{array}{c}
    0\\
    c'\\
    0
  \end{array}\right)$. With equations (\ref{22CP3}) and (\ref{CP3eq21}), we
  check that (we recall that $\tilde{r} = \min (| x - \vec{d}_+ |, | x -
  \vec{d}_- |)$)
  \begin{eqnarray*}
    \| u - Q \|_{L^{\infty} (\{ \tilde{r} \leqslant 2 R \})} & \leqslant & \|
    u - Q_{c'} \|_{L^{\infty} (\{ \tilde{r} \leqslant 2 R \})} + \| Q - Q_{c'}
    \|_{L^{\infty} (\mathbb{R}^2)}\\
    & \leqslant &   K \nu + K \left( | X | + \frac{| c'' - c' |}{c^{\prime 2}}
    + | \gamma | \right),
  \end{eqnarray*}
  and as in Lemma 7.1 of {\cite{Chi_Pac_2}}, this implies
  \begin{equation}
    \| u - Q \|_{C^1 (\{ \tilde{r} \leqslant R \})} \leqslant K \nu 
    + K \left(
    | X | + \frac{| c'' - c' |}{c^{\prime 2}} + | \gamma | \right) .
    \label{CP3eq22}
  \end{equation}
  Now, we compute
  \begin{eqnarray*}
    &  & \left| \partial_X \left( \mathfrak{Re} \int_{B
    (\vec{d}_+, R) \cup B (\vec{d}_-, R)} \partial_{x_2} Q \overline{(u - Q)}
    \right) - \int_{B (\vec{d}_+, R) \cup B (\vec{d}_-, R)} | \partial_{x_2} Q
    |^2 \right|\\
    & \leqslant & \int_{\partial B (\vec{d}_+, R) \cup \partial B (\vec{d}_-,
    R)} | \partial_{x_2} Q \overline{(u - Q)} | + \int_{B (\vec{d}_+, R) \cup
    B (\vec{d}_-, R)} | \partial_{x_2}^2 Q \overline{(u - Q)} |,
  \end{eqnarray*}
  therefore, with (\ref{basik}) and (\ref{CP3eq22}), we check that
  \begin{eqnarray*}
    &  & \int_{\partial B (\vec{d}_+, R) \cup \partial B (\vec{d}_-, R)} |
    \partial_{x_2} Q \overline{(u - Q)} | + \int_{B (\vec{d}_+, R) \cup B
    (\vec{d}_-, R)} | \partial_{x_2}^2 Q \overline{(u - Q)} |\\
    & \leqslant & K \nu 
    + K \left( | X | + \frac{| c'' - c' |}{c^{\prime 2}}
    + | \gamma | \right),
  \end{eqnarray*}
  hence
  \begin{eqnarray*}
    &  & \left| \partial_X \left( \mathfrak{Re} \int_{B
    (\vec{d}_+, R) \cup B (\vec{d}_-, R)} \partial_{x_2} Q \overline{(u - Q)}
    \right) - \int_{B (\vec{d}_+, R) \cup B (\vec{d}_-, R)} | \partial_{x_2} Q
    |^2 \right|\\
    & \leqslant & K \nu 
    + K \left( | X | + \frac{| c'' - c' |}{c^{\prime 2}}
    + | \gamma | \right) .
  \end{eqnarray*}
  With similar computations, using Lemma 2.6 of {\cite{Chi_Pac_1}}, equations
  (\ref{basik}) and (\ref{CP3eq22}), we infer that
  \[ \left| \partial_X G - \left(\begin{array}{c}
       \int_{B (\vec{d}_+, R) \cup B (\vec{d}_-, R)} | \partial_{x_2} Q |^2\\
       \mathfrak{Re} \int_{B (\vec{d}_+, R) \cup B (\vec{d}_-, R)}
       \partial_d V \overline{\partial_{x_2} Q}\\
       \mathfrak{Re} \int_{B (\vec{d}_+, R) \cup B (\vec{d}_-, R)}
       i Q \overline{\partial_{x_2} Q}
     \end{array}\right) \right| \leqslant K \nu 
     + K \left( | X | + \frac{| c''
     - c' |}{c^{\prime 2}} + | \gamma | \right) . \]
  By the symmetries of $Q (. + X \vec{e}_2) \ex^{- i \gamma}$ and $\partial_d V
  (. + X \vec{e}_2) \ex^{- i \gamma}$, we have that
  \[ \mathfrak{Re} \int_{B (\vec{d}_+, R) \cup B (\vec{d}_-, R)}
     \partial_d V \overline{\partial_{x_2} Q} = 0, \]
  and from Theorem \ref{th1} (with $p = + \infty$), with the symmetries of
  $Q_c$ and $V_1$ (see subsections \ref{CP3vor} and \ref{CP3sym}), we have
  \[ \left| \mathfrak{Re} \int_{B (\vec{d}_+, R) \cup B (\vec{d}_-,
     R)} i Q \overline{\partial_{x_2} Q} - 2\mathfrak{Re} \int_{B
     (0, R)} i V_1 \overline{\partial_{x_2} V_1} \right| 
     \leqslant K \left( | X | + \frac{| c'' - c' |}{c^{\prime 2}} \right) . \]
  By decomposition in harmonics and Lemma \ref{lemme3new}, we check easily
  that $\mathfrak{Re} \int_{B (0, R)} i V_1
  \overline{\partial_{x_2} V_1} = 0$, thus

  \[ \left| \partial_X G - \left(\begin{array}{c}
        \int_{B (\vec{d}_+, R) \cup B (\vec{d}_-, R)} | \partial_{x_2} Q
       |^2\\
       0\\
       0
     \end{array}\right) \right| \leqslant K \nu 
     + K \left( | X | + \frac{| c''
     - c' |}{c^{\prime 2}} + | \gamma | \right) . \]
  Similarly, we check that (using $\partial_c (d_c) = \frac{- 1 + o_{c
  \rightarrow 0} (1)}{c^2}$ from section 4.6 of {\cite{Chi_Pac_1}}, and Lemma 2.6 of
  {\cite{Chi_Pac_1}})
  \[ \left| c^{\prime 2} \partial_{c''} G - \left(\begin{array}{c}
       0\\
     \int_{B (\vec{d}_+, R) \cup B (\vec{d}_-, R)}
       | \partial_d V |^2\\
       0
     \end{array}\right) \right| \leqslant K \nu 
     + K \left( | X | + \frac{| c''
     - c' |}{c^{\prime 2}} + | \gamma | \right) \]
  (we use here the fact that $c \mapsto \partial_d V$ and $c \mapsto
  \vec{d}_{\pm}$ are differentiable) and
  \[ \left| \partial_{\gamma} G - \left(\begin{array}{c}
       0\\
       0\\
       - \int_{B (\vec{d}_+, R) \cup B (\vec{d}_-, R)} | Q |^2
     \end{array}\right) \right| \leqslant K \nu 
     + K \left( | X | + \frac{| c''
     - c' |}{c^{\prime 2}} + | \gamma | \right) . \]
  From (\ref{basik}) and Theorem \ref{th1} (for $p = + \infty$) as well as
  Lemma 2.6 of {\cite{Chi_Pac_1}}, there exists a universal constant $K > 0$ such
  that
  \[ \frac{1}{K} \leqslant \int_{B (\vec{d}_+, R) \cup B (\vec{d}_-, R)} |
     \partial_{x_2} Q |^2 \leqslant K, \]
  \[ \frac{1}{K} \leqslant \int_{B (\vec{d}_+, R)
     \cup B (\vec{d}_-, R)} | \partial_d V |^2 \leqslant K \]
  and
  \[ \frac{1}{K} \leqslant \int_{B (\vec{d}_+, R) \cup B (\vec{d}_-, R)} | Q
     |^2 \leqslant K , \]
 provided $ |X| + c'' $ is small enough. We deduce that there exists $K_1,K_2, \nu_0 > 0 $ 
 such that, for $ 0 < \nu \ls \nu_0 $ and $u$ satisfying the hypotheses of Proposition 
  \ref{locC1} with the parameters $ \lambda $, $ \nu $, 
  $d G$ is invertible in the ball $ \{ ( X, c'', \gamma ) \in \mathbb{R}^3 \, {\rm s.t. } 
   | X | + \frac{| c'' - c' |}{c^{\prime 2}} + | \gamma | \ls K_1 \nu \} $, 
  and that there exists $X $, $c''$, $ \gamma \in \mathbb{R}$ such that
  \[ G \left(\begin{array}{c}
       X\\
       c''\\
       \gamma
     \end{array}\right) = 0, \]
  with
  \[ \frac{| c'' - c' |}{c^{\prime 2}} + | X | + | \gamma | \leqslant K_2 \nu .
  \]
\end{proof}

\subsubsection{Construction and properties of the perturbation term}

We define $\eta$ a smooth cutoff function with $\eta (x) = 0$ for $x \in B
(\vec{d}_{\pm}, 2 R)$ and $\eta (x) = 1$ for $x \in \mathbb{R}^2 \backslash B
(\pm \vec{d}_{\pm}, 2 R + 1)$ even in $x_1$. We infer the following result, where 
the space $H_Q^{\exp,s}$ is simply defined by
\[ H_Q^{\exp, s} \assign \left\{ \varphi \in H^1_{\tmop{loc}} (\mathbb{R}^2,
   \mathbb{C}), \| \varphi \|_{H_{Q_{}}^{\exp}} < + \infty, \forall (x_1, x_2)
   \in \mathbb{R}^2, \varphi (- x_1, x_2) = \varphi (x_1, x_2) \right\}, \]
with, for $\tilde{r}$ the minimum of the distances to the zeros of $Q$,
$\varphi = Q \psi$,
\[ \| \varphi \|^2_{H_Q^{\exp}} \assign \| \varphi \|^2_{H^1 (\{ \tilde{r}
   \leqslant 10 \})} + \int_{\{ \tilde{r} \geqslant 5 \}} | \nabla \psi |^2
   +\mathfrak{Re}^2 (\psi) + \frac{| \psi |^2}{\tilde{r}^2 \ln^2 \tilde{r} }, \]
and $B_Q^{\exp}$ has the same definition than $B_{Q_c}^{\exp}$, replacing
$\tilde{\eta}$ by $\eta$ and $Q_c$ by $Q$.

\begin{lemma}
  \label{CP3L27} 
  There exists $K_1,K_2  > 0, \nu_0 > \nu_1  > 0$ universal constants such that, for $u$ 
  satisfying the hypotheses of Proposition \ref{locC1} for values of $ \lambda_*,  \lambda, \varepsilon( \lambda), \nu$ 
  described above, if $\nu \leqslant \nu_1$, then there exists a function $\varphi = Q \psi \in H_Q^{\exp, s} \cap C^1
  (\mathbb{R}^2, \mathbb{C})$ such that, for $Q$ defined in (\ref{Qdef}) with
  the values of $c'', X, \gamma \in \mathbb{R}$ from Lemma \ref{CP3L26},
  \[ u - Q = (1 - \eta) \varphi + \eta Q (\ex^{\psi} - 1) . \]
  Furthermore,
  \[ B_Q^{\exp} (\varphi) \geqslant K_1 \| \varphi \|_{H_Q^{\exp}}^2 \]
  and
  \[ \| \varphi \|_{C^1 (\{ \tilde{r} \leqslant \lambda \})} +
  \| \mathfrak{Re} (\psi) \|_{L^{\infty} (\{ \tilde{r} \geqslant \lambda \})} \leqslant K_2 \nu
  . \]
\end{lemma}

The goal of this lemma is to decompose the error $u - Q$ in a particular form.
In the area $\{ \eta = 1 \}$, that is far from the zeros of $Q$, the error is
written in an exponential form: $u = Q \ex^{\psi}$. This form was already used
in {\cite{Chi_Pac_2}}, {\cite{Chi_Pac_1}}, and is useful to have a particular form on the
cubic error terms. Furthermore, we fix the parameters of $Q$ such that
$\varphi$ satisfies the orthogonality conditions of Corollary \ref{CP3coerc},
yielding the coercivity.

Remark that we have 
no smallness on $\mathfrak{I}\mathfrak{m} (\psi)$
in $\{ \tilde{r} \geqslant \lambda \}$, where $\varphi = Q \psi$. We will
simply be able to show that it is bounded (see equation (\ref{CP3dumdum}) below),
with no a priori bound on it. This lack of smallness is one of the main
difficulties in the proof of Proposition \ref{locC1}. Analogously, we show 
that $\varphi \in H_Q^{\exp, s}$, but we have no good control on $\| \varphi \|_{H_Q^{\exp}}$: 
this quantity might be a priori very large at this point.

\begin{proof}
  This proof follows some ideas of the proofs of Lemmas 7.2 and 7.3 of
  {\cite{Chi_Pac_2}}. First, in the area $\{ \tilde{r} \leqslant \lambda \}$, the
  proof is identical to that of Lemma 7.2 of {\cite{Chi_Pac_2}} for the existence of
  $\varphi = Q \psi \in C^1 (\{ \tilde{r} \leqslant \lambda \}, \mathbb{C})$
  such that $u - Q = (1 - \eta) \varphi + \eta Q (\ex^{\psi} - 1)$ in $\{
  \tilde{r} \leqslant \lambda \}$, with $\| \varphi \|_{C^1 (\{ \tilde{r}
  \leqslant \lambda \})} \leqslant K \nu $ (this is a consequence of the estimate 
  $\| u - Q \|_{C^1 (\{ \tilde{r} \leqslant \lambda \})} \leqslant K \nu $, obtained using Lemma \ref{CP3L26}). 
  The main idea is that $ u - Q $ is
  small there (in $ C^1 (\{ \tilde{r} \leqslant \lambda \}, \mathbb{C})$), and
  the equation on $\varphi$ is a perturbation of the identity for functions
  $\varphi$ that are small in $C^1 (\{ \tilde{r} \leqslant \lambda \}, \mathbb{C})$.  
  In particular, since $u $ and $Q$ are symmetric with respect to the 
  $x_2$-axis, $ \varphi $ and $ \psi $ are also symmetric with respect to the 
  $x_2$-axis.
  
  We then focus our attention in the area $\{ \tilde{r} \geqslant \lambda \}$, 
  where $ \eta \equiv 1$, so that the problem reduces to the equation
  \[
  u = Q \ex^\psi .
  \]
 By Theorem \ref{th1} and the hypotheses of Proposition  \ref{locC1}, there exists $\nu_1 > 0$ such that, if $\nu \leqslant \nu_1$, then, as a consequence of
 \[ \varepsilon(\lambda) \leqslant \text{min}(\nu_1, \frac{1}{10 \lambda^2 + 100}), \]
the domain $ \{ \tilde{r} \geqslant \lambda \} $ 
  consists in the complement of the two disjointed disks $ B ( \vec{d}_\pm , \lambda) $, with
  \[
  | Q| \gs 1/2 , \quad \quad |u| \gs 1/2  \quad \quad {\rm in \, } \{ \tilde{r} \gs \lambda \} 
  \]
and
\[
 {\rm deg} ( Q , \p B ( \vec{d}_\pm , \lambda ) ) = {\rm deg} ( u , \p B ( \vec{d}_\pm , \lambda ) ) = \pm 1 ,
\]
so that $ u / Q $ is smooth in $ \{ \tilde{r} \gs \lambda \} = \R^2 \setminus ( B ( \vec{d}_+ , \lambda ) \cup B ( \vec{d}_- , \lambda ) ) $, 
does not vanish and has zero degree on the two circles $\p B ( \vec{d}_\pm , \lambda ) $. 
It then follows from standard lifting theorems (even though $ \{ \tilde{r} \gs \lambda \} $ is 
not simply connected), that there exists $ \psi^\dag \in C^1 ( \{ \tilde{r} \geqslant \lambda \} ) $ 
such that $ \ex^{\psi^\dag} = u / Q $, as wished. We then notice that $ u $ and $Q $ are 
symmetric with respect to the $ x_2 $-axis, thus $ x \mapsto \psi^\dag ( -x_1 , x_2 ) $ is 
also a lifting of $u / Q $ in the connected set $ \{ \tilde{r} \geqslant \lambda \} $, which 
implies that there exists $ q \in \Z $ such that 
$ \psi^\dag ( -x_1 , x_2 ) = \psi^\dag ( x_1 , x_2 ) + 2 i q \pi $ in $ \{ \tilde{r} \geqslant \lambda \} $. 
Letting $ x_1 = 0 $, we obtain $ q= 0 $: $ \psi^\dag $ is also symmetric with 
respect to the $ x_2 $-axis. 

Recalling that $ \psi \assign \varphi / Q $ in 
the set $ \{ \lambda \ls \tilde{r} \ls 2 \lambda \} $ (where $ Q$ does not vanish), we see that, 
since $\eta \equiv 1 $ there, the equation $u - Q = (1 - \eta) \varphi + \eta Q (\ex^{\psi} - 1)$ 
becomes $ u = Q \ex^\psi $. We then infer that there exists $ m \in \Z $ such that 
$ \psi = \psi^\dag + 2im \pi $ in the connected annulus 
$ B ( \vec{d}_+ , 2 \lambda ) \setminus B ( \vec{d}_+ , \lambda ) $. By symmetry in $x_1$, 
this is also true in the annulus 
$ B ( \vec{d}_-, 2\lambda ) \setminus B ( \vec{d}_- , \lambda ) $. It then suffices to extend 
$ \psi $ by the formula $ \psi = \psi^\dag + 2im \pi $ in $ \{ \tilde{r} \geqslant \lambda \} $ 
to obtain the formula $ u - Q = (1 - \eta) \varphi + \eta Q (\ex^{\psi} - 1) $. In the 
region $ \{ \tilde{r} \geqslant \lambda \} $, the relation $ u = Q \ex^{\psi } $ yields
 \[ \ex^{\mathfrak{Re} (\psi)} = \left| \frac{u}{Q} \right| , \]
 thus, decomposing $ | \frac{u}{Q} | = 1 + |u|-1 + \frac{(|u|-1)-(|Q|-1)}{|Q|}$, since there exists a universal constant $K'>0 $ such that in this region, $ \left| |u|-1 + \frac{(|u|-1)-(|Q|-1)}{|Q|} \right| \leqslant K' \nu $, we deduce that, for $\nu \leqslant \nu_1$ with $\nu_1$ small enough,
  \[ \| \mathfrak{Re} (\psi) \|_{L^{\infty} (\{ \tilde{r} \geqslant \lambda \})} 
 \leqslant K \nu 
  . \]

Since $u$ is a travelling wave and $E (u) < + \infty$, 
  $u$ converges to a constant at infinity (uniformly in all directions) by
  {\cite{Gra}}. Therefore, $\frac{u}{Q}$ converges to a constant at
  infinity, and the function $\psi$ converges to a constant, and thus it is
  bounded near infinity, that is
  \begin{equation}
    \| \psi \|_{L^{\infty} (\{ \tilde{r} \geqslant \lambda \})} < + \infty .
    \label{CP3dumdum}
  \end{equation}

Now, we want to show that $\varphi \in H_Q^{\exp, s}$. We already know that 
$\varphi$ satisfies the symmetry 
  \[ \forall (x_1, x_2) \in \mathbb{R}^2, \varphi (- x_1, x_2) = \varphi (x_1,
     x_2) . \]
  Furthermore, to check that $\| \varphi \|_{H_Q^{\exp}} < + \infty$, since
  $\varphi \in C^1 (\mathbb{R}^2, \mathbb{C})$, we only have to check the
  integrability in $\{ \tilde{r} \geqslant \lambda \}$, where $\ex^{\psi} =
  \frac{u}{Q}$. We check that there, with (\ref{CP3dumdum}),
  \[ \int_{\{ \tilde{r} \geqslant \lambda \}} \frac{| \psi |^2}{\tilde{r}^2
     \ln^2 (\tilde{r})} < + \infty . \]
  Now, using Theorem 11 of {\cite{Gra}} (we recall that $E (u) < +
  \infty, E (Q) < + \infty$),
  \[ | \ex^{\mathfrak{Re} (\psi)} - 1 | = \frac{| | u | - | Q | |}{|
     Q |} \leqslant 2 (| | u | - 1 | + | | Q | - 1 |) \leqslant \frac{K (u, c , Q, c'')}{(1 + r)^2}, \]
  where $K (u, c , Q, c'') > 0$ is a constant depending on $u$, $c$, $c'' $ and $Q$, 
  hence $| \mathfrak{Re} (\psi) | \leqslant \frac{K (u, c , Q, c'')}{(1 + r)^2}$ and
  \[ \int_{\{ \tilde{r} \geqslant \lambda \}} \mathfrak{Re}^2
     (\psi) \leqslant \int_{\{ \tilde{r} \geqslant \lambda \}} \frac{K (u, c , Q, c'' )}{(1 + r)^4} < + \infty . \]
  We finally compute
  \[ \nabla \psi =  \frac{\nabla u}{u} - \frac{\nabla Q}{Q}, \]
  and with Theorem 11 of {\cite{Gra}}, in $\{ \tilde{r} \geqslant
  \lambda \}$, we deduce that
  \[ (1 + r)^2 | \nabla \psi | \leqslant (1 + r)^2 \left| \frac{\nabla u}{u}
     \right| + (1 + r)^2 \left| \frac{\nabla Q}{Q} \right| \leqslant K (u, c , Q, c''),
  \]
  therefore
  \[ \int_{\{ \tilde{r} \geqslant \lambda \}} | \nabla \psi |^2 < + \infty .
  \]
  This concludes the proof that $\varphi = Q \psi \in H_Q^{\exp, s}$. The fact
  that $B_Q^{\exp} (\varphi) \geqslant K \| \varphi \|_{H_Q^{\exp}}^2$ is a
  consequence of Corollary \ref{CP3coerc} and Lemma \ref{CP3L26}, using in
  particular that $B_Q^{\exp} (\varphi) = B_{Q_{c''}}^{\exp} (\varphi (. + X
  \vec{e}_2) \ex^{- i \gamma})$ and $\| \varphi \|_{H_Q^{\exp}} = \| \varphi (.
  + X \vec{e}_2) \ex^{- i \gamma} \|_{H_{Q_{c''}}^{\exp}}$.
\end{proof}

We now compute the equation satisfied by $\varphi$.
By Lemma \ref{CP3L27}, in $ \{ 0 < \eta < 1 \} = \{ 2R < \tilde{r} < 2R+1 \} $, 
we have $ | \mathfrak{Re} (\psi) | = |\mathfrak{Re} (\varphi/ Q )| \leqslant  K \nu $ 
uniformly, thus $ | \ex^{ \mathfrak{Re} (\psi) } -1 | \leqslant K \nu $ 
uniformly in this region and then $ \lvert (1 - \eta) + \eta \ex^{\psi} \rvert \geqslant 1/2 $ 
for $ \nu \leqslant \nu_1 $, possibly diminishing 
$ \nu_1 $ of Lemma \ref{CP3L27}.

\begin{lemma}
  \label{CP3LL26} For $u$ satisfying the hypotheses of Proposition \ref{locC1} for values of $ \lambda_*,  \lambda, \varepsilon( \lambda), \nu$ described above, if $\nu \leqslant \nu_1$ (where $\nu_1$ is defined in Lemma \ref{CP3L27}) 
  , then the function $\varphi = Q \psi$ defined in Lemma \ref{CP3L27}
  satisfies the equation
  \[ L_Q (\varphi) - i (c - c'') \vec{e}_2 .H (\psi) + \tmop{NL}_{\tmop{loc}}
     (\psi) + F (\psi) = 0, \]
  with $L_Q$ the linearized operator around $Q$: $L_Q (\varphi) = - \Delta
  \varphi - i c'' \partial_{x_2} \varphi - (1 - | Q |^2) \varphi +
  2\mathfrak{Re} (\bar{Q} \varphi) Q$,
  \[ S (\psi) \assign \ex^{2\mathfrak{Re} (\psi)} - 1 -
     2\mathfrak{Re} (\psi), \]
  \[ F (\psi) \assign Q \eta (- \nabla \psi . \nabla \psi + | Q |^2 S (\psi)),
  \]
  \[ H (\psi) \assign \nabla Q + \frac{\nabla (Q \psi) (1 - \eta) + Q \nabla \psi
     \eta \ex^{\psi}}{(1 - \eta) + \eta \ex^{\psi}} \]
  and $\tmop{NL}_{\tmop{loc}} (\psi)$ is a sum of terms at least quadratic in
  $\psi$, localized in the area where $\eta \neq 1$. Furthermore,
  \[ | \langle \tmop{NL}_{\tmop{loc}} (\psi), Q \psi \rangle |
     \leqslant 
      K \| \tmop{NL}_{\tmop{loc}} (\psi) \|_{L^2( \{ \eta < 1 \} ) } 
       \| \varphi \|_{L^\ii( \{ \eta < 1 \} ) }    \leqslant K \nu \| \varphi \|_{H^1 (\{ \eta \neq 1 \})}^2 . \]
\end{lemma}

Notice that $F (\psi)$ (the notation $ X.Y $ for complex vector fields stands for $ X_1 Y_1 + X_2 Y_2 $) 
contains all the nonlinear terms far from the zeros of
$^{} Q$, and its structure relies on the fact that the error is written in an exponential form far from the vortices. 
Close to the zeros of $Q$, this particular form does not hold,
but it will not be necessary, since there the error $\varphi$ is small in the
$C^1$ norm whereas, at infinity, it is small only in a weaker norm.

\begin{proof}
  The proof is identical to the proof of Lemma 7.5 of {\cite{Chi_Pac_2}}, and it is
  in the particular case where all the speeds are along $\vec{e}_2$. The proof
  consists simply in decomposing the equation
  \[ 0 = (\tmop{TW}_c) (u) = \tmop{TW}_c (Q + (1 - \eta) \varphi + \eta Q
     (\ex^{\psi} - 1)) \]
  in the different terms.
  
  The last estimate uses Lemma \ref{CP3L27} and Lemma \ref{CP3L26}.
\end{proof}

This result shows in particular that $\psi \in C^2 (\{ \eta \neq 0 \},
\mathbb{C})$, and we can check with it, as in Lemma 7.3 of {\cite{Chi_Pac_2}}, that
$\| \Delta \psi (1 + r)^2 \|_{L^{\infty} (\{ \tilde{r} \geqslant \lambda \})}
\leqslant K (u, Q , c , c'' )$.

\

We now infer a critical estimate on the differences of the speeds of the
problem, namely $c$ (the speed of $u$) and $c''$ (the speed of $Q$). The
method for the estimate has been used in {\cite{Chi_Pac_2}} (we take the scalar
product of the equation of Lemma \ref{CP3LL26} with $\partial_c Q$), but since
we have worse estimates on the error term, we need to be more careful ($\| \varphi \|_{H_Q^{\exp}} $ is not a priori small at this point).

\begin{lemma}
  \label{CP3neurones} 
  There exists universal constants $K > 0, \nu_1 \geqslant  \nu_2 > 0$ (where $\nu_1$ is defined in Lemma \ref{CP3L27}), such that, for $u$ satisfying the hypotheses of Proposition \ref{locC1} for values of $ \lambda_*,  \lambda, \varepsilon( \lambda), \nu$ described above, if $\nu \leqslant \nu_2$, then, with $\varphi = Q \psi$ defined in Lemma \ref{CP3L27}, we have
  \[ | c'' - c | \leqslant K \sqrt{c''} \| \varphi \|_{H_Q^{\exp}} . \]
\end{lemma}

\begin{proof}
  First, from equation (\ref{CP3ccp}) and Lemma \ref{CP3L26}, taking
  $\nu > 0$ small enough, we have
  \begin{equation}
    | c'' - c | \leqslant | c'' - c' | + | c' - c | \leqslant K c'' .
    \label{CP3cece}
  \end{equation}
  We will show the following estimate:
  \begin{equation}
    | c'' - c | \leqslant K \left( c^{\prime\prime 2} \ln \left( \frac{1}{c''}
    \right) \| \varphi \|_{H_Q^{\exp}} + \| \varphi \|_{H_Q^{\exp}}^2 \right)
    + K | c'' - c | \| \varphi \|_{H_Q^{\exp}} . \label{CP3olala}
  \end{equation}
  This is related to equation (7.13) of {\cite{Chi_Pac_2}} (its proof is in step 1
  in subsection 7.3.1 of {\cite{Chi_Pac_2}}). With both estimates, we can conclude
  the proof of this lemma. Indeed, either $\| \varphi \|_{H_Q^{\exp}}
  \geqslant \sqrt{c''}$, and in that case
  \[ | c'' - c | \leqslant K c'' \leqslant K \sqrt{c''} \| \varphi
     \|_{H_Q^{\exp}}, \]
  or $\| \varphi \|_{H_Q^{\exp}} \leqslant \sqrt{c''}$, and then with
  (\ref{CP3olala}),
  \begin{eqnarray*}
    | c'' - c | & \leqslant & K \left( c^{\prime\prime 2} \ln \left( \frac{1}{c''}
    \right) \| \varphi \|_{H_Q^{\exp}} + \| \varphi \|_{H_Q^{\exp}}^2 \right)
    + K | c'' - c | \| \varphi \|_{H_Q^{\exp}}\\
    & \leqslant & K \sqrt{c''} \| \varphi \|_{H_Q^{\exp}} + C_2 \sqrt{c''} |
    c'' - c |,
  \end{eqnarray*}
  therefore, for $c'' > 0$ small enough such that $C_2 \sqrt{c''} < 1 / 2$
  (which is implied by taking $\nu > 0$ small enough, independently of $ \lambda $), 
  we have $| c'' - c | \leqslant K \sqrt{c''} \| \varphi \|_{H_Q^{\exp}}$.
  
  \
  
  We now focus on the proof of (\ref{CP3olala}). We take the scalar product
  of the equation
  \[ L_Q (\varphi) - i (c - c'') \vec{e}_2 .H (\psi) + \tmop{NL}_{\tmop{loc}}
     (\psi) + F (\psi) = 0 \]
  with $c^{\prime\prime 2} \partial_{c''} Q$. We estimate, as in subsection
  7.3.1 of {\cite{Chi_Pac_2}}, that
  \[ | \langle L_Q (\varphi), c^{\prime\prime 2} \partial_{c''} Q \rangle | =
     c^{\prime\prime 2} | \langle \varphi, L_Q (\partial_{c''} Q) \rangle | =
     c^{\prime\prime 2} | \langle \varphi, i \partial_{x_2} Q \rangle |
     \leqslant K c^{\prime\prime 2} \ln \left( \frac{1}{c''} \right) \|
     \varphi \|_{H_Q^{\exp}} . \]
  We recall that
  \[ i \vec{e}_2 .H (\psi) = i \partial_{x_2} Q + i \frac{\partial_{x_2} (Q
     \psi) (1 - \eta) + Q \partial_{x_2} \psi \eta \ex^{\psi}}{(1 - \eta) + \eta
     \ex^{\psi}}, \]
  and we check (estimating the local terms in the area where $\eta \neq 1$ by
  Cauchy-Schwarz and $\| c^{\prime\prime 2} \partial_{c''} Q \|_{L^{\infty}
  (\mathbb{R}^2)} \leqslant K$ from Theorem \ref{th1} for $p = + \infty$ and
  Lemma 2.6 of {\cite{Chi_Pac_1}})
  \begin{eqnarray*}
    &  & | (c - c'') \langle i \vec{e}_2 .H (\psi), c^{\prime\prime 2}
    \partial_{c''} Q \rangle - (c - c'') \langle i \partial_{x_2} Q,
    c^{\prime\prime 2} \partial_{c''} Q \rangle |\\
    & \leqslant & K (| c - c'' | \| \varphi \|_{H^1 (\{ \eta \neq 1 \})} + |
    (c - c'') \langle \eta Q i \partial_{x_2} \psi, c^{\prime\prime 2}
    \partial_{c''} Q \rangle |)\\
    & \leqslant & K (| c - c'' | \| \varphi \|_{H_Q^{\exp}} + | (c - c'')
    \langle \eta Q i \partial_{x_2} \psi, c^{\prime\prime 2} \partial_{c''} Q
    \rangle |) .
  \end{eqnarray*}
  We recall from subsection 7.3.1 of {\cite{Chi_Pac_2}} (using decay estimates on
  $c^{\prime\prime 2} \partial_{c''} Q \bar{Q}$ and integrations by parts),
  that
  \[ | (c - c'') \langle \eta Q i \partial_{x_2} \psi, c^{\prime\prime 2}
     \partial_{c''} Q \rangle | \leqslant K | c - c'' | \| \varphi
     \|_{H_Q^{\exp}} \]
  and, from Proposition 1.2 of {\cite{Chi_Pac_2}} (we check easily that the
  translation and phase on $Q$ instead of $Q_{c''}$ do not change the
  computation),
  \[ (c - c'') \langle i \partial_{x_2} Q, c^{\prime\prime 2} \partial_{c''} Q
     \rangle = (2 \pi + o_{c'' \rightarrow 0} (1)) (c - c'') = (2 \pi +
     o_{\nu \rightarrow 0} (1)) (c - c'') . \]
  We deduce that, taking $\nu > 0$ small enough (independently 
  of $ \lambda $), that
  \[ | c - c'' | \leqslant K c^{\prime\prime 2} \ln \left( \frac{1}{c''}
     \right) \| \varphi \|_{H_Q^{\exp}} + K | c - c'' | \| \varphi
     \|_{H_Q^{\exp}} + K | \langle \tmop{NL}_{\tmop{loc}} (\psi) + F (\psi),
     c^{\prime\prime 2} \partial_{c''} Q \rangle | . \]
We take $\nu_2 > 0$ with $\nu_2 \leqslant \nu_1$ such that all the above condition on the smallness of $\nu$ are satisfied if $\nu \leqslant \nu_2$.  Since $\tmop{NL}_{\tmop{loc}} (\psi)$ contains terms at least quadratic in
  $\varphi$, $\| \varphi \|_{C^1 (\{ \eta \neq 1 \})} \leqslant C_3 \nu$ 
  from Lemma \ref{CP3L27} and $\| c^{\prime\prime 2} \partial_{c''} Q
  \|_{L^{\infty} (\mathbb{R}^2)} \leqslant K$, we obtain that for $\nu \leqslant \nu_2$, diminishing 
  $ \nu_2 $ if necessary so that 
 $ \| \varphi \|_{C^1 (\{ \eta \neq 1 \})} \leqslant K \nu \leqslant 1 $,
  \[ | \langle \tmop{NL}_{\tmop{loc}} (\psi), c^{\prime\prime 2}
     \partial_{c''} Q \rangle | \leqslant K \| \varphi \|_{H^1 (\{ \eta \neq 1
     \})}^2 \leqslant K \| \varphi \|_{H_Q^{\exp}}^2 . \]
  Finally, we estimate, using $\| c^{\prime\prime 2} \partial_{c''} Q
  \|_{L^{\infty} (\mathbb{R}^2)} \leqslant K$,
  \[ | \langle Q \eta \nabla \psi . \nabla \psi, c^{\prime\prime 2}
     \partial_{c''} Q \rangle | \leqslant K \int_{\mathbb{R}^2} \eta | \nabla
     \psi |^2 \| c^{\prime\prime 2} \partial_{c''} Q \|_{L^{\infty}
     (\mathbb{R}^2)} \leqslant K \| \varphi \|_{H_Q^{\exp}}^2 . \]
Similarly, since $ \| \eta \mathfrak{Re} (\psi) \|_{L^{\infty} (\{ \tilde{r} \geqslant \lambda \})} 
  \leqslant K \nu $ by 
  Lemma \ref{CP3L27}, diminishing 
  $ \nu_2 $ if necessary, for $\nu \leqslant \nu_2$,
  then $ \| \eta \mathfrak{Re} (\psi) \|_{L^{\infty} (\{ \tilde{r} \geqslant \lambda \})} \leqslant 1 $, hence
  \[ | Q \eta | Q |^2 S (\psi) | = | Q \eta | Q |^2
     (\ex^{2\mathfrak{Re} (\psi)} - 1 - 2\mathfrak{Re}
     (\psi)) | \leqslant K \eta \mathfrak{Re}^2 (\psi), \]
  therefore
  \[ | \langle Q \eta | Q |^2 S (\psi), c^{\prime\prime 2} \partial_{c''} Q
     \rangle | \leqslant K \int_{\mathbb{R}^2} \eta \mathfrak{Re}^2
     (\psi) \| c^{\prime\prime 2} \partial_{c''} Q \|_{L^{\infty}
     (\mathbb{R}^2)} \leqslant K \| \varphi \|_{H_Q^{\exp}}^2 . \]
  This concludes the proof of (\ref{CP3olala}), and therefore of the lemma.
\end{proof}

\subsubsection{Proof of Proposition \ref{locC1} completed}
  
  We take $u$ satisfying the hypotheses of Proposition \ref{locC1} for values of $ \lambda_*,  \lambda, \varepsilon( \lambda), \nu$ described above, with $\nu \leqslant \nu_2$, where $\nu_2$ is defined in Lemma \ref{CP3neurones}.
  We want to take the scalar product of the equation of Lemma \ref{CP3LL26}
  with $\varphi$. It is however not clear at this point that every term is
  integrable. In subsection 7.3 of {\cite{Chi_Pac_2}}, we took the scalar product of
  the equation with $\varphi + i \gamma Q$ for some $\gamma \in \mathbb{R}$,
  using a decay estimate $\| \mathfrak{I}\mathfrak{m} (\psi + i \gamma) (1 +
  r) \|_{L^{\infty} (\{ \tilde{r} \leqslant \lambda \})} \leqslant K (u, Q, c, c'' ) $
  to justify that some terms are well defined, and to do some integration by
  parts. Here, we need to change a little our approach. We first require 
  better decay estimates on $\psi$. At this stage, we know (see Theorem 11 of {\cite{Gra}} 
  and the proof of Lemma \ref{CP3L27}) that
  \begin{eqnarray*}
    &  & \| \Delta \psi (1 + r)^2 \|_{L^{\infty} (\{ \tilde{r} \geqslant
    \lambda \})} + \| (1 + r)^2 \nabla \psi \|_{L^{\infty} (\{ \tilde{r}
    \geqslant \lambda \})}\\
    & + & \| \psi \|_{L^{\infty} (\{ \tilde{r} \geqslant \lambda \})} + \| (1
    + r)^2 \mathfrak{Re} (\psi) \|_{L^{\infty} (\{ \tilde{r}
    \geqslant \lambda \})}\\
    & \leqslant & K (u, Q, c , c'' ) .
  \end{eqnarray*}

  Now, let us show the following improvements:
  \begin{equation}
    \| \mathfrak{I}\mathfrak{m} (\Delta \psi) (1 + r)^3 \|_{L^{\infty} (\{
    \tilde{r} \geqslant \lambda \})} + \| (1 + r)^3 \mathfrak{Re}
    (\nabla \psi) \|_{L^{\infty} (\{ \tilde{r} \geqslant \lambda \})}
    \leqslant K (u, Q,c, c'') . \label{piz}
  \end{equation}
  The proof of $\| (1 + r)^3 | \mathfrak{Re} (\nabla \psi) |
  \|_{L^{\infty} (\{ \tilde{r} \geqslant \lambda \})} \leqslant K (u, Q , c , c'' )$ is
  identical to the one for the same result in Lemma 7.3 of {\cite{Chi_Pac_2}} (see the penultimate estimate of its proof). We
  focus on the estimate on $\mathfrak{I}\mathfrak{m} (\Delta \psi)$. In $\{
  \tilde{r} \geqslant \lambda \}$, we have $u = Q \ex^{\psi}$, therefore,
  \[ \Delta \psi = - \frac{\Delta Q}{Q} + \frac{\Delta u}{u} - 2 \frac{\nabla
     Q}{Q} . \nabla \psi - \nabla \psi . \nabla \psi . \]
  With the previous estimates and Theorem 11 of {\cite{Gra}}, we have
  \[ \left\| \left( - 2 \frac{\nabla Q}{Q} . \nabla \psi - \nabla \psi .
     \nabla \psi \right) (1 + r)^4 \right\|_{L^{\infty} (\{ \tilde{r}
     \geqslant \lambda \})} \leqslant K (u, Q, c, c''), \]
  and since $(\tmop{TW}_{c''}) (Q) = 0$,
  \[ \frac{\Delta Q}{Q} = i c''  \frac{\partial_{x_2} Q}{Q} - (1 - | Q |^2),
  \]
  therefore, with {\cite{Gra}} ($E (Q) < + \infty$),
  \[ \left| \mathfrak{I}\mathfrak{m} \left( \frac{\Delta Q}{Q} \right) \right|
     \leqslant c'' \left| \mathfrak{Re} \left( \frac{\partial_{x_2}
     Q}{Q} \right) \right| \leqslant \frac{K (Q, c'')}{(1 + r)^3} . \]
  Similarly, since $(\tmop{TW}_c) (u) = 0$ and $E (u) < + \infty$,
  \[ \left| \mathfrak{I}\mathfrak{m} \left( \frac{\Delta u}{u} \right) \right|
     \leqslant c \left| \mathfrak{Re} \left( \frac{\partial_{x_2}
     u}{u} \right) \right| \leqslant \frac{K (u,c )}{(1 + r)^3}, \]
  thus
  \[ \| \mathfrak{I}\mathfrak{m} (\Delta \psi) (1 + r)^3 \|_{L^{\infty}
     (\{ \tilde{r} \geqslant \lambda \})} \leqslant K (u, Q, c, c'') . \]

  We infer, with these two additional estimates on $\psi$, that we can do the
  same computations as in the proof of Lemma 7.4 of {\cite{Chi_Pac_2}}, with
  $\gamma = 0$. The only difference is that, when we used $\|
  \mathfrak{I}\mathfrak{m} (\psi + i \gamma) (1 + r) \|_{L^{\infty} (\{
  \tilde{r} \geqslant \lambda \})} \leqslant K (u, Q)$, we can use (\ref{piz})
  instead to get the same decay for these terms, with $\|
  \mathfrak{I}\mathfrak{m} (\psi) \|_{L^{\infty} (\{ \tilde{r} \leqslant
  \lambda \})} \leqslant K (u , Q)$.
  The only two terms where this change is needed are
  \begin{eqnarray*}
  &  & \left| \int_{\mathbb{R}} \eta | Q |^2 \mathfrak{R}\mathfrak{e} (\Delta
  \psi \bar{\psi})  \right|\\
  & \leqslant & \left| \int_{\mathbb{R}} \eta | Q |^2
  \mathfrak{R}\mathfrak{e} (\Delta \psi) \mathfrak{R}\mathfrak{e} (\psi) 
  \right| + \left| \int_{\mathbb{R}} \eta | Q |^2 \mathfrak{I}\mathfrak{m}
  (\Delta \psi) \mathfrak{I}\mathfrak{m} (\psi)  \right|\\
  & \leqslant & K (\| \mathfrak{R}\mathfrak{e} (\Delta \psi) (1 + r)^2
  \|_{L^{\infty} (\{ \tilde{r} \geqslant \lambda \})} \|
  \mathfrak{R}\mathfrak{e} (\psi) (1 + r)^2 \|_{L^{\infty} (\{ \tilde{r}
  \geqslant \lambda \})})\\
  & + & K (\| \mathfrak{I}\mathfrak{m} (\Delta \psi) (1 + r)^3 \|_{L^{\infty}
  (\{ \tilde{r} \geqslant \lambda \})} \| \mathfrak{I}\mathfrak{m} (\psi)
  \|_{L^{\infty} (\{ \tilde{r} \geqslant \lambda \})})
\end{eqnarray*}
and
\begin{eqnarray*}
  &  & \left| \int_{\mathbb{R}} \eta | Q |^2 \mathfrak{R}\mathfrak{e} (i
  \partial_{x_2} \psi \bar{\psi})  \right|\\
  & \leqslant & \left| \int_{\mathbb{R}} \eta | Q |^2
  \mathfrak{R}\mathfrak{e} (\partial_{x_2} \psi) \mathfrak{I}\mathfrak{m}
  (\psi)  \right| + \left| \int_{\mathbb{R}} \eta | Q |^2
  \mathfrak{I}\mathfrak{m} (\partial_{x_2} \psi) \mathfrak{R}\mathfrak{e}
  (\psi)  \right|\\
  & \leqslant & K (\| \mathfrak{R}\mathfrak{e} (\partial_{x_2} \psi) (1 +
  r)^3 \|_{L^{\infty} (\{ \tilde{r} \geqslant \lambda \})} \|
  \mathfrak{I}\mathfrak{m} (\psi) \|_{L^{\infty} (\{ \tilde{r} \geqslant
  \lambda \})})\\
  & + & K (\| \mathfrak{I}\mathfrak{m} (\partial_{x_2} \psi) (1 + r)^2
  \|_{L^{\infty} (\{ \tilde{r} \geqslant \lambda \})} \|
  \mathfrak{R}\mathfrak{e} (\psi) (1 + r)^2 \|_{L^{\infty} (\{ \tilde{r}
  \geqslant \lambda \})}) .
\end{eqnarray*}

  We deduce, taking the scalar product of
  the equation of Lemma \ref{CP3LL26} with $ \varphi $, that
  
  \begin{equation}
    B_Q^{\exp} (\varphi) - \langle i (c - c'') \vec{e}_2 .H (\psi), \varphi
    \rangle + \langle \tmop{NL}_{\tmop{loc}} (\psi), \varphi \rangle + \langle
    F (\psi), \varphi \rangle = 0. \label{CP3eqeq}
  \end{equation}
  From Lemma \ref{CP3L27},
  \begin{equation}
    B_Q^{\exp} (\varphi) \geqslant K \| \varphi \|_{H_{Q_c}^{\exp}}^2,
    \label{CP3eq210}
  \end{equation}
  and from Lemma \ref{CP3LL26},
  \begin{equation}
    | \langle \tmop{NL}_{\tmop{loc}} (\psi), \varphi \rangle | \leqslant K \nu 
     \| \varphi \|^2_{H^1 (\{ \eta \neq 1 \})} 
    \leqslant K \nu 
     \| \varphi \|^2_{H_{Q_c}^{\exp}} . \label{CP3eq211}
  \end{equation}
  Let us now show that
  \begin{equation}
    | \langle i (c - c'') \vec{e}_2 .H (\psi), \varphi \rangle | \leqslant K \nu 
    \| \varphi \|^2_{H_{Q_c}^{\exp}} . \label{CP3eq212}
  \end{equation}
  We recall that
  \[ i \vec{e}_2 .H (\psi) = i \partial_{x_2} Q + i \frac{\partial_{x_2} (Q
     \psi) (1 - \eta) + Q \partial_{x_2} \psi \eta \ex^{\psi}}{(1 - \eta) + \eta
     \ex^{\psi}} . \]
  We compute, with Lemma \ref{CP3neurones} and Lemma 5.4 of {\cite{Chi_Pac_2}}, 
  \[ | (c - c'') \langle i \partial_{x_2} Q, \varphi \rangle | \leqslant K
     \sqrt{c''} \| \varphi \|_{H_Q^{\exp}} | \langle i \partial_{x_2} Q,
     \varphi \rangle | \leqslant K \sqrt{c''} \ln \left( \frac{1}{c''} \right)
     \| \varphi \|_{H_Q^{\exp}}^2 
     \leqslant K \nu \| \varphi \|^2_{H_Q^{\exp}}
     . \]
    Indeed, although $Q = Q_{c''} (. - X \vec{e}_2) \ex^{i \gamma}$ has a phase that is not 
    present in Lemma 5.4 of {\cite{Chi_Pac_2}}, since $ \varphi = Q \psi $, we have 
    $ \partial_{x_2} Q \overline{\varphi} = \partial_{x_2} Q \overline{Q} \overline{\psi}$ that no longer depends on $\gamma$.
    
  Now, with $\| \varphi \|_{H^1 (\{ \eta \neq 1 \})} \leqslant K \nu$ 
  from Lemmas \ref{CP3L26} and \ref{CP3L27}, we compute easily that
  \[ \left| \left\langle i \frac{\partial_{x_2} (Q \psi) (1 - \eta) + Q
     \partial_{x_2} \psi \eta \ex^{\psi}}{(1 - \eta) + \eta \ex^{\psi}}, \varphi
     \right\rangle - \langle i Q \partial_{x_2} \psi \eta, \varphi \rangle
     \right| \leqslant K \nu 
     \| \varphi \|_{H_Q^{\exp}} \]
  since the left hand side is supported in $\{ \eta \neq 1 \}$, therefore
  \[ | \langle i (c - c'') \vec{e}_2 .H (\psi), \varphi \rangle | \leqslant 
  K \nu 
   \| \varphi \|^2_{H_{Q_c}^{\exp}} + | (c - c'') \langle i Q
     \partial_{x_2} \psi \eta, \varphi \rangle | . \]
  With the same computations as in subsection 7.3.2 of {\cite{Chi_Pac_2}} (taking
  $\gamma' = 0$), we check that
  \[ | \langle i Q \partial_{x_2} \psi \eta, \varphi \rangle | \leqslant K \|
     \varphi \|_{H_Q^{\exp}}^2, \]
  therefore, using Lemma \ref{CP3L26} and equation (\ref{CP3cece}), for
  $\nu > 0$ small enough,
  \begin{eqnarray*}
    &  & | (c - c'') \langle i Q \partial_{x_2} \psi \eta, \varphi \rangle
    |\\
    & \leqslant & K | c - c'' | \| \varphi \|_{H_Q^{\exp}}^2\\
    & \leqslant & K \nu \| \varphi \|_{H_Q^{\exp}}^2 .
  \end{eqnarray*}
  This completes the proof of equation (\ref{CP3eq212}). We focus now on the
  proof of
  \begin{equation}
    | \langle F (\psi), \varphi \rangle | \leqslant K \nu \| \varphi
    \|^2_{H_Q^{\exp}} . \label{CP3eq213}
  \end{equation}
  We compute
  \[ \int_{\mathbb{R}^2} \mathfrak{Re} (Q \eta (| Q |^2 S (\psi))
     \bar{\varphi}) = \int_{\mathbb{R}^2} | Q |^4 \eta
     (\ex^{2\mathfrak{Re} (\psi)} - 1 - 2\mathfrak{Re}
     (\psi)) \mathfrak{Re} (\psi), \]
  and since, as already seen at the end of the proof of Lemma \ref{CP3neurones}, 
  we have 
  $ \| \mathfrak{Re} (\psi) \|_{L^\ii ( \{ \tilde{r} \geqslant \lambda \} ) } 
  \leqslant 1$ if $ \nu \leqslant \nu_2 $, we deduce
  \[ | \ex^{2\mathfrak{Re} (\psi)} - 1 - 2\mathfrak{Re}
     (\psi) | \leqslant 
     K \mathfrak{Re}^2 (\psi) \]
  and
  \begin{align*} 
  \left| \int_{\mathbb{R}^2} \mathfrak{Re} (Q \eta (| Q |^2 S
     (\psi)) \bar{\varphi}) \right| \leqslant K \int_{\mathbb{R}^2} \eta
     \mathfrak{Re}^3 (\psi) 
     & \leqslant K \nu 
      \int_{\mathbb{R}^2}
     \eta \mathfrak{Re}^2 (\psi) 
     \\ & \leqslant K \nu
     \| \varphi \|^2_{H_Q^{\exp}} . 
   \end{align*} 
  We are left with the estimation of $\int_{\mathbb{R}^2}
  \mathfrak{Re} (Q \eta (- \nabla \psi . \nabla \psi)
  \bar{\varphi})$, which will be slightly more delicate. First, we compute,
  using $\varphi = Q \psi$
  \begin{eqnarray*}
    &  & \int_{\mathbb{R}^2} \mathfrak{Re} (Q \eta (- \nabla \psi
    . \nabla \psi) \bar{\varphi})\\
    & = & - \int_{\mathbb{R}^2} | Q |^2 \eta \mathfrak{Re} (\nabla
    \psi . \nabla \psi \bar{\psi})\\
    & = & - \int_{\mathbb{R}^2} | Q |^2 \eta \mathfrak{Re} (\nabla
    \psi . \nabla \psi) \mathfrak{Re} (\psi)\\
    & - & \int_{\mathbb{R}^2} | Q |^2 \eta \mathfrak{I}\mathfrak{m} (\nabla
    \psi . \nabla \psi) \mathfrak{I}\mathfrak{m} (\psi)\\
    & = & - \int_{\mathbb{R}^2} | Q |^2 \eta \mathfrak{Re} (\nabla
    \psi . \nabla \psi) \mathfrak{Re} (\psi)\\
    & - & 2 \int_{\mathbb{R}^2} | Q |^2 \eta \mathfrak{Re} (\nabla
    \psi) .\mathfrak{I}\mathfrak{m} (\nabla \psi) \mathfrak{I}\mathfrak{m}
    (\psi) .
  \end{eqnarray*}
  Remark that there exists a universal constant $K >0 $ such that $ \| \mathfrak{Re} (\psi) \|_{L^\ii ( \{ \tilde{r} \geqslant R \} ) } \leqslant K \nu $ by Lemma \ref{CP3L27} (considering the regions $\{ \tilde{r} \geqslant \lambda \} $ with $ \psi $ 
  and $ \{ \tilde{r} \leqslant \lambda \} $ with $\varphi $). Then, we estimate 
  \[ \left| - \int_{\mathbb{R}^2} | Q |^2 \eta \mathfrak{Re}
     (\nabla \psi . \nabla \psi) \mathfrak{Re} (\psi) \right|
     \leqslant K \nu \int_{\mathbb{R}^2} \eta | \nabla \psi |^2 \leqslant K
     \nu \| \varphi \|^2_{H_Q^{\exp}} . \]
  Now, by integration by parts (that can be justified as in {\cite{Chi_Pac_2}}), we
  have
  \begin{eqnarray*}
    &  & \int_{\mathbb{R}^2} | Q |^2 \eta \mathfrak{Re} (\nabla
    \psi) .\mathfrak{I}\mathfrak{m} (\nabla \psi) \mathfrak{I}\mathfrak{m}
    (\psi)\\
    & = & - \int_{\mathbb{R}^2} \nabla (| Q |^2) \eta
    \mathfrak{Re} (\psi) .\mathfrak{I}\mathfrak{m} (\nabla \psi)
    \mathfrak{I}\mathfrak{m} (\psi)\\
    & - & \int_{\mathbb{R}^2} | Q |^2 \nabla \eta \mathfrak{Re}
    (\psi) .\mathfrak{I}\mathfrak{m} (\nabla \psi) \mathfrak{I}\mathfrak{m}
    (\psi)\\
    & - & \int_{\mathbb{R}^2} | Q |^2 \eta \mathfrak{Re} (\psi)
    \mathfrak{I}\mathfrak{m} (\Delta \psi) \mathfrak{I}\mathfrak{m} (\psi)\\
    & - & \int_{\mathbb{R}^2} | Q |^2 \eta \mathfrak{Re} (\psi)
    \mathfrak{I}\mathfrak{m} (\nabla \psi) .\mathfrak{I}\mathfrak{m} (\nabla
    \psi),
  \end{eqnarray*}
  and with $| \nabla (| Q |^2) | \leqslant \frac{K}{(1 + \tilde{r})^{5 / 2}}$ from 
  equation (2.9) of {\cite{Chi_Pac_2}} (for $\sigma = 1 / 2$) with $K > 0$ a 
  universal constant, we have by Cauchy-Schwarz
  \begin{eqnarray*}
    &  & \left| \int_{\mathbb{R}^2} \nabla (| Q |^2) \eta
    \mathfrak{Re} (\psi) .\mathfrak{I}\mathfrak{m} (\nabla \psi)
    \mathfrak{I}\mathfrak{m} (\psi) \right|\\
    & \leqslant & K \nu \sqrt{\int_{\mathbb{R}^2} \eta | \nabla \psi |^2
    \int_{\mathbb{R}^2} \eta \frac{| \psi |^2}{(1 + \tilde{r})^5}}\\
    & \leqslant & K \nu \| \varphi \|^2_{H_Q^{\exp}}
  \end{eqnarray*}
  and
  \[ \left| \int_{\mathbb{R}^2} | Q |^2 \eta \mathfrak{Re} (\psi)
     \mathfrak{I}\mathfrak{m} (\nabla \psi) .\mathfrak{I}\mathfrak{m} (\nabla
     \psi) \right| \leqslant K \nu \int_{\mathbb{R}^2} \eta | \nabla \psi |^2
     \leqslant K \nu \| \varphi \|^2_{H_{Q_c}^{\exp}} . \]
  Since $\nabla \eta$ is supported in $\{ 0 < \eta < 1 \}$, 
  we check easily that
  \[ \left| \int_{\mathbb{R}^2} | Q |^2 \nabla \eta \mathfrak{Re}
     (\psi) .\mathfrak{I}\mathfrak{m} (\nabla \psi) \mathfrak{I}\mathfrak{m}
     (\psi) \right| \leqslant K \nu \| \varphi \|^2_{H_Q^{\exp}} . \]
  We focus now on the estimation of the last remaining term,
  $\int_{\mathbb{R}^2} | Q |^2 \eta \mathfrak{Re} (\psi)
  \mathfrak{I}\mathfrak{m} (\Delta \psi) \mathfrak{I}\mathfrak{m} (\psi)$. For
  that purpose, we define more generally for $n \geqslant 1$
  \[ A_n \assign \int_{\mathbb{R}^2} | Q |^2 \eta^n \mathfrak{Re}^n
     (\psi) \mathfrak{I}\mathfrak{m} (\Delta \psi) \mathfrak{I}\mathfrak{m}
     (\psi) . \]
  Remark that we want to estimate $A_1$.
  
  \bigbreak
  
  We compute, using that $(\tmop{TW}_{c''}) (Q) = 0$, that
  \[ L_Q (\varphi) = Q \left( - \Delta \psi - i c'' \partial_{x_2} \psi - 2
     \frac{\nabla Q}{Q} . \nabla \psi + 2\mathfrak{Re} (\psi) | Q
     |^2 \right), \]
  therefore, by Lemma \ref{CP3LL26}, in $\{ \eta \neq 0 \}$,
  \begin{eqnarray*}
    \mathfrak{I}\mathfrak{m} (\Delta \psi) & = & \mathfrak{I}\mathfrak{m}
    \left( - i c'' \partial_{x_2} \psi - 2 \frac{\nabla Q}{Q} . \nabla \psi +
    2\mathfrak{Re} (\psi) | Q |^2 + \frac{- i (c - c'') \vec{e}_2
    .H (\psi) + \tmop{NL}_{\tmop{loc}} (\psi) + F (\psi)}{Q} \right)\\
    & = & - c'' \mathfrak{Re} (\partial_{x_2} \psi) -
    2\mathfrak{I}\mathfrak{m} \left( \frac{\nabla Q}{Q} . \nabla \psi \right)
    +\mathfrak{I}\mathfrak{m} \left( \frac{- i (c - c'') \vec{e}_2 .H (\psi) +
    \tmop{NL}_{\tmop{loc}} (\psi) + F (\psi)}{Q} \right) .
  \end{eqnarray*}
  We compute, by integration by parts, with $\mathfrak{Re}^n (\psi)
  \mathfrak{Re} (\partial_{x_2} \psi) = \frac{1}{n + 1}
  \partial_{x_2} (\mathfrak{Re}^{n + 1} (\psi))$, that
  \begin{eqnarray*}
    &  & \int_{\mathbb{R}^2} | Q |^2 \eta^n \mathfrak{Re}^n (\psi)
    c'' \mathfrak{Re} (\partial_{x_2} \psi)
    \mathfrak{I}\mathfrak{m} (\psi)\\
    & = & \frac{- 1}{n + 1} \int_{\mathbb{R}^2} (\partial_{x_2} | Q |^2)
    \eta^n \mathfrak{Re}^{n + 1} (\psi) c''
    \mathfrak{I}\mathfrak{m} (\psi)\\
    & - & \frac{n}{n + 1} \int_{\mathbb{R}^2} | Q |^2 \partial_{x_2} \eta
    \eta^{n - 1} \mathfrak{Re}^{n + 1} (\psi) c''
    \mathfrak{I}\mathfrak{m} (\psi)\\
    & - & \frac{1}{n + 1} \int_{\mathbb{R}^2} | Q |^2 \eta^n
    \mathfrak{Re}^{n + 1} (\psi) c'' \mathfrak{I}\mathfrak{m}
    (\partial_{x_2} \psi) .
  \end{eqnarray*}
  Since $| c'' | \leqslant \nu$ by equation (\ref{CP3ccp}) (diminishing 
  $ \nu_2 $ if necessary), Lemma \ref{CP3L26} and the hypotheses of
  Proposition \ref{locC1}, $\| \varphi \|_{C^1 (\{ \tilde{r} \leqslant \lambda
  \})} + \| \mathfrak{Re} (\psi) \|_{L^{\infty} (\{ \tilde{r}
  \geqslant \lambda \})} \leqslant K \nu$ by Lemma \ref{CP3L27} and $| \nabla
  (| Q |^2) | \leqslant \frac{K}{(1 + \tilde{r})^{5 / 2}}$ from equation (2.9)
  of {\cite{Chi_Pac_2}}, we infer by Cauchy-Schwarz that
  \begin{eqnarray}
    &  & \left| \int_{\mathbb{R}^2} (\partial_{x_2} | Q |^2) \eta^n
    \mathfrak{Re}^{n + 1} (\psi) c'' \mathfrak{I}\mathfrak{m}
    (\psi) \right| \nonumber\\
    & \leqslant & 
    K c'' \nu^n \sqrt{\int_{\mathbb{R}^2} \eta
    \mathfrak{I}\mathfrak{m}^2 (\psi) (\partial_{x_2} | Q |^2)^2
    \int_{\mathbb{R}^2} \eta \mathfrak{Re}^2 (\psi)} 
    \nonumber\\
    & \leqslant & K \nu^n \| \varphi \|^2_{H_Q^{\exp}},  \label{320ziz}
  \end{eqnarray}
  \begin{equation}
    \left| \int_{\mathbb{R}^2} | Q |^2 \partial_{x_2} \eta \eta^{n - 1}
    \mathfrak{Re}^{n + 1} (\psi) c'' \mathfrak{I}\mathfrak{m}
    (\psi) \right| \leqslant K \nu^n \| \varphi \|^2_{H_Q^{\exp}}
  \end{equation}
  and
  \begin{equation}
    \left| \int_{\mathbb{R}^2} | Q |^2 \eta^n \mathfrak{Re}^{n + 1}
    (\psi) c'' \mathfrak{I}\mathfrak{m} (\partial_{x_2} \psi) \right|
    \leqslant K \nu^n \sqrt{\int_{\mathbb{R}^2} \eta | \nabla \psi |^2
    \int_{\mathbb{R}^2} \eta \mathfrak{Re}^2 (\psi)} \leqslant K
    \nu^n \| \varphi \|^2_{H_Q^{\exp}} .
  \end{equation}
  We deduce that 
  \begin{equation}
    \left| \int_{\mathbb{R}^2} | Q |^2 \eta^n \mathfrak{Re}^n
    (\psi) c'' \mathfrak{Re} (\partial_{x_2} \psi)
    \mathfrak{I}\mathfrak{m} (\psi) \right| \leqslant (K \nu)^n \| \varphi
    \|^2_{H_Q^{\exp}} . \label{CP32177}
  \end{equation}
  For $\int_{\mathbb{R}^2} | Q |^2 \eta^n \mathfrak{Re}^n (\psi)
  \mathfrak{I}\mathfrak{m} \left( \frac{\nabla Q}{Q} . \nabla \psi \right)
  \mathfrak{I}\mathfrak{m} (\psi)$, we compute
  \[ \mathfrak{I}\mathfrak{m} \left( \frac{\nabla Q}{Q} . \nabla \psi \right)
     =\mathfrak{Re} \left( \frac{\nabla Q}{Q} \right)
     .\mathfrak{I}\mathfrak{m} (\nabla \psi) +\mathfrak{Re} (\nabla
     \psi) .\mathfrak{I}\mathfrak{m} \left( \frac{\nabla Q}{Q} \right), \]
  and with previous estimates, we check easily that
  \begin{eqnarray}
    &  & \left| \int_{\mathbb{R}^2} | Q |^2 \eta^n \mathfrak{Re}^n
    (\psi) \mathfrak{Re} \left( \frac{\nabla Q}{Q} \right)
    .\mathfrak{I}\mathfrak{m} (\nabla \psi) \mathfrak{I}\mathfrak{m} (\psi)
    \right| \nonumber\\
    & \leqslant & (K \nu)^n \sqrt{\int_{\mathbb{R}^2} \eta | \nabla \psi |^2
    \int_{\mathbb{R}^2} \eta \mathfrak{I}\mathfrak{m}^2 (\psi)
    \mathfrak{Re}^2 \left( \frac{\nabla Q}{Q} \right)} \leqslant (K
    \nu)^n \| \varphi \|^2_{H_Q^{\exp}}, 
  \end{eqnarray}
  and by integration by parts, with computations similar to those for the proof of
  (\ref{CP32177}), using
  \[ \left| \nabla .\mathfrak{I}\mathfrak{m} \left( \frac{\nabla Q}{Q} \right)
     \right| \leqslant \frac{K}{(1 + \tilde{r})^{3 / 2}} \]
  from (2.9) to (2.11) of {\cite{Chi_Pac_2}} (for $\sigma = 1 / 2$) for a universal constant $K > 0$ and Lemma
  \ref{lemme3new}, we infer that
  \begin{equation}
    \left| \int_{\mathbb{R}^2} | Q |^2 \eta^n \mathfrak{Re}^n
    (\psi) \mathfrak{Re} (\nabla \psi) .\mathfrak{I}\mathfrak{m}
    \left( \frac{\nabla Q}{Q} \right) \mathfrak{I}\mathfrak{m} (\psi) \right|
    \leqslant (K \nu)^n \| \varphi \|^2_{H_Q^{\exp}},
  \end{equation}
  and we check easily that
  \begin{equation}
    \left| \int_{\mathbb{R}^2} | Q |^2 \eta^n \mathfrak{Re}^n
    (\psi) \mathfrak{I}\mathfrak{m} \left( \frac{\tmop{NL}_{\tmop{loc}}
    (\psi)}{Q} \right) \mathfrak{I}\mathfrak{m} (\psi) \right| \leqslant (K
    \nu)^n \| \varphi \|^2_{H_Q^{\exp}} .
  \end{equation}

  Now, we look at $\int_{\mathbb{R}^2} | Q |^2 \eta^n
  \mathfrak{Re}^n (\psi) \mathfrak{I}\mathfrak{m} \left( \frac{- i
  (c - c'') \vec{e}_2 .H (\psi)}{Q} \right) \mathfrak{I}\mathfrak{m} (\psi)$,
  for the part of $\vec{e}_2 .H (\psi)$ related to the cutoff, the estimation
  can be done as previously, and we are left with the estimation of
  \begin{eqnarray*}
    &  & (c - c'') \int_{\mathbb{R}^2} | Q |^2 \eta^n
    \mathfrak{Re}^n (\psi) \mathfrak{I}\mathfrak{m} \left( - i
    \frac{\partial_{x_2} Q}{Q} - i \partial_{x_2} \psi \right)
    \mathfrak{I}\mathfrak{m} (\psi)\\
    & = & (c - c'') \int_{\mathbb{R}^2} | Q |^2 \eta^n
    \mathfrak{Re}^n (\psi) \mathfrak{Re} \left( 
    \frac{\partial_{x_2} Q}{Q} + \partial_{x_2} \psi \right)
    \mathfrak{I}\mathfrak{m} (\psi) .
  \end{eqnarray*}
  From equation (\ref{CP3ccp}) and Lemma \ref{CP3L26}, we have $| c - c'' |
  \leqslant \nu$ (diminishing 
  $ \nu_2 $ if necessary), and from equation (2.9)
  of {\cite{Chi_Pac_2}}, $\left| \mathfrak{Re} \left( 
  \frac{\partial_{x_2} Q}{Q} \right) \right| \leqslant \frac{K}{(1 +
  \tilde{r})^{5 / 2}}$, therefore
  \begin{eqnarray}
    &  & \left| (c - c'') \int_{\mathbb{R}^2} | Q |^2 \eta^n
    \mathfrak{Re}^n (\psi) \mathfrak{Re} \left( 
    \frac{\partial_{x_2} Q}{Q} \right) \mathfrak{I}\mathfrak{m} (\psi) \right|
    \nonumber\\
    & \leqslant & (K \nu)^n \sqrt{\int_{\mathbb{R}^2} \eta
    \mathfrak{Re}^2 (\psi) \int_{\mathbb{R}^2} \eta
    \mathfrak{Re}^2 \left(  \frac{\partial_{x_2} Q}{Q} \right)
    \mathfrak{I}\mathfrak{m}^2 (\psi)} \nonumber\\
    & \leqslant & (K \nu)^n \| \varphi \|^2_{H_Q^{\exp}}, 
  \end{eqnarray}
  and we estimate
  \begin{equation}
    \left| (c - c'') \int_{\mathbb{R}^2} | Q |^2 \eta^n
    \mathfrak{Re}^n (\psi) \mathfrak{Re} (\partial_{x_2}
    \psi) \mathfrak{I}\mathfrak{m} (\psi) \right| \leqslant (K \nu)^n \|
    \varphi \|^2_{H_Q^{\exp}}
  \end{equation}
  by (\ref{CP32177}). For the last remaining term, since
  \[ \mathfrak{I}\mathfrak{m} \left( \frac{F (\psi)}{Q} \right)
     =\mathfrak{I}\mathfrak{m} (- \eta \nabla \psi . \nabla \psi), \]
  we have $\int_{\mathbb{R}^2} | Q |^2 \eta^n \mathfrak{Re}^n
  (\psi) \mathfrak{I}\mathfrak{m} \left( \frac{F (\psi)}{Q} \right)
  \mathfrak{I}\mathfrak{m} (\psi) = - 2 \int_{\mathbb{R}^2} | Q |^2 \eta^{n +
  1} \mathfrak{Re}^n (\psi) \mathfrak{I}\mathfrak{m} (\nabla \psi)
  .\mathfrak{Re} (\nabla \psi) \mathfrak{I}\mathfrak{m} (\psi)$. In
  particular,
  \begin{eqnarray}
    &  & \left| \int_{\mathbb{R}^2} | Q |^2 \eta^n \mathfrak{Re}^n
    (\psi) \mathfrak{I}\mathfrak{m} \left( \frac{F (\psi)}{Q} \right)
    \mathfrak{I}\mathfrak{m} (\psi) \right| \nonumber\\
    & \leqslant & (K \nu)^n \| \eta \mathfrak{I}\mathfrak{m} (\psi)
    \|_{L^{\infty} (\mathbb{R}^2)} \int_{\mathbb{R}^2} \eta | \nabla \psi |^2
    \nonumber\\
    & \leqslant & (K \nu)^n \| \eta \mathfrak{I}\mathfrak{m} (\psi)
    \|_{L^{\infty} (\mathbb{R}^2)} \| \varphi \|^2_{H_Q^{\exp}} . 
    \label{329ziz}
  \end{eqnarray}
  Combining this result with the previous estimates, this implies that
  \begin{equation}
    | A_n | \leqslant (C_6 \nu)^n (1 + \| \eta \mathfrak{I}\mathfrak{m} (\psi)
    \|_{L^{\infty} (\mathbb{R}^2)}) \| \varphi \|^2_{H_Q^{\exp}}
    \label{CP32188}
  \end{equation}
  for some universal constant $C_6 > 0$, but that is not enough to show that
  we have $\left| \int_{\mathbb{R}^2} | Q |^2 \eta^n
  \mathfrak{Re}^n (\psi) \mathfrak{I}\mathfrak{m} \left( \frac{F
  (\psi)}{Q} \right) \mathfrak{I}\mathfrak{m} (\psi) \right| \leqslant (K
  \nu)^n \| \varphi \|^2_{H_Q^{\exp}}$, since we have no control on $\| \eta
  \mathfrak{I}\mathfrak{m} (\psi) \|_{L^{\infty} (\mathbb{R}^2)}$ other than
  the fact that it is a finite quantity. By integration by parts (integrating
  $\mathfrak{Re} (\nabla \psi)$), with computations similar as for
  the proof of (\ref{CP32177}), we infer that
  \begin{eqnarray*}
    &  & \left| 2 \int_{\mathbb{R}^2} | Q |^2 \eta^{n + 1}
    \mathfrak{Re}^n (\psi) \mathfrak{I}\mathfrak{m} (\nabla \psi)
    .\mathfrak{Re} (\nabla \psi) \mathfrak{I}\mathfrak{m} (\psi)
    \right|\\
    & \leqslant & \left| 2 \int_{\mathbb{R}^2} | Q |^2 \eta^{n + 1}
    \mathfrak{Re}^n (\psi) \mathfrak{I}\mathfrak{m} (\Delta \psi)
    \mathfrak{Re} (\psi) \mathfrak{I}\mathfrak{m} (\psi) \right| +
    (K \nu)^n \| \varphi \|^2_{H_Q^{\exp}}\\
    & \leqslant & 2 | A_{n + 1} | + (K \nu)^n \| \varphi \|^2_{H_Q^{\exp}} .
  \end{eqnarray*}
  Combining this result with estimates (\ref{320ziz}) to (\ref{329ziz}), we
  deduce that for some universal constant $C_7 > 0$,
  \[ | A_n | \leqslant 2 | A_{n + 1} | + (C_7 \nu)^n \| \varphi
     \|^2_{H_Q^{\exp}}, \]
  therefore, by induction,
  \[ | A_1 | \leqslant 2^n | A_n | + \sum_{k = 1}^{n - 1} (2 C_7 \nu)^k \|
     \varphi \|^2_{H_Q^{\exp}}, \]
  hence, with (\ref{CP32188}),
  \[ | A_1 | \leqslant \left( (2 C_6 \nu)^n (1 + \| \eta
     \mathfrak{I}\mathfrak{m} (\psi) \|_{L^{\infty} (\mathbb{R}^2)}) + \sum_{k
     = 1}^{n - 1} \nobracket (2 C_7 \nu \nobracket)^k \right) \| \varphi
     \|^2_{H_Q^{\exp}} . \]
  Taking $\nu > 0$ such that $\nu \leqslant \nu_2$ and $2 C_6 \nu < 1 / 2$ and $2 C_7 \nu < 1 / 2$, then
  $n \geqslant 1$ large enough (depending on $\| \eta \mathfrak{I}\mathfrak{m}
  (\psi) \|_{L^{\infty} (\mathbb{R}^2)}$) such that
  \[ \frac{1}{2^{n - 1}} (1 + \| \eta \mathfrak{I}\mathfrak{m} (\psi)
     \|_{L^{\infty} (\mathbb{R}^2)}) \leqslant 1, \]
  we conclude that
  \[ | A_1 | \leqslant \left( 2 C_6 + 2 C_7 \sum_{k = 0}^{n - 2} \frac{1}{2^k}
     \right) \nu \| \varphi \|^2_{H_Q^{\exp}} \leqslant 2 (C_6 + 2 C_7) \nu \|
     \varphi \|^2_{H_Q^{\exp}} . \]
  This concludes the proof of equation (\ref{CP3eq213}).
  
  \
  
  Combining estimates (\ref{CP3eq210}) to (\ref{CP3eq213}) in equation
  (\ref{CP3eqeq}), we deduce that
  \[ (1 - C_8 \nu) \| \varphi \|_{H_Q^{\exp}}^2 \leqslant 0 \]
  for some universal constant $C_8 > 0$, therefore, taking $\nu > 0$ small enough such that 
  the previous constraints are satisfied and
  $C_8 \nu < 1 / 2$ , we have $\| \varphi \|_{H_Q^{\exp}} = 0$. From Lemma
  \ref{CP3neurones}, we deduce $c'' = c$. The proof is complete.

\subsection{Proof of Corollary \ref{CP3cor}}

  Take a function $u$ satisfying the hypotheses of Corollary \ref{CP3cor}.
  Then, $u$ is even in $x_1$ and it has finite energy. Furthermore, by Theorem
  \ref{th1} (for $p = + \infty$),
  \begin{eqnarray*}
    &  & \| u - V_1 ( \cdot - d_c \vec{e}_1) V_{- 1} ( \cdot + d_c \vec{e}_1)
    \|_{L^{\infty} (\mathbb{R}^2)}\\
    & \leqslant & \| u - Q_c \|_{L^{\infty} (\mathbb{R}^2)} + \| Q_c - V_1 (\cdot 
    - d_c \vec{e}_1) V_{- 1} (\cdot  + d_c \vec{e}_1) \|_{L^{\infty}
    (\mathbb{R}^2)}\\
    & \leqslant & \varepsilon + o_{c \rightarrow 0} (1) .
  \end{eqnarray*}
 Next,
  \[ \| | u | - 1 \|_{L^{\infty} (\{ \tilde{r}_d \geqslant \lambda \})}
     \leqslant \| u - Q_c \|_{L^{\infty} (\{ \tilde{r}_d \geqslant \lambda
     \})} + \| | Q_c | - 1 \|_{L^{\infty} (\{ \tilde{r}_d \geqslant \lambda
     \})} \leqslant \varepsilon + \frac{K}{\lambda} \]
  by equation (2.6) of {\cite{Chi_Pac_2}}. We now fix the parameters. 
  We first choose $ \lambda \geqslant \lambda_* $ large enough so that 
  $ K / \lambda \leqslant 1/(2 \lambda_* ) $. Then, we fix $ c_0 > 0 $ and 
  $\varepsilon > 0 $ so small that $ \varepsilon \leqslant 1/(2 \lambda_* ) $, 
  $ \lvert c d_c - 1 \rvert \ls \varepsilon (\lambda ) $, $ d_c \geqslant 1/ \varepsilon (\lambda ) $ 
  and $ \varepsilon + o_{c \rightarrow 0} (1) \leqslant \varepsilon (\lambda ) $ 
  for $ c < c_0 $. Therefore, $u$ satisfies the hypotheses of Proposition \ref{locC1} 
  with $ d = d_c $, and this concludes.

%%%%%%%%%%%%%%%%%%%%%%%%%%%%%%%%%%%%%%%%%%%%%%%%%%%%%%%
%%%%%%%%%%%%%%%%%%%%%%%%%%%%%%%%%%%%%%%%%%%%%%%%%%%%%%%
%%%%%%%%%%%%%%%%%%%%%%%%%%%%%%%%%%%%%%%%%%%%%%%%%%%%%%%
\section{Properties of quasi-minimizers of the energy and proof of Theorem \ref{th8} \label{sec3}}
\label{sec:quasimin}

%%%%%
\subsection{Tools for the vortex analysis}

We list in this section some results useful for the analysis of travelling waves 
for small speeds or, equivalently, large momentum, with vorticity. We shall denote 
$ \langle u| v \rangle = \textrm{Re} ( u \bar{v} ) $ the real scalar product of the complex 
numbers $u$, $v $. The jacobian (or vorticity)
\[
 J v \assign \langle i \p_1 v | \p_2 v \rangle 
 = \frac{1}{2} \p_1 \langle i v | \p_2 v \rangle - \frac{1}{2} \p_2 \langle i v | \p_1 v \rangle
\]
is then relevant, and we shall use the following concentration property of the jacobian. We denote
\[
 E_\e ( u , \Omega ) 
 \assign \frac{1}{2} \int_{ \Omega } \lvert \nabla u \rvert^2 + \frac{1}{2\e^2} ( 1 - |u|^2 )^2 \, dx .
\]

\begin{theorem}[Concentration of the Jacobian - \cite{Alb_Bal_Orl}, \cite{Jer_Son}]
\label{Jacob}
Let $ M_0 > 0 $, $ R > 0 $ and $ \beta \in ] 0,1 ] $. Then, for every $ \delta >0 $, there exists $ \e_0 > 0 $ 
(depending only on $ \beta $, $ \delta $, $ R $ and $ M_0 $) such that for any $ 0 < \e < \e_0 $, and for 
any $ u \in H^1( B(0,4R) ) $ such that $ E_\e ( u , B(0, 4R) ) \ls M_0 \lvert \ln \e \rvert $ 
and $ |u| \gs 1/2 $ in $B (0, 4R) \setminus B(0, R) $, 
there exist $ N \in \N $, $ y_1, \dotsc , y_N \in \bar{B}(0, R ) $, $ d_1, \dotsc , d_N \in \Z $ such that
\[
 \Big\| Ju - \pi \sum_{k=1}^N d_k \delta_{y_k} \Big\|_{ [ \BC^{0, \beta}_c( B(0, 4R) ) ]^*} \ls \delta
\]
and
\[
 \pi \sum_{k=1}^N \lvert d_k \rvert \ls \frac{E_\e (u, B(0,4R )) }{ \lvert \ln \e \rvert } + \delta .
\]
Finally, we may choose the points $ y_k $, $1 \ls k \ls N $, in $ \{ |u | \ls 1/2 \} $. 
\end{theorem}

Here, we recall that the space $ [ \BC^{0, \beta}_c( B(0, R) ) ]^* $ is endowed with the dual norm 
associated with $ \| \zeta \|_{ \BC^{0, \beta}_c( B(0, R) )} = \sup_{x \neq y \in B(0,R) } \frac{ \lvert \zeta( x) - \zeta( y) \rvert}{|x-y|^\beta} $, 
for $ \zeta \in \BC^{0, \beta} ( B(0, R) ) $ compactly supported.

%%%%
\begin{remark} 
The above mentioned theorem is actually Lemma 3.3 in \cite{BOS}. It is related to the works 
\cite{Alb_Bal_Orl}, \cite{Jer_Son}, which both correspond to the limit $ \e \to 0 $, whereas 
we have here a statement (obtained by compactness) at fixed $ \e $. The hypothesis 
"$ |u| \gs 1/2 $ in $B (0, 4R) \setminus B(0, R) $" ensures that the vortices do not approach the 
boundary $ \p B (0,4R) $.

\end{remark}
%%%%

\begin{theorem}[Clearing-out Theorem - \cite{BOS}] 
\label{clearingout}
Let $ M_0 > 0 $ and $ \sigma > 0 $ be given. Then there exist $ \ep_0 > 0 $ and $ \eta > 0 $, depending 
only on $ M_0 $ and $ \sigma $, such that, if $ R_0 = 1/ (1 + M_0 ) $, if $ U : B (0, R_0) \to \C $ solves
\begin{equation}
\label{GL_transpo}
 \Delta U + i \mathfrak{c} \p_2 U + \frac{1}{\ep^2} U ( 1 - |U|^2 ) = 0 
\end{equation}
in $ B (0, R_0) \subset \R^2 $, with $ \ep < \ep_0 $, $ \lvert \mathfrak{c} \rvert \ls M_0 \lvert \ln \ep \rvert $, and 
\[
 E_\ep (U, B (0, R_0) ) \ls \eta \lvert \ln \ep \rvert ,
\]
then
\[
 \lvert U(0) \rvert \gs 1 - \sigma .
\]
\end{theorem}

For the elliptic PDE
\begin{equation}
\label{GL}
 \Delta \BU + \frac{1}{\e^2} \BU ( 1 - |\BU|^2 ) = 0 ,
\end{equation}
that is without the transport term $  i \p_2 U $, this result has been shown in 2d in \cite{BBH} for 
minimizing maps, then in \cite{BethRiv} for the Ginzburg-Landau equation with magnetic field. 
In higher dimension, see \cite{Lin_Riviere} and \cite{BBO} for \eqref{GL} and \cite{BOS} for an 
equation including the Ginzburg-Landau equation with magnetic field and \eqref{GL_transpo}. 
One may use the change of unknown
\[
 \BU(x) \assign ( 1+ \mathfrak{c}^2 \ep^2 /4 )^{-1/2} \ex^{i \mathfrak{c} x_2/2} U(x) ,
 \quad \quad \quad 
 \e = \ep ( 1+ \mathfrak{c}^2 \ep^2/4 )^{-1/2} 
\]
to transform the equation \eqref{GL} without the transport term into the equation 
\eqref{GL_transpo} with the transport term. However, the assumptions 
$ E_\ep (U, B(0, R_0) ) \ls \eta \lvert \ln \ep \rvert $ and $ E_\e ( \BU, B (0, R_0) ) \ls \eta \lvert \ln \e \rvert $ 
are not equivalent (due to the extra phase term).

%%%%%
\subsection{Vortex structure for quasi-minimizers of $E$ at fixed $P$}

In this section, some $ \Lambda_0 > 0 $ is fixed and we consider a large momentum 
$ \mathfrak{p} $ and $ u_\mathfrak{p} $ such that
\begin{equation}
\label{upper_bound}
 E( u_\mathfrak{p} ) \ls 2\pi \ln \mathfrak{p} + \Lambda_0 
\end{equation}
and such that there exists $ c_\mathfrak{p} > 0 $ (depending on $ u_\mathfrak{p} $) such that
\[
 0 = (\tmop{TW}_{c_\mathfrak{p}} ) (u_\mathfrak{p}) 
 = - i c_\mathfrak{p} \partial_{x_2} u_\mathfrak{p} - \Delta u_\mathfrak{p} - (1 - | u_\mathfrak{p} |^2) u_\mathfrak{p} .
\]
It then follows from \cite{Gra} (see Theorem \ref{decay}) that we may assume, using the phase shift invariance, 
that $ u_\mathfrak{p} \to 1 $ at spatial infinity. In particular, we have
\[
 \mathfrak{p} 
 = P_2 ( u_\mathfrak{p} ) 
 = \frac{1}{2} \int_{\R^2} \langle i \p_2 u_\mathfrak{p} | u_\mathfrak{p} - 1 \rangle \, dx .
\]

Our goal is to show that $ u_\mathfrak{p}$ satisfies the hypothesis of Proposition \ref{locC1}.
We shall follow \cite{Bet_Saut} and \cite{BOS} in order to analyze the vortex structure of $ u_\mathfrak{p} $.

%%%%%%%%%%%%%%%%%%%%%%%%%%%%%%%%%%%%%%%%%%%%%%%%%%%%%%%
\subsubsection{Localizing the vorticity set at scale $ x / \mathfrak{p} $}
\label{sec:scalezchapeau}

We define the following rescaling $ \hat{u}_\mathfrak{p} $ of $ u_\mathfrak{p} $:
\begin{equation}
\label{defuhat}
 \hat{u}_\mathfrak{p} ( \hat{x} ) = u_\mathfrak{p} ( \mathfrak{p} \hat{x} ) .
\end{equation}
Therefore, $ \hat{u}_\mathfrak{p} $ solves
\begin{equation}
\label{TW_resc}
 \Delta \hat{u}_\mathfrak{p} + i c_\mathfrak{p} \mathfrak{p} \p_2 \hat{u}_\mathfrak{p} 
 + \mathfrak{p}^2 \hat{u}_\mathfrak{p} ( 1 - \lvert \hat{u}_\mathfrak{p} \rvert^2 ) = 0 
\end{equation}
which is a particular case of \eqref{GL_transpo} with
\[
 \ep = 1 / \mathfrak{p}, 
 \quad \quad \quad 
 \mathfrak{c} = c_\mathfrak{p} \mathfrak{p} .
\]
The universal $ L^\ii $ bound on the gradient of Corollary \ref{Gradinf} reads now
\begin{equation}
\label{LinfGradResc}
 \| \nabla \hat{u}_\mathfrak{p} \|_{L^\ii (\R^2) } \ls K \mathfrak{p} .
\end{equation}
We shall have, in the end, $ c_\mathfrak{p} \sim 1/ \mathfrak{p} $. 
The first step provides a rough upper bound for the speed $ c_\mathfrak{p} $ (the Lagrange 
multiplier for the minimisation problem $ E_{\rm min} (\mathfrak{p} ) $).

\bigbreak
\noindent {\it Step 1: there exists $ \mathfrak{p}_1 = \mathfrak{p}_1 ( \Lambda_0) $ such that, for 
$ \mathfrak{p} \gs \mathfrak{p}_1 $, we have
\[
 0 < c_\mathfrak{p} \ls \frac{2E ( u_\mathfrak{p} ) }{ \mathfrak{p} } 
 \ls 13 \frac{ \ln \mathfrak{p} }{ \mathfrak{p} } .
\]
In particular, $ c_\mathfrak{p} \ls 1/2 $ and $ \ln \mathfrak{p} \ls 2 \lvert \ln c_\mathfrak{p} \rvert $. }

\bigbreak

We shall use the Pohozaev identity \eqref{Pohoz}, that is:
\[
 \frac{1}{2} \int_{\R^2} (1 - \lvert u_\mathfrak{p} \rvert^2 )^2 \, dx 
 = c_\mathfrak{p} \mathfrak{p} .
\]
At this stage, we only have the rough upper bound $ 0 \ls \frac{1}{4} \int_{\R^2} (1 - \lvert u_\mathfrak{p} \rvert^2 )^2 \, dx 
\ls E ( u_\mathfrak{p} ) \ls  2\pi \ln \mathfrak{p} + \Lambda_0 $, which concludes. 

Another argument we could use for minimizers is that we know from \cite{Bet_Gra_Saut} (see also \cite{Chi_Mar}) 
that $ 0 \ls c_\mathfrak{p} \ls d^+ E_{\rm min} (\mathfrak{p} ) \ls E_{\rm min} (\mathfrak{p} ) / \mathfrak{p} $.

\bigbreak

\noindent {\it Step 2: there exists $ \mathfrak{p}_2 > \mathfrak{p}_1 $, $ R_* \gs 1/ 8 $ and $ n_{*} \in \N $, 
depending only on $ \Lambda_0 $, such that, if $ \mathfrak{p} > \mathfrak{p}_2 $, there exist $ n_\mathfrak{p} $ points 
$ \hat{z}_{ \mathfrak{p} , j } $, $ 1 \ls j \ls n_\mathfrak{p} $ with $ n_\mathfrak{p} \ls n_* $ such that 
$ \{ \lvert\hat{ u}_\mathfrak{p}( \hat{x} ) \rvert \ls 1/2 \} \subset \cup_{j=1}^{n_\mathfrak{p}} B( \hat{z}_{\mathfrak{p} ,j} , R_* ) $ 
and the disks $ \bar{B} ( \hat{z}_{\mathfrak{p}, j} , 4 R_* ) $, $ 1 \ls j \ls n_\mathfrak{p} $, are mutually disjoint.}

We apply Theorem \ref{clearingout} with $ \ep = 1/ \mathfrak{p} $, $ \mathfrak{c} = c_\mathfrak{p} \mathfrak{p} $ 
and $ \sigma = 1/2 $ to $ \hat{u}_\mathfrak{p} $. 
This is possible in view of the upper bound on $ 0 \ls c_\mathfrak{p} \mathfrak{p} \ls 13 \ln \mathfrak{p} $ of Step 1 
(that is $ M_0 = 13 $). We then let $ R_0 \assign 1 / (1+ 13 ) = 1/14 $ for $ \mathfrak{p} \gs \mathfrak{p}_1 $ 
and denote $ \eta_{1/2} $ the positive constant $ \eta $ given by Theorem \ref{clearingout}.

We now proceed in this way: we choose (if it exists) some $ \hat{z}_{\mathfrak{p},1} \in \R^2 $ such that 
$ \lvert \hat{u}_\mathfrak{p} ( \hat{z}_{\mathfrak{p} ,1} ) \rvert < 1/2 $. If $ \{ \lvert \hat{u}_\mathfrak{p} \rvert \ls 1/2 \} \subset 
\bar{B} ( \hat{z}_{ \mathfrak{p} ,1}, 2 R_0 ) $, then we stop. If not, we choose 
$ \hat{z}_{\mathfrak{p},2} \in \R^2 \setminus \bar{B} ( \hat{z}_{\mathfrak{p}, 1} , 2 R_0 )  $ 
such that $ \lvert \hat{u}_\mathfrak{p} ( \hat{z}^{\mathfrak{p}, 2 } ) \rvert < 1/2 $. 
If $ \{ \lvert \hat{u}_\mathfrak{p} \rvert \ls 1/2 \} \subset \cup_{j=1}^{2} \bar{B} ( \hat{z}_{\mathfrak{p}, j} , 2 R_0 ) $, 
then we stop, if not, we continue. This process ends in a finite number of steps (depending only on $K_0 $) since, 
by construction, the disks $ \bar{B} ( \hat{z}_{\mathfrak{p}, j} , R_0 ) $, $ 1 \ls j \ls n $, are 
pairwise disjoint, hence, by Theorem \ref{clearingout}, we have
\begin{align*}
 2 \pi \ln \mathfrak{p} + K_0 
 \gs E( u_\mathfrak{p} ) 
 = E_{1/ \mathfrak{p} } ( \hat{u}_\mathfrak{p} ) 
 & 
 \gs \sum_{j=1}^n E_{1/ \mathfrak{p} } ( \hat{u}_\mathfrak{p} , B(\hat{z}_{\mathfrak{p}, j} , R_0 ) ) 
 \gs n \times \eta_{1/2} \ln \mathfrak{p} ,
\end{align*}
which implies
\[
 n \ls \frac{ 2 \pi \ln \mathfrak{p} + K_0 }{ \eta_{1/2} \ln \mathfrak{p} } 
 \ls \frac{ 7 }{ \eta_{1/2} } 
\]
for $ \mathfrak{p} $ large enough, say $ \mathfrak{p} \gs \mathfrak{p}_2 $. 

At this stage, the disks $ B (\hat{z}_{\mathfrak{p}, j} , 2 R_0 ) $, $1 \ls j \ls n_\mathfrak{p} $, cover the vorticity set 
$ \{ \lvert \hat{u}_\mathfrak{p} \rvert \ls 1/2 \} $, but the disks $ \bar{B} (\hat{z}_{\mathfrak{p} ,j} , 8 R_0 ) $ may not 
be pairwise disjoint. To get this property, we argue as in \cite{BBH} (Theorem IV.1). Let us recall the 
idea: if the disks $ \bar{B} (\hat{z}_{\mathfrak{p}, j} , 8 R_0 ) $, $1 \ls j \ls n_\mathfrak{p} $ are pairwise disjoint, then 
we are done with $ R_* = 2 R_0 $. If not, then we have, for instance, 
$ \lvert \hat{z}_{\mathfrak{p},1} - \hat{z}_{\mathfrak{p},2} \rvert \ls 16 R_0 $. We then remove the disk 
$ B (\hat{z}_{\mathfrak{p},1} , 8 R_0 ) $ from the list and set $ R_1 \ddef 17 R_0 $. The disks 
$ B (\hat{z}_{\mathfrak{p},j} , R_1 ) $, $2 \ls j \ls n_\mathfrak{p} $ cover 
$ \cup_{1 \ls j \ls n_\mathfrak{p} } B (\hat{z}_{\mathfrak{p}, j} , 2 R_0 ) $, hence the vorticity set 
$ \{ \lvert \hat{u}_\mathfrak{p} \rvert \ls 1/2 \} $, and their number has decreased. In a finite number of steps 
(depending only on $K_0 $), we obtain the conclusion. The radius $ R_* $ is necessarily 
$ \ls R_0 \times 17^{n_\mathfrak{p}} \ls R_0 \times17^{n_* } $.

Similar arguments are given in \cite{BOS}, whereas in \cite{Bet_Saut} the vorticity 
set is included in some disks of radii of order $ c_\mathfrak{p}^\gamma $, which requires some extra work.

\bigbreak

\noindent {\it Step 3: we have}
\[
  \mathfrak{p}^2 \int_{\R^2 } ( 1 - \lvert \hat{u}_\mathfrak{p} \rvert^2 )^2 \, d\hat{x} 
 = o_{ \mathfrak{p} \to+\ii } (\ln \mathfrak{p} ) .
\]
This follows exactly as in \cite{BOS} (see Proposition A.1 in the Appendix there). Notice that 
the result in \cite{BOS} is stated for the potential on a compact set in a domain $ \Omega $, 
but it holds as well in the entire plane.

\bigbreak

We then define, as in \cite{BOS}, the function $ \hat{u}'_\mathfrak{p} : \R^2 \to \C $ by
\[
 \hat{u}'_\mathfrak{p} (\hat{x} ) \ddef 
 \left\{\begin{array}{ll}
 \hat{u}_\mathfrak{p} ( \hat{x} ) & \textrm{if } \hat{x} \in \cup_{j=1}^{n_\mathfrak{p}} \bar{B} ( \hat{z}_{\mathfrak{p}, j} , 2 R_* )
 \\ 
 \frac{ \hat{u}_\mathfrak{p} ( \hat{x} ) }{ \lvert \hat{u}_\mathfrak{p} (\hat{x}) \rvert } 
 & \textrm{if } \hat{x} \not\in \cup_{j=1}^{n_\mathfrak{p}} \bar{B} ( \hat{z}_{\mathfrak{p}, j} , 3 R_*)
 \\ 
 ( 3 - \lvert \hat{x} - \hat{z}_{\mathfrak{p}, j} \rvert / R_* ) \hat{u}_\mathfrak{p} ( \hat{x} ) 
 + ( -2 + \lvert \hat{x} - \hat{z}_{\mathfrak{p},j} \rvert / R_* ) \frac{ \hat{u}_\mathfrak{p} ( \hat{x} ) }{ \lvert \hat{u}_\mathfrak{p} (\hat{x} )\rvert} 
 & \textrm{if } \hat{x} \in \bar{B} ( \hat{z}_{\mathfrak{p}, j} , 3 R_* ) \setminus \bar{B} ( \hat{z}_{\mathfrak{p}, j} , 2 R_* )
 \end{array}
 \right.
\]
for some $1 \ls j \ls n_\mathfrak{p} $ (this last formula is valid since the disks $ \bar{B} ( \hat{z}_{\mathfrak{p}, j} , 4 R_* ) $, 
$ 1 \ls j \ls n_\mathfrak{p} $, are mutually disjoint).

\bigbreak

\noindent {\it Step 4: we have, as $ \mathfrak{p} \to +\ii $,
\[
 E_{1/ \mathfrak{p} } ( \hat{u}'_\mathfrak{p} ) 
 \ls 2 \pi \ln \mathfrak{p} + o ( \ln \mathfrak{p} ) .
\]
}

Letting $ \Omega_{R} \ddef \cup_{j=1}^{n_\mathfrak{p}} \bar{B} ( \hat{z}_{\mathfrak{p}, j} , R ) $, we have
\begin{align*}
  \int_{\R^2} (1 - \lvert \hat{u}'_\mathfrak{p} \rvert^2 )^2 \, d\hat{x} 
 & = \int_{\Omega_{2R_*} } (1 - \lvert \hat{u}_\mathfrak{p} \rvert^2 )^2 \, d\hat{x} 
  + \int_{\Omega_{3R_*} \setminus \Omega_{2R_*} } (1 - \lvert \hat{u}'_\mathfrak{p} \rvert^2 )^2 \, d\hat{x} 
\end{align*}
We notice that in $ \Omega_{3 R_*} \setminus \Omega_{2R_*} $, say for 
$ \hat{x} \in \bar{B} ( \hat{z}_{\mathfrak{p},j} , 3 R_* ) \setminus \bar{B} ( \hat{z}_{\mathfrak{p}, j} , 2 R_* ) $, we have 
\[
 \lvert \hat{u}'_\mathfrak{p} ( \hat{x} ) \rvert 
 = ( 3 - \lvert \hat{x} - \hat{z}_{\mathfrak{p}, j} \rvert / R_* ) \lvert \hat{u}_\mathfrak{p} (\hat{x} ) \rvert
 + ( -2 + \lvert \hat{x} - \hat{z}_{ \mathfrak{p}, j} \rvert / R_* )  
 \in [ \lvert \hat{u}_\mathfrak{p} (\hat{x} ) \rvert , 1 ] 
 ,
\]
hence 
$ \big\lvert 1 - \lvert \hat{u}'_\mathfrak{p} ( \hat{x} ) \rvert^2 \big\rvert \ls \big\lvert 1 - \lvert \hat{u}_\mathfrak{p} ( \hat{x} ) \rvert^2 \big\rvert $ and thus
\begin{align}
\label{potard_1}
  \int_{\R^2} (1 - \lvert \hat{u}'_\mathfrak{p} \rvert^2 )^2 \, d\hat{x} 
 & \ls \int_{\Omega_{2R_*} } (1 - \lvert \hat{u}_\mathfrak{p} \rvert^2 )^2 \, d\hat{x} 
  + \int_{\Omega_{3R_*} \setminus \Omega_{2R_*} } (1 - \lvert \hat{u}_\mathfrak{p} \rvert^2 )^2 \, d\hat{x} 
  \nonumber \\ & 
  = \int_{\Omega_{ 3R_*} } (1 - \lvert \hat{u}_\mathfrak{p} \rvert^2 )^2 \, d\hat{x} .
\end{align}
For the kinetic term, we have
\[
  \lvert \nabla \hat{u}'_\mathfrak{p} (\hat{x} ) \rvert^2 
 =
 \lvert \nabla \hat{u}_\mathfrak{p} ( \hat{x} ) \rvert^2 
\]
if $ \hat{x} \in \Omega_{2R_*} $. Outside $ \cup_{j=1}^{n_\mathfrak{p}} \bar{B} ( \hat{z}_{\mathfrak{p}, j} , R_* ) $, then 
$ \lvert \hat{u}_\mathfrak{p} \rvert \gs 1/2 $ and we may then lift, at least locally, 
$ \hat{u}_\mathfrak{p} = A \ex^{i \phi } $ and get
\[
 \lvert \nabla \hat{u}_\mathfrak{p} \rvert^2 
 = A^2 \lvert \nabla \phi \rvert^2 + \lvert \nabla A \rvert^2 .
\]
If $ \hat{x} \not\in \Omega_{3 R_*} $, then, by \eqref{LinfGradResc},
\begin{align*}
 \lvert \nabla \hat{u}'_\mathfrak{p} \rvert^2 
 = \lvert \nabla \phi \rvert^2 
 = A^2 \lvert \nabla \phi \rvert^2 + \frac{1 - A^2}{A^2} \times A^2 \lvert \nabla \phi \rvert^2 
 \ls \lvert \nabla \hat{u}_\mathfrak{p} \rvert^2 + 4 K \mathfrak{p} \lvert 1 - A^2 \rvert \times \lvert \nabla \hat{u}_\mathfrak{p} \rvert
\end{align*}
since $ A = \lvert \hat{u}_\mathfrak{p} \rvert \gs 1/2 $ outside $ \Omega_{R_*} $. Finally, in 
$ \bar{B} ( \hat{z}_{\mathfrak{p}, j} , 3 R_* ) \setminus \bar{B} ( \hat{z}_{\mathfrak{p}, j} , 2 R_* ) $ (for some 
unique $1 \ls j \ls n_{\mathfrak{p}} $), we have
\begin{align*}
 \lvert \nabla \hat{u}'_\mathfrak{p} \rvert^2 
 & = \lvert \nabla \phi \rvert^2 \Big( ( 3 - \lvert \hat{x} - \hat{z}_{\mathfrak{p}, j} \rvert / R_* ) A 
 + ( -2 + \lvert \hat{x} - \hat{z}_{\mathfrak{p}, j} \rvert / R_* )  \Big)^2 
 \\& \quad 
 + \Big\lvert \nabla \Big[ ( 3 - \lvert \hat{x} - \hat{z}_{\mathfrak{p},j} \rvert / R_* ) A
 + ( -2 + | \hat{x} - \hat{z}_{\mathfrak{p},j} | / R_* ) \Big] \Big\rvert^2 .
\end{align*}
We then use that, since $ \lvert \hat{u}_\mathfrak{p} (\hat{x} ) \rvert \gs 1/2 $ and letting 
$ \theta = 3 - | \hat{x} - \hat{z}_{\mathfrak{p}, j} | / R_* \in [ 0 ,1] $,
\begin{align*}
 \lvert \nabla \phi \rvert^2 
 \Big[ & ( 3 - | \hat{x} - \hat{z}_{\mathfrak{p}, j} | / R_* ) A + ( -2 + | \hat{x} - \hat{z}_{\mathfrak{p}, j} | / R_* ) \Big]^2
 \\ & 
 = A^2 \lvert \nabla \phi \rvert^2 \times \frac{1}{A^2} [ 1 + \theta (A-1) ]^2 
 \ls A^2 \lvert \nabla \phi \rvert^2 \times ( 1 + K \lvert A^2 -1 \rvert )
 \\ & 
 \ls A^2 \lvert \nabla \phi \rvert^2 + K \mathfrak{p} \lvert \nabla \hat{u}_\mathfrak{p} \rvert \times \lvert A^2 -1 \rvert ,
\end{align*}
by Corollary \ref{Gradinf}. On the other hand, since $ \lvert \cdot \rvert $ is $1$-Lipschitz continuous,
\begin{align*}
\Big\lvert & \nabla \Big[ ( 3 - | \hat{x} - \hat{z}_{\mathfrak{p},j} | / R_* ) A 
 + ( -2 + | \hat{x} - \hat{z}_{\mathfrak{p}, j} | / R_* ) \Big] \Big\rvert^2
 \\ & 
 \ls \frac{1}{R_*^2 } \lvert 1- A \rvert^2 
 + \lvert \nabla A \rvert^2 
 + \frac{2}{R_* } \lvert 1- A \rvert \times \lvert \nabla A \rvert 
 \\ & 
 \ls 
 \lvert \nabla A \rvert^2 
 + K  ( A^2 -1 )^2 
 + K \lvert \nabla A \rvert \times \lvert A^2 -1 \rvert .
\end{align*}
Therefore, by Cauchy-Schwarz inequality, for some absolute constant $K > 0$,
\begin{align*}
 \int_{\R^2} \lvert \nabla \hat{u}'_\mathfrak{p} \rvert^2 \, d\hat{x} 
 & \ls 
 \int_{\R^2 } \lvert \nabla \hat{u}_\mathfrak{p} \rvert^2 \, d\hat{x} 
 + K \Big( \int_{\R^2} \mathfrak{p}^2 ( 1- \lvert \hat{u}_\mathfrak{p} \rvert^2 )^2  \, d\hat{x} \Big)^{1/2} 
 	\Big( \int_{\R^2} \lvert \nabla \hat{u}_\mathfrak{p} \rvert^2 \, d\hat{x}  \Big)^{1/2} 
	+ K \int_{\R^2} ( 1- \lvert \hat{u}_\mathfrak{p} \rvert^2 )^2  \, d\hat{x} .
\end{align*}
Combining this with \eqref{potard_1} yields 
\begin{align*}
 E_{ 1/ \mathfrak{p}} (\hat{u}'_\mathfrak{p} ) 
 \ls E_{ \mathfrak{p}} (\hat{u}_\mathfrak{p} ) + K \sqrt{ E_{ \mathfrak{p}} (\hat{u}_\mathfrak{p} )} 
 \Big( \int_{\R^2} \mathfrak{p}^2 ( 1- \lvert \hat{u}_\mathfrak{p} \rvert^2 )^2  \, d\hat{x}  \Big)^{1/2} 
 + K \frac{E_{ \mathfrak{p}} (\hat{u}_\mathfrak{p} ) }{\mathfrak{p}^2}
 \ls 2 \pi \ln \mathfrak{p} + o (\ln \mathfrak{p} ) ,
\end{align*}
by the upper bound \eqref{upper_bound} and the estimate for the potential term of Step 3. 

\bigbreak

\noindent {\it Step 5: we claim that for any $ \delta \in ] 0 , \pi /2 [ $, there exist $ \mathfrak{p}^\dag_\delta > \mathfrak{p}_2 $ 
such that for all $ \mathfrak{p} \geqslant \mathfrak{p}^\dag_\delta $, we are in one of the following cases:

\noindent 
case (I) for any $ 1 \ls j \ls n_\mathfrak{p} $,
\[
 \| J \hat{u}'_\mathfrak{p} \|_{ [ \BC^{0,1}_c( B( \hat{z}_{\mathfrak{p}, j} , 4R_*) ) ]^* } 
 \ls \delta
\]

\noindent 
case (II) there exists (up to a relabelling) two points $ \hat{y}_{\mathfrak{p}, \pm } \in \R^2 $, 
depending on $ \hat{u}_\mathfrak{p} $, such that 
\[
 \max_{1 \ls j \ls n_\mathfrak{p} } 
 \Big\| J \hat{u}'_\mathfrak{p} - \pi ( \delta_{\hat{y}_{\mathfrak{p}, + } } - \delta_{\hat{y}_{\mathfrak{p}, - } }) 
 \Big\|_{ [ \BC^{0,1}_c( B( \hat{z}_{\mathfrak{p}, j} , 4R_*) ) ]^* } 
 \ls \delta
\]
}

We apply Theorem \ref{Jacob} to $ \hat{u}_\mathfrak{p}' $ on each disk 
$ B ( \hat{z}_{\mathfrak{p}, j} , 4 R_* ) $, $ 1 \ls j \ls n_\mathfrak{p} $. This yields points 
$ \hat{y}_{\mathfrak{p}, j,k} \in \{ \lvert \hat{u}_\mathfrak{p} \rvert \ls 1/2 \} \subset B ( \hat{z}_{\mathfrak{p}, j} , R_* ) \subset B ( \hat{z}_{\mathfrak{p}, j} , 4R_* ) $ 
and integers 
$ d_{\mathfrak{p}, j,k} \in \Z $, $ 1 \ls k \ls N_{\mathfrak{p}, j} $ such that
\begin{equation}
\label{step4-1}
 \Big\| J \hat{u}'_\mathfrak{p} - \pi \sum_{k=1}^{N_{\mathfrak{p}, j}} 
 d_{\mathfrak{p}, j,k} \delta_{\hat{y}_{\mathfrak{p}, j, k}} \Big\|_{ [ \BC^{0,1}_c( B ( \hat{z}_{\mathfrak{p}, j} , 4R_* ) ) ]^*} 
 \ls \delta 
\end{equation}
and
\begin{equation}
\label{step4-2}
 \pi \sum_{k=1}^{N_{ \mathfrak{p} , j }} 
 \lvert d_{\mathfrak{p}, j,k} \rvert 
 \ls \frac{E_{1/\mathfrak{p}} (\hat{u}'_\mathfrak{p}, B (\hat{z}_{\mathfrak{p},j} , 4R_* )) }{ \ln \mathfrak{p} } + \delta .
\end{equation}
By summing over $ 1 \ls j \ls n_\mathfrak{p} $ the inequalities \eqref{step4-2}, we infer
\[
 \pi \sum_{j=1}^{n_\mathfrak{p}} \sum_{k=1}^{N_{ \mathfrak{p} , j}} 
 \lvert d_{\mathfrak{p}, j,k} \rvert 
 \ls \frac{E_{1/\mathfrak{p}} ( \hat{u}'_\mathfrak{p} , \Omega_{4R_*} ) }{ \ln \mathfrak{p} } + \delta 
 \ls 2.5 \pi 
\]
by using $ \delta < \pi /2 $ and Step 3, and for $ \mathfrak{p} $ large enough. 
Therefore, 
\begin{equation}
\label{Step4-0}
 \sum_{j=1}^{n_\mathfrak{p}} \sum_{k=1}^{N_{ \mathfrak{p} , j}} 
 \lvert d_{\mathfrak{p}, j,k} \rvert 
 \ls 2
\end{equation}
and two cases may occur: all the integers $ d_{\mathfrak{p}, j,k} $ are zero (this is Case (I)) 
or at least one of the integers $ d_{\mathfrak{p}, j,k} $ is not zero. 

In addition, we have, for $ 1 \ls j \ls n_\mathfrak{p} $,
\begin{equation}
\label{degreoudeforce}
 \sum_{k=1}^{N_{ \mathfrak{p}, j } } d_{\mathfrak{p}, j, k} 
= {\rm deg} ( \hat{u}_\mathfrak{p} , \p B ( \hat{z}_{\mathfrak{p} , j} , 3 R_*) )  .
\end{equation}
Indeed, since $ \lvert \hat{u}_\mathfrak{p}' \rvert = 1 $ on $ B ( \hat{z}_{\mathfrak{p} , j} , 4 R_*) 
\setminus B ( \hat{z}_{\mathfrak{p} , j} , 3 R_*) $, we have $ J\hat{u}_\mathfrak{p}' = 0 $ there. 
Therefore, by fixing $ \chi \in \BC^\ii_c ( B (0, 4 R_*) ) $ such that $ \chi \equiv 1 $ on $ \bar{B} (0,3 R_*) $, 
we deduce
\begin{align*}
 \Big\lvert 
 \sum_{k=1}^{N_{ \mathfrak{p}, j } } d_{\mathfrak{p}, j, k} 
 - {\rm deg} ( \hat{u}_\mathfrak{p} , \p B ( \hat{z}_{\mathfrak{p} , j} , 3 R_*) ) \Big\rvert 
 & = \Big\lvert 
 \int_{ B (\hat{z}_{ \mathfrak{p}, j}, 3 R_* )} \sum_{k=1}^{N_{ \mathfrak{p}, j } } 
 d_{\mathfrak{p}, j, k} \delta_{\hat{y}_{\mathfrak{p}, j, k} } \, d \hat{x} 
 - \frac{1}{\pi} \int_{ B (\hat{z}_{\mathfrak{p}, j}, 4 R_* )} J \hat{u}'_\mathfrak{p} \, d \hat{x} 
 \Big\rvert 
 \\ &
 = \frac{1}{\pi} \Big\lvert 
 \int_{ B (\hat{z}_\mathfrak{p}^j, 4 R_* )} \chi ( \hat{x} - \hat{z}_{\mathfrak{p}, j} ) 
 \Big( \sum_{k=1}^{N_{ \mathfrak{p}, j } } d_{\mathfrak{p}, j, k} \delta_{\hat{y}_{\mathfrak{p}, j, k} } - J \hat{u}'_\mathfrak{p} 
 \Big) \, d \hat{x} \Big\rvert
 \\ & 
 \ls \frac{1}{\pi} \| \chi \| \times 
 \Big\| J \hat{u}'_\mathfrak{p} - \pi \sum_{k=1}^{N_{\mathfrak{p}, j}} 
 d_{\mathfrak{p}, j,k} \delta_{\hat{y}_{\mathfrak{p}, j, k}} \Big\|_{ [ \BC^{0,1}_c( \bar{D} ( \hat{z}_{\mathfrak{p}, j} , 4R_* ) ) ]^*} 
 \end{align*}
by \eqref{step4-1}.  
Since the left-hand side is an integer and the right-hand side is $ \ls 1/2 $ provided 
$ \mathfrak{p} \gs \mathfrak{p}_{2,1}(\delta, \Lambda_0) $, \eqref{degreoudeforce} follows. 

We finally notice that the degree of $ \hat{u}'_\mathfrak{p} $ on some large circle $ \p B (0, R ) $ 
(with $ R \gg \max_{1 \ls j \ls n_\mathfrak{p} } \lvert \hat{z}_{\mathfrak{p}, j} \rvert $) is zero, for otherwise 
$ \hat{u}'_\mathfrak{p} $ (and $ \hat{u}_\mathfrak{p} $) would have infinite kinetic energy. Therefore,
\[
 0 = 
 \sum_{ j= 1}^{ n_\mathfrak{p}} {\rm deg} ( \hat{u}_\mathfrak{p} , \p B ( \hat{z}_{\mathfrak{p} , j} , 3 R_*) ) 
 = \sum_{ j= 1}^{ n_\mathfrak{p}} \sum_{k=1}^{N_{ \mathfrak{p}, j } } d_{\mathfrak{p}, j, k} .
\]
Combining this with \eqref{Step4-0}, we deduce that if we are not in Case (I), then one of the 
$ d_{\mathfrak{p}, j,k} $ must be equal to $+1 $ and another one must be equal to $ -1$, which is Case (II). 

Notice that for Case (II), if $ B ( \hat{z}_{\mathfrak{p},j} , 4 R_* ) $ contains neither $ y_{\mathfrak{p}, + } $ 
nor $ y_{\mathfrak{p}, - } $, then 
$ \| J \hat{u}'_\mathfrak{p} \|_{ [ \BC^{0,1}_c( B( \hat{z}_{\mathfrak{p},j} , 4R_*) ) ]^* } \ls \delta $.

\bigbreak

As in \cite{Bet_Saut}, we now relate the location of the points $ \hat{y}_{\mathfrak{p}, \pm } $ to the 
momentum $ P( \hat{u}_\mathfrak{p} ) $. 

\bigbreak

\noindent {\it Step 6: Case (I) does not occur for $ \mathfrak{p} $ sufficiently large, say $ \mathfrak{p} \gs \mathfrak{p}_3 $. 
In addition, we have
\[
 1 = P ( \hat{u}_\mathfrak{p} ) 
 = \pi \big( ( \hat{y}_{\mathfrak{p}, + } )_1 - ( \hat{y}_{\mathfrak{p}, -} )_1 \big) + o(1) .
\]
}

First, we have, by computations similar to those of Step 3, $ \hat{u}_\mathfrak{p} = A \ex^{i\vp} $ 
locally outside $ \Omega_{R_*} $, hence 
$ \langle i \hat{u}_\mathfrak{p} | \nabla \hat{u}_\mathfrak{p} \rangle = A^2 \nabla \vp $ and then, 
outside $ \Omega_{3 R_*} $,
\[
 \langle i \hat{u}_\mathfrak{p} | \nabla \hat{u}_\mathfrak{p} \rangle 
 - \langle i \hat{u}'_\mathfrak{p} | \nabla \hat{u}'_\mathfrak{p} \rangle 
 = A^2 \nabla \vp - \nabla \vp 
 = \frac{A^2 -1}{A} \times A \nabla \vp .
\]
In $ B (\hat{z}_{\mathfrak{p}, j} , 3R_* ) \setminus B (\hat{z}_{\mathfrak{p}, j} , 2R_* ) $, we obtain
\[
 \lvert \langle i \hat{u}_\mathfrak{p} | \nabla \hat{u}_\mathfrak{p} \rangle 
 - \langle i \hat{u}'_\mathfrak{p} | \nabla \hat{u}'_\mathfrak{p} \rangle \rvert
 = \lvert A^2 \nabla \vp - \lvert \hat{u}'_\mathfrak{p} \rvert^2 \nabla \vp \rvert
 \ls \frac{\lvert A^2 -1 \rvert}{A} \times \lvert A \nabla \vp \rvert ,
\]
since $ \lvert \hat{u}'_\mathfrak{p} \rvert \in [ \lvert \hat{u}_\mathfrak{p} \rvert , 1 ] $. Therefore,
\begin{equation}
\label{pre-jaco} 
\| \langle i \hat{u}_\mathfrak{p} | \nabla \hat{u}_\mathfrak{p} \rangle 
 - \langle i \hat{u}'_\mathfrak{p} | \nabla \hat{u}'_\mathfrak{p} \rangle \|_{L^1(\R^2)} 
 \ls K \int_{\R^2 \setminus \Omega_{2 R_*} } \big\lvert 1 - \lvert \hat{u}_\mathfrak{p} \rvert^2 \big\rvert 
 \times \lvert \nabla \hat{u}_\mathfrak{p} \rvert \, d \hat{x}
 \ls \frac{K}{ \mathfrak{p} } E_{1/\mathfrak{p}} ( \hat{u}_\mathfrak{p} ) 
 \ls K \frac{\ln \mathfrak{p} }{ \mathfrak{p} } .
\end{equation}
Following \cite{Bet_Saut}, \cite{BOS}, we write
\begin{align*}
 1 = \frac{ P ( u_\mathfrak{p} ) }{ \mathfrak{p} }
 = P ( \hat{u}_\mathfrak{p} ) 
 & = \frac{1}{2} \int_{\R^2} \langle i \p_2 \hat{u}_\mathfrak{p} | \hat{u}_\mathfrak{p} - 1 \rangle \, d \hat{x} 
 \\ 
 & = \frac{1}{2} \int_{\R^2} \langle i \p_2 \hat{u}'_\mathfrak{p} | \hat{u}'_\mathfrak{p} - 1 \rangle \, d \hat{x} 
 + \frac{1}{2} \int_{\R^2} ( \langle i \p_2 \hat{u}_\mathfrak{p} | \hat{u}_\mathfrak{p} - 1 \rangle 
 - \langle i \p_2 \hat{u}'_\mathfrak{p} | \hat{u}'_\mathfrak{p} - 1 \rangle ) \, d \hat{x} 
 .
\end{align*}
For the second integral, we write that, on the one hand,
\begin{align*}
 \Big\lvert \int_{\R^2} ( \langle i \hat{u}_\mathfrak{p} | \p_2 \hat{u}_\mathfrak{p} \rangle 
 - \langle i \hat{u}'_\mathfrak{p} | \p_2 \hat{u}'_\mathfrak{p} \rangle ) \, d \hat{x} \Big\rvert 
 \ls \| \langle i \hat{u}_\mathfrak{p} | \nabla \hat{u}_\mathfrak{p} \rangle 
 - \langle i \hat{u}'_\mathfrak{p} | \nabla \hat{u}'_\mathfrak{p} \rangle \|_{L^1(\R^2)} 
 \ls K \frac{\ln \mathfrak{p} }{ \mathfrak{p} } \to 0 
\end{align*}
when $ \mathfrak{p} \to +\ii $; on the other hand, by the decays given in Theorem \ref{decay},
\begin{align*}
 \Big\lvert \int_{\R^2} ( \langle i \p_2 \hat{u}_\mathfrak{p} | 1 \rangle 
 - \langle i \p_2 \hat{u}'_\mathfrak{p} | 1 \rangle ) \, d \hat{x} \Big\rvert 
 & = \lim_{ r \to +\ii} \Big\lvert \int_{\p B(0,r)} \nu_2 \mathfrak{Im} ( \hat{u}_\mathfrak{p} - \hat{u}'_\mathfrak{p}) \, d \ell \Big\rvert 
 \\ & 
 \ls \lim_{ r \to +\ii} \int_{\p B(0,r)} \lvert A -1 \rvert \, d \ell 
 = \lim_{ r \to +\ii} \mathcal{O}(1/r) 
 = 0 .
\end{align*}
We then integrate by parts to get
\begin{align*}
  \frac{1}{2} \int_{\R^2} \langle i \p_2 \hat{u}'_\mathfrak{p} | \hat{u}'_\mathfrak{p} - 1 \rangle \, d \hat{x} 
  & 
 = \frac{1}{2} \int_{\R^2} \p_1 \hat{x}_1 \langle i \p_2 \hat{u}'_\mathfrak{p} | \hat{u}'_\mathfrak{p} - 1 \rangle 
  - \p_2 \hat{x}_1 \langle i \p_1 \hat{u}'_\mathfrak{p} | \hat{u}'_\mathfrak{p} - 1 \rangle \, d \hat{x} 
  \\ & 
  = \int_{\R^2} J \hat{u}'_\mathfrak{p} \hat{x}_1 \, d \hat{x} 
  .
\end{align*}
The integration by parts is justified by the algebraic decay at infinity given in Theorem \ref{decay}: 
$ \hat{x}_1 \langle i \p_2 \hat{u}'_\mathfrak{p} | \hat{u}'_\mathfrak{p} - 1 \rangle 
 = \mathcal{O} ( 1 / |x|^{2} ) $. 
 
Then, since $ J \hat{u}'_\mathfrak{p} $ is supported in $ \Omega_{R_*} $, we obtain
\begin{align*}
 \int_{\R^2} \hat{x}_1 J \hat{u}'_\mathfrak{p} \, d \hat{x}
 & = \sum_{j=1}^{n_\mathfrak{p}} \int_{ B (\hat{z}_{\mathfrak{p}, j} , 3 R_* )} \hat{x}_1 J \hat{u}'_\mathfrak{p} \, d \hat{x} 
 \\ & 
 = \sum_{j=1}^{n_\mathfrak{p}} \int_{ B (\hat{z}_{\mathfrak{p} , j }, 3 R_* )} ( \hat{x}_1 - (\hat{z}_{\mathfrak{p}, j})_1 ) J \hat{u}'_\mathfrak{p} \, d \hat{x} 
 + \sum_{j=1}^{n_\mathfrak{p}} \hat{z}_{\mathfrak{p}, j,1} \int_{ B (\hat{z}_{\mathfrak{p},j} , 3 R_* )} J \hat{u}'_\mathfrak{p} \, d \hat{x} .
\end{align*}
We then fix $ \chi \in \BC^\ii_c ( B( 0, 4R_*) ) $ such that $ \chi \equiv 1 $ on $ \bar{B} (0,3 R_*) $. 
Next, for any $ 1 \ls j \ls n_\mathfrak{p} $, we write,
\begin{align*}
 \int_{ B (\hat{z}_{\mathfrak{p} , j }, 3 R_* )} ( \hat{x}_1 - ( \hat{z}_{\mathfrak{p},j} )_1 ) J \hat{u}'_\mathfrak{p} \, d \hat{x} & 
 = \int_{ B (\hat{z}_{\mathfrak{p} ,j} , 4 R_* )} ( \hat{x}_1 - (\hat{z}_{\mathfrak{p},j})_1 ) \chi( \hat{x} - \hat{z}_{\mathfrak{p}, j} ) 
 J \hat{u}'_\mathfrak{p} \, d \hat{x} 
  \\ & 
 = \int_{ B (\hat{z}_{\mathfrak{p}, j }, 4 R_* )} ( \hat{x}_1 - ( \hat{z}_{\mathfrak{p},j} )_1 ) \chi( \hat{x} - \hat{z}_{\mathfrak{p}, j} ) 
 \Big( J \hat{u}'_\mathfrak{p} - \pi \sum_{k=1}^{N_{ \mathfrak{p} , j }} d_{\mathfrak{p}, j, k } \delta_{y_{\mathfrak{p}, j,k}} \Big) \, d \hat{x} 
 \\ & \quad
 + \pi \sum_{k=1}^{N_{ \mathfrak{p} , j }} d_{\mathfrak{p}, j, k } \big( ( y_{\mathfrak{p}, j,k} )_1 - ( \hat{z}_{\mathfrak{p},j} )_1 \big) .
\end{align*}
We now estimate the first integral (actually, a duality bracket) by using Step 5:
\begin{align*}
 & \Big\lvert \int_{ B (\hat{z}_{\mathfrak{p}, j} , 2 R_* )} ( \hat{x}_1 - ( \hat{z}_{\mathfrak{p},j} )_1 ) \chi( \cdot - \hat{z}_{\mathfrak{p}, j} ) 
 \Big( J \hat{u}'_\mathfrak{p} - \pi \sum_{k=1}^{N_{ \mathfrak{p} , j }} d_{\mathfrak{p}, j, k} \delta_{y_{\mathfrak{p}, j,k}} \Big) \, d \hat{x} \Big\rvert 
 \\ & 
 \ls \| ( \hat{x}_1 - ( \hat{z}_{\mathfrak{p}, j} )_1 ) \chi( \cdot - \hat{z}_{\mathfrak{p}, j} ) \|_{\BC^{0,1}_c( B ( \hat{z}_{\mathfrak{p}, j} , 2R_* ) )} 
 \Big\| J \hat{u}'_\mathfrak{p} - \pi \sum_{k=1}^{N_{ \mathfrak{p} ,j}} 
 d_{\mathfrak{p}, j,k} \delta_{y_{\mathfrak{p}, j,k}} \Big\|_{ [ \BC^{0,1}_c( B ( \hat{z}_{\mathfrak{p},j} , 2R_* ) ) ]^*} 
 \\ & 
 \ls K o(1) .
\end{align*}
As a consequence of \eqref{degreoudeforce}, which implies, for each $1 \ls j \ls n_\mathfrak{p} $,
\[
 \sum_{k=1}^{N_{\mathfrak{p},j}} d_{\mathfrak{p},j,k} 
 = {\rm deg} ( \hat{u}_\mathfrak{p} , \p B ( \hat{z}_{\mathfrak{p},j}, 3 R_* ) )
 = {\rm deg} ( \hat{u}'_\mathfrak{p} , \p B ( \hat{z}_{\mathfrak{p},j}, 3 R_* ) )
 = \int_{ B ( \hat{z}_{\mathfrak{p},j}, 3 R_* ) } J \hat{u}'_\mathfrak{p} \, d \hat{x} ,
\]
we infer, after some cancellation,
\begin{align}
\label{Step5-1}
 \Big\lvert 
 P( \hat{u}_\mathfrak{p} ) - \pi \sum_{j=1}^{n_\mathfrak{p}} \sum_{k=1}^{N_{ \mathfrak{p} , j }} d_{\mathfrak{p}, j, k } ( y_{\mathfrak{p}, j,k} )_1  \Big\rvert 
  \ls K \frac{\ln \mathfrak{p} }{\mathfrak{p}} 
  + n_* K o(1) .
\end{align}
Since $ P( \hat{u}_\mathfrak{p} ) = 1 $, it follows that for $ \mathfrak{p} $ large enough, we can not 
be in Case (I), and the conclusion is a recasting of \eqref{Step5-1}.

\bigbreak

\noindent {\it Step 7: there exists $ \mathfrak{p}_4 $ large such that, for $ \mathfrak{p} \gs \mathfrak{p}_4 $, we have 
$ \{ \lvert \hat{u}_\mathfrak{p} \rvert \ls 1/2 \} \subset B ( \hat{y}_{\mathfrak{p} ,+} , 3/20 ) \cup B ( \hat{y}_{\mathfrak{p} ,-} , 3/20 ) $ 
and $ \textrm{deg} ( u , \p B ( \hat{y}_{\mathfrak{p} , \pm} , 3/20 ) ) = \pm 1 $.
}

From Step 6, we know that $ 1 = P ( \hat{u}_\mathfrak{p} ) 
= \pi ( ( \hat{y}_{\mathfrak{p}, + } )_1 - ( \hat{y}_{\mathfrak{p}, -} )_1 ) + o(1) $, hence the two points 
$ \hat{y}_{\mathfrak{p}, \pm } $ are far away from each other : 
\[
 \lvert \hat{y}_{\mathfrak{p}, + } - \hat{y}_{\mathfrak{p}, -} \rvert \gs 4/ 10 
\]
(since $ 1/ \pi \approx 0.318 < 4 /10 $) for $\mathfrak{p}$ large enough (but they may be, at this stage, 
very far away from each other). By applying Theorem 1.1 $(i)$ of \cite{Alb_Bal_Orl} or 
Theorem 3.1 of \cite{Jer_Son} (this is not very far from Theorem \ref{Jacob}), since 
$ J \hat{u}_\mathfrak{p} (  \hat{y}_{\mathfrak{p} ,\pm} + \cdot ) \to \pm \pi \delta_0 $ weakly, we deduce
\[
 E_{1/\mathfrak{p}} ( \hat{u}_\mathfrak{p} , B ( \hat{y}_{\mathfrak{p} , \pm} , 1/10 ) ) 
 \gs ( \pi + o(1)) \ln \mathfrak{p} ,
\]
hence, by the upper bound \eqref{upper_bound},
\[
 E_{1/\mathfrak{p}} ( \hat{u}_\mathfrak{p} , \R^2 \setminus ( B ( \hat{y}_{\mathfrak{p} ,+} , 1/10 ) 
 \cup B ( \hat{y}_{\mathfrak{p} , -} , 1/10 ) ) ) 
 \ls o( \ln \mathfrak{p} ) ,
\]
and this in turn implies, by the clearing-out theorem (Theorem \ref{clearingout}), that if 
$ \mathfrak{p} $ is large enough, say $ \mathfrak{p} \gs \mathfrak{p}_4 $, then
\[
 \forall \hat{x} \in \R^2 \setminus ( B ( \hat{y}_{\mathfrak{p},+} , 3/20 ) \cup B ( \hat{y}_{\mathfrak{p},-} , 3/20 ) ) , 
 \quad \quad \quad
 \lvert \hat{u}_\mathfrak{p} ( \hat{x}) \rvert \gs 3/4 ,
\]
as wished. In particular, 
$ \hat{z}_{ \mathfrak{p}, \pm} \in B ( \hat{y}_{\mathfrak{p},+} , 3/20 ) \cup B ( \hat{y}_{\mathfrak{p},-} , 3/20 ) $.

\bigbreak

We emphasize that at this stage, we have $ \lvert \hat{y}_{\mathfrak{p}, + } - \hat{y}_{\mathfrak{p}, -} \rvert \gtrsim 1 $, 
but we do not know whether $ \lvert \hat{y}_{\mathfrak{p}, + } - \hat{y}_{\mathfrak{p}, -} \rvert \lesssim 1 $ 
or $ \lvert \hat{y}_{\mathfrak{p}, + } - \hat{y}_{\mathfrak{p}, -} \rvert \gg 1 $. We may now take advantage 
of the fact that $ \hat{u}_{\mathfrak{p} } $ is by hypothesis symmetric with respect to the $ x_2 $ axis 
({\it i.e.} $ \hat{u}_\mathfrak{p} ( - \hat{x}_1 , \hat{x}_2) = \hat{u}_\mathfrak{p} ( \hat{x}_1 , \hat{x}_2) $), so that, 
possibly translating along the $x_2$-axis, we may assume
\begin{equation}
\label{posi_vortex}
 ( \hat{y}_{\mathfrak{p}, -} )_2 = ( \hat{y}_{\mathfrak{p}, +} )_2 = 0 
 \quad \quad \quad {\rm and} \quad \quad \quad 
 - ( \hat{y}_{\mathfrak{p}, -} )_1 = ( \hat{y}_{\mathfrak{p}, +} )_1 \to \frac{1}{2\pi} .
\end{equation}

If we do not assume {\it a priori} the symmetry in $ x_1 $, then we may remove the translation invariance 
by imposing $ \hat{y}_{\mathfrak{p}, +} + \hat{y}_{\mathfrak{p}, -} = 0 $, and then we may still show that 
$ \hat{y}_{\mathfrak{p}, +} = - \hat{y}_{\mathfrak{p}, -}  \to ( 1 / (2\pi) , 0 ) $ by using the Hopf differential 
as in \cite{BBH} (chapter VII).

%%%%%%%%%%%%%%%%%%%%%%%%%%%%%%%%%%%%%%%%%%%%%%%%%%%%%%%
\subsubsection{Strong convergence outside the vorticity set at scale $ x / \mathfrak{p} $}
\label{sec:cvhatxoutside}

We start with a $ W^{1,p}_{\rm loc} $ bound at scale $ \hat{x} $, for $ 1 \ls p < 2 $.

\bigbreak

\noindent {\it Step 1: for any $ 1 \ls p < 2 $, there exists $ C_p $ such that, for any $ \hat{X} \in \R^2 $, 
we have
\[
 \int_{ B(\hat{X}, 1 ) } \lvert \nabla \hat{u}_\mathfrak{p} \rvert^p \, d \hat{x} \ls C_p .
\]
}
We shall adapt the proof of \cite{BOS} (see proof of Theorem 4, Step 3, p. 83) to the two-dimensional case. 
Actually, the only modification to make in the estimate is to replace (C.26) there by the standard convolution
\[
 \psi_{0,i} ( \hat{x} ) 
 = - \frac{\ln r}{2 \pi} \star \omega_{0,i} (\hat{x})  
 = - \frac{1}{2\pi} \int_{\textrm{Supp} (\omega_{0,i}) } \omega_{0,i} (\hat{y}) \ln | \hat{x} - \hat{y} | \, d\hat{y} ,
\]
and then use, for $ | \hat{x} - \hat{y}_{\mathfrak{p} ,\pm} | \gs 3 R_* $, that
\begin{align*}
 \lvert \nabla \psi_{0,\pm} ( \hat{x} ) \rvert 
 & = \Big\lvert \frac{1}{2\pi} \int_{\textrm{Supp} (\omega_{0,\pm}) } \omega_{0,i} (\hat{y}) \nabla_{\hat{x}} \ln | \hat{x} - \hat{y} | \, d\hat{y} \Big\rvert 
 \\ & 
 \ls \frac{1}{2\pi} \| \omega_{0,\pm} \|_{[ \BC^{0,1}_c ( B(\hat{y}_{\mathfrak{p} ,\pm} ,2R_*) ) ]^* } 
 \| (\hat{x} - \hat{y}) / | \hat{x} - \hat{y} |^2 \|_{\BC^{0,1}( B(\hat{y}_{\mathfrak{p} ,\pm} ,3R_*) )} 
 \\ &
 \ls K 
\end{align*}
(the estimate $ \| \psi_{0,\pm} \|_{\BC^k(\R^2 \setminus B( \hat{y}_{\mathfrak{p} ,\pm} , 3 R_*) ) } \ls C_k $ does 
not hold since the two dimensional fundamental solution $ ( \ln r )/(2\pi ) $ goes to $ +\ii $ at spatial infinity, 
but $ \| \nabla \psi_{0,\pm} \|_{\BC^k(\R^2 \setminus B( \hat{y}_{\mathfrak{p} ,\pm} , 3 R_*) ) } \ls C_k $ is true). 
The rest of the proof remains unchanged.

\bigbreak

\noindent {\it Step 2: for any 
$ \hat{X } \in \R^2 \setminus ( B ( \hat{y}_{\mathfrak{p},+} , 2/10 ) \cup B ( \hat{y}_{\mathfrak{p},-} , 2/10 ) ) $, 
we may write $ \hat{u}_\mathfrak{p} = A \ex^{i \phi } $ in $ B ( \hat{X} , 1/20 ) $, with, for any $k \in \N $,
\begin{equation}
\label{cv_out_hatx}
 \Big\| 2( 1 - A ) - \frac{c_\mathfrak{p} }{\mathfrak{p}} \p_2 \phi \Big\|_{\BC^k ( B ( \hat{X} , 1/20 ) ) } 
 \ls \frac{C_k}{ \mathfrak{p}^2 } ,
 \quad \quad \quad
 \| \nabla \phi \|_{\BC^k ( B ( \hat{X} , 1/20 ) ) } \ls C_k ,
\end{equation}
for some constant $ C_k $ independent of $ \hat{X} $.
}

The proof (relying on Step 1) follows the lines of the proof of Step 7 (p. 48) of Theorem 1 in \cite{BOS} and is omitted. 

\bigbreak
In view of the upper bound of Step 1 of subsection \ref{sec:scalezchapeau}, we infer the uniform estimate
\begin{equation}
\label{unif_1}
 \| 1 - \lvert \hat{u}_\mathfrak{p} \rvert \|_{ \BC^k ( B ( \hat{X} , 1/20 ) )} 
 \ls C_k \frac{\ln \mathfrak{p} }{ \mathfrak{p}^2 } ,
\end{equation}
for $ \hat{X } \in \R^2 \setminus ( B ( \hat{y}_{\mathfrak{p},+} , 2/10 ) \cup B ( \hat{y}_{\mathfrak{p},-} , 2/10 ) ) $.

%%%%%%%%%%%%%%%%%%%%%%%%%%%%%%%%%%%%%%%%%%%%%%%%%%%%%%%
\subsubsection{Lower bound for the energy and upper bound for the potential energy}
\label{sec:lower}

\noindent {\it Step 1: upper bound for the potential.} We claim that
\[
 \int_{\R^2} \big\lvert \nabla \lvert \hat{u}_\mathfrak{p} \rvert \big\rvert^2 
 + \frac{ \mathfrak{p}^2}{2} ( 1 - \lvert \hat{u}_\mathfrak{p} \rvert^2 )^2 \, d \hat{x} \ls C( \Lambda_0) 
\]
and that
\[
 \int_{\R^2 \setminus ( B ( \hat{y}_{\mathfrak{p},+} , 2/10 ) \cup B ( \hat{y}_{\mathfrak{p},-} , 2/10 ) ) } 
 \lvert \nabla \hat{u}_\mathfrak{p} \rvert^2 
 + \frac{ \mathfrak{p}^2}{2} ( 1 - \lvert \hat{u}_\mathfrak{p} \rvert^2 )^2 \, d \hat{x} \ls C( \Lambda_0) .
\]

The proof of this upper bound will be a direct consequence of the lower bounds established in \cite{Sandier_lower} 
(see Theorems 2  and 3 there).

\begin{theorem}[\cite{Sandier_lower}]
\label{Th_Sandier_lower}
Let $ \Omega \subset \R^2 $ be a bounded smooth domain. Assume that $ u \in H^1 (\Omega , \C ) $ and that 
$ u_{| \p \Omega} \in \BC^1 ( \p \Omega, \mathcal{S}^1 ) $. Let $ \delta \in ] 0 , 1 [ $.

$ (i) $ There exists a constant $ \Lambda_1 $, depending on $ \Omega $ and $ \| u_{| \p \Omega} \|_{\BC^1} $, 
such that
\[
 \frac{1}{2} \int_{\Omega } \lvert \nabla u \rvert^2 + \frac{ 1}{2 \delta^2 } ( 1 - \lvert u \rvert^2 )^2 
\gs \pi \lvert {\rm deg} ( u_{| \p \Omega} , \p \Omega ) \rvert \ln ( 1 / \delta) - \Lambda_1 .
\]
$ (ii) $ If, moreover, for some constant $ \Lambda_2 $, we have
\[
 \frac{1}{2} \int_{\Omega } \lvert \nabla u \rvert^2 + \frac{ 1}{2 \delta^2 } ( 1 - \lvert u \rvert^2 )^2 
\ls \pi \lvert {\rm deg} ( u_{| \p \Omega} , \p \Omega ) \rvert \ln ( 1 / \delta) + \Lambda_2 ,
\]
then
\[
 \frac{1}{2} \int_{\Omega } 
 \big\lvert \nabla \lvert u \lvert \big\rvert^2 
 + \frac{ 1}{2 \delta^2 } ( 1 - \lvert u \rvert^2 )^2 \ls C( \Omega , \Lambda_2 , \| u_{| \p \Omega} \|_{\BC^1} ) .
\]
\end{theorem}

We shall apply this result with $ \delta = 1 / \mathfrak{p} \ll 1 $, $ \Omega = B ( \hat{y}_{\mathfrak{p}, \pm} , 2/10 ) $ 
and $ u = \hat{u}_\mathfrak{p} $. Since $ {\rm deg} ( \hat{u}_\mathfrak{p} , \p B ( \hat{y}_{\mathfrak{p}, \pm} , 2/10 ) ) = \pm 1 $, 
and in view of the upper bound \eqref{upper_bound} on the energy of $ \hat{u}_\mathfrak{p} $, this yields
\[
 \int_{ B( \hat{y}_{\mathfrak{p}, \pm} , 2/10 ) } 
 \lvert \nabla \hat{u}_\mathfrak{p} \rvert^2 
 + \frac{ \mathfrak{p}^2}{2} ( 1 - \lvert \hat{u}_\mathfrak{p} \rvert^2 )^2 \, d \hat{x} 
 \gs \pi \ln \mathfrak{p} - \Lambda_1  
\]
and
\[
 \int_{ B( \hat{y}_{\mathfrak{p}, \pm} , 2/10 )} \big\lvert \nabla \lvert \hat{u}_\mathfrak{p} \rvert \big\rvert^2 
 + \frac{ \mathfrak{p}^2}{2} ( 1 - \lvert \hat{u}_\mathfrak{p} \rvert^2 )^2 \, d \hat{x} \ls C( \Lambda_0) .
\]
We conclude by using once again the upper bound \eqref{upper_bound}. Actually, 
$\hat{u}_\mathfrak{p} $ does not belong to $ \BC^1 ( \p B ( \hat{y}_{\mathfrak{p}, \pm} , 2/10 ) ) $, 
but it is easy, using \eqref{cv_out_hatx}, to construct an extension of $ \hat{u}_\mathfrak{p} $ on 
$ B ( \hat{y}_{\mathfrak{p}, \pm} , 3/10 ) $ with the required properties by linear interpolation 
(see, for instance the Lemma on p. 395-396 in \cite{Sandier_lower}).

\bigbreak

\noindent {\it Step 2: there exists $ \sigma_0 > 0 $ such that we have, for $ R \gs 1 $,
\[
 \int_{ \R^2 \setminus B( 0 , R )} \lvert \nabla \hat{u}_\mathfrak{p} \rvert^2 
 + \frac{ \mathfrak{p}^2}{2} ( 1 - \lvert \hat{u}_\mathfrak{p} \rvert^2 )^2 \, d \hat{x} 
 \ls \frac{ C( \Lambda_0) }{ R^{\sigma_0} } 
 .
\]
}

The proof is similar to that of Lemma 5.1 (p. 50) in \cite{BOS}, and relies on the fact that 
$ \lvert \hat{u}_\mathfrak{p} \rvert \gs 1/2 $ in $ \R^2 \setminus B( 0, 1) $ (hence we may write the PDE 
in terms of modulus and phase), and the upper bound in $ \R^2 \setminus ( B(\hat{y}_{\mathfrak{p}, +} , 2/10 ) 
\cup B( \hat{y}_{\mathfrak{p}, -} , 2/10 ) ) \supset \R^2 \setminus B( 0, 1) $ of the energy of $ \hat{u}_\mathfrak{p} $ 
(in \cite{BOS}, this last upper bound was derived differently).

%%%%%%%%%%%%%%%%%%%%%%%%%%%%%%%%%%%%%%%%%%%%%%%%%%%%%%%
\subsubsection{Convergence on the scale $ x / \mathfrak{p} $}
\label{sec:cvhatx}

By Step 1 of subsection \ref{sec:lower} and \eqref{posi_vortex}, we have, as $ \mathfrak{p} \to +\ii $,
\begin{equation}
\label{vortexCV}
 \hat{y}_{\mathfrak{p}, \pm} \to \hat{y}_{\ii, \pm} \assign \pm ( 1/ (2\pi ), 0 ) \in \R^2 .   
\end{equation}
We then define (identifying $ \R^2 $ and $ \C $)
\[
 \hat{u}_\ii (\hat{x} ) \assign \frac{ \hat{x} - \hat{y}_{\ii, +} }{ \lvert \hat{x} - \hat{y}_{\ii, +} \rvert } 
 \times \overline{ \frac{ \hat{x} + \hat{y}_{\ii, -} }{ \lvert \hat{x} + \hat{y}_{\ii, -} \rvert } } .
\]

\noindent {\it Step 1: for any $ p \in [1, 2 [ $, there holds, in $ W^{1, p}_{\rm loc} (\R^2) $,}
\[
 \hat{u}_\mathfrak{p} \rightharpoonup \hat{u}_\ii . 
\]

From the $ W^{1,p}_{\rm loc} $ upper bound of Step 1 in subsection \ref{sec:cvhatxoutside} and by weak compactness, 
there exists $ \hat{U} \in W^{1, p}_{\rm loc} (\R^2) $ such that $ \hat{u}_\mathfrak{p} \rightharpoonup \hat{U} $ 
in $ W^{1, p}_{\rm loc} ( \R^2) $. Moreover, $ \hat{U} \in \BC^{\ii}_{\rm loc} ( \R^2 \setminus \{ \hat{y}_{\ii, +} , \hat{y}_{\ii, -} \} ) $ 
and the convergence holds in $\BC^k_{\rm loc} ( \R^2 \setminus \{ \hat{y}_{\ii, +} , \hat{y}_{\ii, -} \} ) $ 
by Step 2 of subsection \ref{sec:cvhatxoutside} (for any $k \in \N $). In order to determine $ \hat{U} $, we shall pass 
to the limit in the system
\[
 \left\{\begin{array}{l}
 \nabla \cdot ( \hat{u}_\mathfrak{p} \wedge \nabla \hat{u}_\mathfrak{p} ) 
 = - \frac{1}{2} c_\mathfrak{p} \mathfrak{p} \p_2 ( \lvert \hat{u}_\mathfrak{p} \rvert^2 -1 )
 \\ 
 \nabla^\perp \cdot ( \hat{u}_\mathfrak{p} \wedge \nabla \hat{u}_\mathfrak{p} ) = 2 J \hat{u}_\mathfrak{p} 
 \end{array}
 \right.
\]
obtained from \eqref{TW_resc} and the definition of the Jacobian. From \eqref{upper_bound} 
(implying $ c_\mathfrak{p} \mathfrak{p} \p_2 ( \lvert \hat{u}_\mathfrak{p} \rvert^2 -1 ) \to 0 $ in the 
distributional or the $ H^{-1}$ sense) and Step 5 of subsection \ref{sec:scalezchapeau}, we then infer
\[
 \left\{\begin{array}{l}
 \nabla \cdot ( \hat{U} \wedge \nabla \hat{U} ) = 0
 \\ 
 \nabla^\perp \cdot ( \hat{U} \wedge \nabla \hat{U} ) = 2 \pi ( \delta_{\hat{y}_{\ii,+}} - \delta_{\hat{y}_{\ii,-}} ) .
 \end{array}
 \right.
\]
It then follows that $ \hat{U} \wedge \nabla \hat{U} = \hat{u}_\ii \wedge \nabla \hat{u}_\ii $, hence 
the existence of $ \Theta \in \R $ such that $ \hat{U} = \ex^{i \Theta } \hat{u}_\ii $. 
We finally use the $x_1$-symmetry to infer $ \Theta = 0 $.

\bigbreak

\noindent {\it Step 2: as $ \mathfrak{p} \to +\ii $, we have}
\[
 \mathfrak{p} c _\mathfrak{p} = \frac{\mathfrak{p}^2}{2} \int_{\R^2} ( 1 - \lvert \hat{u}_\mathfrak{p} \rvert^2 )^2 \, d \hat{x} 
 \to 2 \pi .
\]
This is claimed in \cite{Bet_Saut} (Proposition VI.7 there), but the proof is not clearly given. 

One way to prove this point is to use the Hopf differential as in \cite{BBH} (chapter VII). We shall follow the 
alternative proof of Theorem VII.2 given in section VII.1 there. The first equality is the Pohozaev 
identity \eqref{Pohoz}.

First, notice that
\[
 W_\mathfrak{p} \assign \frac{ \mathfrak{p}^2 }{2} ( 1 - \lvert \hat{u}_\mathfrak{p} \rvert^2 )^2 
\]
is a nonnegative function which is bounded in $ L^1 (\R^2) $ by Step 1 of subsection \ref{sec:lower} 
and enjoys the decay estimate of Step 2 of subsection \ref{sec:lower}. In addition, by \eqref{unif_1} 
(see Step 2 of subsection \ref{sec:cvhatxoutside}), we have $ W_\mathfrak{p} \to 0 $ locally 
uniformly in $ \R^2 \setminus \{ \pm ( 1/ (2\pi) , 0 ) \} $. Up to a subsequence, we may then assume that 
\[
 W_\mathfrak{p} \rightharpoonup \mu_+ \delta_{\hat{y}_{\ii,+}} + \mu_- \delta_{\hat{y}_{\ii,-}}
\]
in the weak $*$ topology of $ \BC_b( \R^2 ) $, for some two reals $ \mu_\pm \gs 0 $, with 
$ \mu_+ + \mu_- = \lim_{\mathfrak{p} \to +\ii } \int_{\R^2} W_\mathfrak{p} $.

We shall now compute $ \mu_+ $ (the case of $\mu_- $ is similar). First, we write, for some $ R_5 \ls 2/10 $, 
the Pohozaev identity for $ \hat{u}_\mathfrak{p} $ on $ B( \hat{y}_{\ii,+} , R_5) $ 
(obtained by multiplying the equation by the conjugate of $ (\hat{x} - \hat{y}_{\ii,+}) \cdot \nabla \hat{u}_\mathfrak{p} $ 
and integrating the real part over $ B( \hat{y}_{\ii,+} , R_5) $), which yields
\begin{align*}
 \int_{ B ( \hat{y}_{\ii,+} , R_5) } \frac{ \mathfrak{p}^2 }{2} ( 1 - \lvert \hat{u}_\mathfrak{p} \rvert^2 )^2 
 & + c_\mathfrak{p} \mathfrak{p} \int_{ B ( \hat{y}_{\ii,+} , R_5) } (\hat{x}_1 - \hat{y}_{\ii,+,1}) \langle i \p_2 \hat{u}_\mathfrak{p} | \p_1 \hat{u}_\mathfrak{p} \rangle 
 \\ & = 
 \frac{R_5}{2} \int_{\p B ( \hat{y}_{\ii,+} , R_5) } \lvert \p_\tau \hat{u}_\mathfrak{p} \rvert^2 - \lvert \p_\nu \hat{u}_\mathfrak{p} \rvert^2 
 + \frac{ \mathfrak{p}^2 }{4} ( 1 - \lvert \hat{u}_\mathfrak{p} \rvert^2 )^2 .
\end{align*}
We then pass to the limit $ \mathfrak{p} \to +\ii $. For the boundary term, we use the strong convergences outside 
the vorticity set; for the second term of the first line, we prove that it tends to zero by following the arguments given 
for Step 6 in subsection \ref{sec:scalezchapeau}. 
We then get
\[
 \mu_+ 
 = 
 \frac{R_5}{2} \int_{\p B ( \hat{y}_{\ii,+} , R_5) } \lvert \p_\tau \hat{u}_\ii \rvert^2 - \lvert \p_\nu \hat{u}_\ii \rvert^2 
 .
\]
By Step 1, we know that $ \hat{u}_\ii = \exp( i {\rm Arg} ( \hat{x} - \hat{y}_{\ii,+} ) - i {\rm Arg} ( \hat{x} - \hat{y}_{\ii,-} ) ) $ 
on $\p B ( \hat{y}_{\ii,+} , R_5) $, and the second term $ {\rm Arg} ( \hat{x} - \hat{y}_{\ii,-} ) $ is smooth and 
harmonic in $ \bar{ D} ( \hat{y}_{\ii,+} , R_5) $. 
As a consequence, we have the Pohozaev identity for $ {\rm Arg} ( \cdot - \hat{y}_{\ii,-} ) $ 
\[
 0 = 
 \frac{R_5}{2} \int_{\p B ( \hat{y}_{\ii,+} , R_5) } \lvert \p_\tau {\rm Arg} ( \hat{x} - \hat{y}_{\ii,-} ) \rvert^2 
 - \lvert \p_\nu {\rm Arg} ( \hat{x} - \hat{y}_{\ii,-} ) \rvert^2 
 ,
\]
$ \p_\tau {\rm Arg} ( \hat{x} - \hat{y}_{\ii,+} = 1 / R_5 $, $ \p_\nu {\rm Arg} ( \hat{x} - \hat{y}_{\ii,+} ) = 0 $, 
and thus by expansion
\begin{align*}
 \mu_+ 
 & = 
 \frac{R_5}{2} \int_{\p B ( \hat{y}_{\ii,+} , R_5) } \lvert \p_\tau \hat{u}_\ii \rvert^2 - \lvert \p_\nu \hat{u}_\ii \rvert^2 
 = 
 \frac{R_5}{2} \int_{\p B ( \hat{y}_{\ii,+} , R_5) } 1/R_5^2 + 2 \p_\tau {\rm Arg} ( \hat{x} - \hat{y}_{\ii,-} ) / R_5 
 = \pi 
  .
\end{align*}
This concludes the proof.

%%%%%%%%%%%%%%%%%%%%%%%%%%%%%%%%%%%%%%%%%%%%%%%%%%%%%%%
\subsubsection{Convergence on the scale $ x $}
\label{sec:cxlocale}

 We shall now focus on the verification of hypothesis 2 of Proposition \ref{locC1}. The main tool is the following 
result. We now work on the scale $ x $.

\begin{proposition}
\label{classif_zoom}
Assume that $ \hat{z}_\mathfrak{p} \in \R^2 $ is such that
\[
 \limsup_{\mathfrak{p} \to +\ii } \lvert \hat{u}_\mathfrak{p} (  \hat{z}_\mathfrak{p} ) \rvert < 1 
\]
and consider the rescaled mapping
\[
 U_\mathfrak{p}  (y) \ddef \hat{u}_\mathfrak{p} ( \hat{z}_\mathfrak{p} + y / \mathfrak{p} ) .
\]
Then, there exists a sign $ \pm $ and $ \beta \in \R $ (depending on the choice of the family $( \hat{z}_\mathfrak{p} )$) 
such that, up to a subsequence, we have, in $ \BC^k_{\rm loc} (\R^2 ) $ for any $ k \in \N $,
\[
 U_\mathfrak{p} \to \ex^{i\beta} V_\pm .
\]
\end{proposition}

\begin{proof} The rescaling $ U_\mathfrak{p} $ solves
\[
  \Delta U_\mathfrak{p} + i c_\mathfrak{p} \p_2 U_\mathfrak{p} 
  + U_\mathfrak{p} ( 1 - \lvert U_\mathfrak{p} \rvert^2 ) = 0 
\]
and satisfies $  \limsup_{\mathfrak{p} \to +\ii } \lvert U_\mathfrak{p} ( 0) \rvert < 1 $ and, by Step 2 of 
subsection \ref{sec:cvhatx},
\[
 \int_{\R^2} ( 1 - \lvert U_\mathfrak{p} \rvert^2 )^2 \, dy 
 = 4 \pi + o_{\mathfrak{p} \to + \ii} (1) .
\]
Then, from the uniform bounds of Theorem \ref{Linf} and Corollary \ref{Gradinf}, we may assume, 
up to a subsequence, that
\begin{equation}
\label{cvBlowup} 
U_\mathfrak{p} \to U_\ii 
\end{equation}
in $ \BC^k_{\rm loc} (\R^2 ) $ with $ \lvert U_\ii ( 0) \rvert < 1 $,
\[
 \Delta U_\ii + U_\ii ( 1 - \lvert U_\ii \rvert^2 ) = 0 
\]
and, by Fatou's lemma,
\[
 \int_{\R^2} ( 1 - \lvert U_\ii \rvert^2 )^2 \, dy 
 \ls 4 \pi  .
\]
By \cite{BMR}, we know that $ \int_{ \R^2} ( 1 - \lvert U_\ii \rvert^2 )^2 \, dy = 2 \pi d^2 $, where 
$ d \in \Z $ is the degree of $ U_\ii $ at infinity. It follows that $ \lvert d \rvert \ls 1 $, and that the 
case $ d = 0 $ is excluded since $ \lvert U_\ii ( 0) \rvert < 1 $, hence $ \lvert U_\ii \rvert \not\equiv 1 $. 
Therefore $ d = \pm 1 $. It then follows from \cite{Miro_radial} that $ U_\ii = \ex^{i\beta} V_d $ 
for some $ \beta \in \R $.
\end{proof}

\bigbreak

We may now localize the set $ \{ \lvert \hat{u}_\mathfrak{p} \rvert \ls 1 - \frac{1}{\lambda_*} \} $, where 
$ \lambda_* $ is as in Proposition \ref{locC1}, rather precisely.

\bigbreak
\noindent {\it Step 1: there exists $ \mathfrak{p}_6 $ large such that, for $ \mathfrak{p} \gs \mathfrak{p}_6 $, 
$ \hat{u}_\mathfrak{p} $ has exactly two zeros $ \hat{z}_{\mathfrak{p}, \pm} $. 
Up to a translation in the $ x_2 $ direction, we may assume
\[
 \R \times \{ 0 \} \ni \hat{z}_{\mathfrak{p}, \pm} \to ( \pm 1/( 2\pi), 0) \in \R^2 .
\]
Moroever, there exists $ R_0 > 0 $ such that $ \{ \lvert \hat{u}_\mathfrak{p} \rvert \ls 1 - \frac{1}{\lambda_*} \} \subset 
B ( \hat{z}_{\mathfrak{p}, +} , R_0 / \mathfrak{p} ) \cup B ( \hat{z}_{\mathfrak{p}, -} , R_0 / \mathfrak{p} ) $. 
Here, $ \lambda_* > 0 $ is the large universal constant appearing in Proposition \ref{locC1}.
}

By Step 8 of subsection \ref{sec:scalezchapeau}, we know (due to the nonzero degree) that 
$ \hat{u}_\mathfrak{p} $ has at least two zeroes, one in each disk $ B ( \hat{y}_{\mathfrak{p}, \pm } , 3/20 ) $. 

Now, if $ \hat{z}_\mathfrak{p} $ is a zero of $ \hat{u}_\mathfrak{p} $, we know by Proposition \ref{classif_zoom} 
that, for some $ \beta \in \R $ (depending on the sequence $ ( \hat{z}_\mathfrak{p})_{\mathfrak{p}} $) and 
$ d_0 = \pm 1 $, we have
\begin{equation} 
\label{Convergence_loc}
 \hat{u}_\mathfrak{p} ( \hat{z}_\mathfrak{p} + \mathfrak{p} y ) 
 \to \ex^{i\beta} V_{d_0} ( y)
\end{equation}
in $ \BC^k_{\rm loc} (\R^2 ) $. As noticed in \cite{Qing}, since $ V_{\pm} : \R^2 \to \C \approx \R^2 $ 
has nonzero jacobian at the origin, we deduce that for any $ R > 0 $, and for $ \mathfrak{p} \gs \mathfrak{p}_R $ 
large enough $ 0 $ is the only zero of $ U_\mathfrak{p} $ in $ B (0, R ) $. Roughly speaking, 
there does not exist zeroes $ \hat{z} $, $ \hat{z}' $ of $ \hat{u}_\mathfrak{p} $ such that 
$ 0 < \lvert \hat{z} - \hat{z}' \rvert = \mathcal{O}( 1/ \mathfrak{p} ) $.

We now fix $ R_0 > 0 $ sufficiently large so that
\[
 \int_{\{ \lvert y \rvert \ls R_0 /2 \} } ( 1- \lvert V_1 ( y) \rvert^2 )^2 \, dy \gs \frac{3\pi}{2} .
\]
and we assume that (for any large $ \mathfrak{p} $), $ \{ \lvert \hat{u}_\mathfrak{p} \rvert \ls 1 - \frac{1}{\lambda_*} \} $ 
(where $ \lambda_* > 0 $ is the one appearing in Proposition \ref{locC1}) 
is not included in $ B ( \hat{z}_{\mathfrak{p}, +} , R_0 / \mathfrak{p} ) \cup B ( \hat{z}_{\mathfrak{p}, -} , R_0 / \mathfrak{p} ) $. 
This means that there exists $ \hat{Z}_{\mathfrak{p}} \in B ( \hat{z}_{\mathfrak{p}, +} , 3 / 20 ) 
\setminus B ( \hat{z}_{\mathfrak{p}, +} , R_0 / \mathfrak{p} ) $ (say) 
with $ \lvert \hat{u}_\mathfrak{p} ( \hat{Z}_\mathfrak{p}) \rvert \ls 1 - \frac{1}{\lambda_*} $. 
By Proposition \ref{classif_zoom}, the rescaled mapping 
$ U_\mathfrak{p} (y) \ddef \hat{u}_\mathfrak{p} ( \hat{Z}_\mathfrak{p} + \mathfrak{p} y ) $ 
converges (up to a subsequence) in $ \BC^k_{\rm loc} (\R^2 ) $ to $ U_\ii \in \mathbb{S}^1 V_\pm $ and 
we know (from \cite{BMR}) that $ \int_{ \R^2} ( 1 - \lvert U_\ii \rvert^2 )^2 \, dy = 2 \pi $. 
As a consequence, 
since $ \lvert \hat{z}_{\mathfrak{p}, +} - \hat{Z}_{\mathfrak{p}} \rvert \gs R_0 / \mathfrak{p} $,
\begin{align*}
 2 \pi + o(1) 
 & = \mathfrak{p}^2 \int_{B (\hat{y}_{\mathfrak{p},+}, 3/ 20 )} ( 1 - \lvert \hat{u}_{\mathfrak{p}} \rvert^2 )^2 \, d\hat{x} 
 \\ & 
 \gs \mathfrak{p}^2 \int_{B (\hat{z}_{\mathfrak{p},+}, R_0 / (2\mathfrak{p} ) )} ( 1 - \lvert \hat{u}_{\mathfrak{p}} \rvert^2 )^2 \, d\hat{x} 
 + \mathfrak{p}^2 \int_{B (\hat{Z}_{\mathfrak{p}}, R_0 / (2\mathfrak{p}) )} ( 1 - \lvert \hat{u}_{\mathfrak{p}} \rvert^2 )^2 \, d\hat{x} 
 \\ & 
 \gs \int_{\{ \lvert y \rvert \ls R_0 /2 \} } ( 1- \lvert V_1 \rvert^2 )^2 \, dy 
 + \int_{\{ \lvert y \rvert \ls R_0 /2 \} } ( 1 - \lvert U_\ii \rvert^2 )^2 \, dy + o(1)
 \\ & 
 \gs \frac{3\pi}{2} + \frac{3\pi}{2} + o(1) ,
\end{align*}
which is absurd. We then conclude $ \| \lvert u_\mathfrak{p} \rvert - 1 \|_{L^\ii (\{ \tilde{r}_d \gs R_0 \} ) } \ls \frac{1}{\lambda_*} $ 
for $ \mathfrak{p} $ sufficiently large, then proving hypothesis 3 of Proposition \ref{locC1} 
with $ \lambda = \max ( R_0, \lambda_* ) $. 
Another consequence of this fact is that $ \hat{u}_\mathfrak{p} $ possesses at most two (simple) zeroes 
$ \hat{z}_{\mathfrak{p}, \pm} $.

We then define $ d = d_\mathfrak{p} $ such that the unique zero $ \hat{z}_{\mathfrak{p}, + } $ 
of $ \hat{u}_\mathfrak{p} $ in the right half-plane is
\[
 \hat{z}_{\mathfrak{p}, + } =  \frac{d_\mathfrak{p}}{\mathfrak{p}} \overrightarrow{e_1} 
 \to ( 1/( 2\pi), 0) \in \R^2 .
\]
We deduce from Step 2 of subsection \ref{sec:cvhatx} that
\[
 d_\mathfrak{p} \sim \frac{\mathfrak{p}}{ 2\pi  } \sim \frac{ 1 }{ c_\mathfrak{p} } ,
\]
so that hypothesis 4 of Proposition \ref{locC1} is satisfied for $ \mathfrak{p} $ large enough 
(still for $ \lambda = \max ( R_0, \lambda_* ) $). 
Furthermore, hypothesis 2 of Proposition \ref{locC1} is satisfied by taking 
$ \mathfrak{p} $ large enough, associated with the choice $ \lambda = \max ( R_0, \lambda_* ) $.

\bigbreak

\noindent {\it Step 2: conclusion.} Applying Proposition \ref{locC1} to $ \ex^{- i \beta } u_\mathfrak{p} $, 
we infer that there exists $ \gamma_\mathfrak{p} \in \R $ such that (for large $ \mathfrak{p} $)
\[
 u_\mathfrak{p} = \ex^{i \gamma_\mathfrak{p} } Q_{c_\mathfrak{p}}  
\]
(no translation is needed in the $x_2 $ direction at this stage since the zeros of $\hat{u}_{\mathfrak{p}} $ are 
on the $x_1$-axis).

%%%%%%%%%%%%%%%%%%%%%%%%%%%%%%%%%%%%%%%%%%%%%%%%%%%%%%%%%%
\subsection{Decay slightly away from the vortices}
\label{subsec:decaylightly}

In this section, we provide some estimates for $ \hat{u}_\mathfrak{p} $ in the region 
$ B( \hat{z}_{\mathfrak{p},+} , 2 R_0 ) \cup B( \hat{z}_{\mathfrak{p},-} , 2 R_0 ) $. 
For the Ginzburg-Landau (stationary) model, such estimates have been first given in \cite{Miro_explicit} 
for minimizing solutions and later generalized in \cite{Com_Mir} to non-minimizing solutions. 
However, the paper \cite{Miro_explicit} being difficult to find, we give here 
a proof of these estimates that includes the transport term. 
They improve some estimates in \cite{Chi_Pac_1} 
and are not specific to the way we construct the solutions.

\begin{proposition}
\label{decay_slightly}
We have, for $ \lvert \hat{y} \lvert \ls \frac{3}{20} $,
\[
 \big\lvert \lvert \hat{u}_\mathfrak{p} ( \hat{z}_{\mathfrak{p}, \pm} + \hat{y} ) \rvert - 1 \big\rvert 
 \ls \frac{C}{ \mathfrak{p}^2 \lvert \hat{ y} \rvert^2 } , 
 \quad \quad 
 \big\lvert \nabla \lvert \hat{u}_\mathfrak{p} \rvert ( \hat{z}_{\mathfrak{p}, \pm} + \hat{y} ) \big\rvert 
 \ls \frac{C}{ \mathfrak{p}^2 \lvert \hat{ y} \rvert^3 } , 
 \quad \quad 
 \big\lvert \nabla \hat{u}_\mathfrak{p} ( \hat{z}_{\mathfrak{p}, \pm} + \hat{y} ) \big\rvert  
 \ls \frac{C}{ \lvert \hat{ y} \rvert } 
 .
\]
\end{proposition}

\begin{proof} 
We work near $ \hat{z}_{\mathfrak{p}, +}  $ (the minus sign is similar), say in the annulus 
$ B( \hat{z}_{\mathfrak{p}, +} , 1/10 ) \setminus B( \hat{z}_{\mathfrak{p}, +} , 1/\mathfrak{p} ) $ and set
\[
 \hat{u}_\mathfrak{p} ( \hat{z}_{\mathfrak{p}, +} + \hat{y} ) 
 = \hat{A}_\mathfrak{p} ( \hat{y} ) \ex^{i \theta + i \hat{\varphi}_\mathfrak{p} ( \hat{y} ) } 
\]
with $ \hat{A}_\mathfrak{p} $ and $ \hat{\varphi}_\mathfrak{p} $ real-valued and smooth in the annulus ($ \theta $ 
is the polar angle centered at $ \hat{z}_{\mathfrak{p}, +} $). Then, we obtain the system
\begin{equation*}
\left\{ 
 \begin{array}{l}
 \Delta \hat{A}_\mathfrak{p} 
 - \hat{A}_\mathfrak{p} \lvert \nabla \hat{\varphi}_\mathfrak{p} \rvert^2 
 + \mathfrak{p}^2 \hat{A}_\mathfrak{p} \lvert V_1 \rvert^2 ( 1 - \hat{A}_\mathfrak{p}^2 ) 
 - 2 \hat{A}_\mathfrak{p} \frac{\p_\theta \varphi}{ r^2 }
 - c_\mathfrak{p} \mathfrak{p} \hat{A}_\mathfrak{p} \p_2 \hat{\varphi}_\mathfrak{p} 
 - c_\mathfrak{p} \mathfrak{p} \frac{ \cos \theta }{r} \hat{A}_\mathfrak{p} 
 = 0 \\
 \hat{A}_\mathfrak{p} \Delta \hat{\varphi}_\mathfrak{p} 
 + 2 \nabla \hat{A}_\mathfrak{p} \cdot \nabla \hat{\varphi}_\mathfrak{p} 
 + 2 \frac{\p_\theta \hat{A}_\mathfrak{p} }{r^2}
 + c_\mathfrak{p} \mathfrak{p} \p_2 \hat{A}_\mathfrak{p} 
= 0 .
 \end{array} \right.
\end{equation*}
The second equation may be recast as
\begin{equation}
\label{accuphase}
 \nabla ( \hat{A}_\mathfrak{p}^2 \nabla \hat{\varphi}_\mathfrak{p} ) 
 + \frac{\p_\theta \hat{A}_\mathfrak{p}^2 }{r^2}
 = - \frac{ c_\mathfrak{p} \mathfrak{p} }{2}  \p_2 ( \hat{A}_\mathfrak{p}^2 - 1 ) .
\end{equation}
Multiplying by $ \hat{\varphi}_\mathfrak{p} $ and integrating over $ DB( 0, 3/20 ) \setminus B( 0, R_0/\mathfrak{p} ) $, we obtain
\begin{align*}
 \int_{B( 0, 3/20 ) \setminus B( 0, R_0/\mathfrak{p} )} \hat{A}_\mathfrak{p}^2 \lvert \nabla \hat{\varphi}_\mathfrak{p} \rvert^2 \, d \hat{y} 
 & = 
 \int_{B( 0, 3/20 ) \setminus B( 0, R_0/\mathfrak{p} )} ( 1 - \hat{A}_\mathfrak{p}^2 ) \frac{\p_\theta \hat{\varphi}_\mathfrak{p} }{r^2}  
 + \frac{ c_\mathfrak{p} \mathfrak{p} }{2} ( 1 - \hat{A}_\mathfrak{p}^2 ) \p_2 \hat{\varphi}_\mathfrak{p} \, d \hat{y}
 \\ & 
 \quad + \int_{\p B( 0, 3/20 )}  \hat{A}_\mathfrak{p}^2 \frac{ \p \hat{\varphi}_\mathfrak{p} }{\p \nu } 
 + \frac{ c_\mathfrak{p} \mathfrak{p} }{2} ( \hat{A}_\mathfrak{p}^2 - 1 ) \hat{\varphi}_\mathfrak{p} \nu_2 \, d \ell .
\end{align*}
By Cauchy-Schwarz inequality, \eqref{upper_bound} and Step 1 of subsection \ref{sec:lower}, we infer
\begin{align*}
 \| \nabla \hat{\varphi}_\mathfrak{p} \|_{L^2( B( 0, 3/20 ) \setminus B( 0, R_0/\mathfrak{p} ) )}^2 
 \ls C ( 1+ c_\mathfrak{p} ) \| \nabla \hat{\varphi}_\mathfrak{p} \|_{L^2( B( 0, 3/20 ) \setminus B( 0, R_0/\mathfrak{p} ) )} 
 + C 
\end{align*}
where, for the contribution of the integral over $ \p B( 0, 3/20 ) $, we have used \eqref{unif_1} and 
\eqref{cv_out_hatx} (see Step 2 of subsection \ref{sec:cvhatxoutside}). This implies
\begin{equation}
\label{enphase}
 \| \nabla \hat{\varphi}_\mathfrak{p} \|_{L^2( B( 0, 3/20 ) \setminus B( 0, R_0/\mathfrak{p} ) )}
 \ls C .
\end{equation}

We fix $ \hat{y} \in \R^2 $ such that $ 2 R_0 / \mathfrak{p} \ls \lvert \hat{y} \lvert \ls \frac{3}{20} $. 
Then, since $ \lvert \hat{u}_{\mathfrak{p}} \rvert \gs 1/2 $ in the annulus 
$ B( 0, 3/20 ) \setminus B( 0, R_0/\mathfrak{p} ) \supset B( \hat{y}, \lvert \hat{y} \rvert/ 2 ) $, 
we deduce
\begin{align*}
 \int_{B( \hat{y}, \lvert \hat{y} \rvert/ 2 ) } 
  \hat{A}_\mathfrak{p}^2 \lvert \nabla \hat{\varphi}_\mathfrak{p} + \vec{e}_\theta / r \rvert^2 \, d \hat{x} 
 & \ls C \int_{B( \hat{y}, \lvert \hat{y} \rvert/ 2 ) } 
  \lvert \nabla \hat{\varphi}_\mathfrak{p} \rvert^2 + \frac{1}{ r^2 } \, d \hat{x} 
  \ls C 
\end{align*}
by \eqref{enphase} and the fact that $ r = \lvert \hat{x} \rvert \gs \lvert \hat{y} \rvert/ 2 $. 
By Step 1 of subsection \ref{sec:lower}, we then infer the upper bound (also shown in \cite{Miro_explicit})
\begin{equation}
\label{supperbound}
 E_{1/\mathfrak{p}} ( \hat{u}_\mathfrak{p} , B( \hat{y}, \lvert \hat{y} \rvert/ 2 ) ) \ls C . 
\end{equation}
We now make some rescaling and consider
\[
 v (X) \ddef \hat{u}_\mathfrak{p} \Big( \hat{y} + \frac{\lvert \hat{y} \rvert}{ 2} X \Big) 
\]
in $ B( 0,1 ) $ ($v$ depends on $ \hat{y} $ and $ \mathfrak{p} $), which solves
\[
 \Delta v + i \frac{c_\mathfrak{p}}{\delta} \p_2 v + \frac{1}{\delta^2} v ( 1 - \lvert v \rvert^2 ) = 0 
\]
in $ B( 0, 1 ) $, with $ \delta \assign 2 / ( \mathfrak{p} \lvert \hat{y} \lvert ) $. This equation is of the type 
\eqref{GL_transpo} with "$ \ep = \delta $" and "$ \mathfrak{c} = c_\mathfrak{p} /\delta $". Let us check that the 
assumption $ \lvert \mathfrak{c} \rvert \ls M_0 \lvert \ln \ep \rvert $ is satisfied with $ M_0 = 1$. As a matter of 
fact, we have $ \delta = 2 / ( \mathfrak{p} \lvert \hat{y} \lvert ) \in ] 40/ ( 3 \mathfrak{p}) , 1/2 ] $, thus
\[
 M_0 \delta \lvert \ln \delta \rvert 
 \gs \frac{40}{ 3 \mathfrak{p} } \ln 2
 \gs c_\mathfrak{p} 
 = \frac{ 2\pi }{ \mathfrak{p} } + o(1) 
\]
by Step 2 of subsection \ref{sec:cvhatx} (note $ 40 (\ln 2) /3 \approx 9.24(1) > 2 \pi $). 
Furthermore, the upper bound \eqref{supperbound} reads now
\begin{equation*}
 E_{ \delta } ( v , B( 0, 1 ) ) \ls C . 
\end{equation*}
It then follows from the proof of Step 7 (p. 48) of Theorem 1 in \cite{BOS} that, for 
$ \delta $ sufficiently small,
\[
 \| 2 \delta^{-2} ( 1 - \lvert v \rvert ) - c_\mathfrak{p} \delta^{-1} \p_2 \arg(v) \|_{\BC^1 ( B ( 0 , 1/2 ) ) } 
 \ls C ,
 \quad \quad \quad
 \| \nabla \arg(v) \|_{\BC^1 ( B ( 0, 1/2 ) ) } \ls C .
\]
Therefore, by Step 2 of subsection \ref{sec:lower},
\[
 \big\lvert 1 - \lvert v (0) \rvert \big\rvert + \big\lvert \nabla \lvert v \rvert (0) \big\rvert 
 \ls C c_\mathfrak{p} \delta + C \delta^2 
 \ls \frac{C}{ \mathfrak{p}^2 \lvert \hat{y} \lvert^2 } , 
 \quad \quad 
 \big\lvert \nabla \arg(v) (0) \big\rvert 
 \ls C ,
\]
and scaling back this yields the conclusion, at least for $ \delta = 2 / ( \mathfrak{p} \lvert \hat{y} \lvert ) $ 
sufficiently small, say $ \mathfrak{p} \lvert \hat{y} \lvert \gs \delta_0 / 2 $, but the estimate 
is easy to show if $ \mathfrak{p} \lvert \hat{y} \lvert \ls \delta_0 / 2 $.
\end{proof}

%%%%%%%%%%%%%%%%%%%%%%%%%%%%%%%%%%%%%%%%%%%%%%%%%%%%%%%%%%
\subsection{Some remarks on the non symmetrical case}
\label{subsec:nonsym}
In the case where we do not assume the $x_1$ symmetry for $ u_\mathfrak{p} $, the location of 
the vortices $ \hat{y}_{\mathfrak{p},\pm} $ is more delicate. Indeed, we can no longer assume 
\eqref{posi_vortex}, that is
\[
  ( \hat{y}_{\mathfrak{p}, -} )_2 = ( \hat{y}_{\mathfrak{p}, +} )_2 = 0 
 \quad \quad \quad {\rm and} \quad \quad \quad 
 - ( \hat{y}_{\mathfrak{p}, -} )_1 = ( \hat{y}_{\mathfrak{p}, +} )_1 \to \frac{1}{2\pi} .
\]
Up to a translation, we may assume $ \hat{y}_{\mathfrak{p}, +} + \hat{y}_{\mathfrak{p}, -} = 0 $, 
and it remains true that $ \hat{y}_{\mathfrak{p}, + , 1} - \hat{y}_{\mathfrak{p}, -, 1 } \to 1 / \pi $, 
but we may have $ \lvert \hat{y}_{\mathfrak{p}, +} - \hat{y}_{\mathfrak{p}, -} \rvert \gg 1 $. 
By following carefully the proof in \cite{Sandier_lower}, one could show that 
\[
 \lvert \hat{y}_{\mathfrak{p}, +} - \hat{y}_{\mathfrak{p}, -} \rvert \ls C . 
\]
Then, the location of the limiting vortices $ \hat{y}_{\ii, \pm } = \lim_{ \mathfrak{p} \to +\ii} \hat{y}_{\mathfrak{p}, \pm } $ 
can be obtained through the use of the Hopf differential as in \cite{BBH} (chapter VII), and would lead as before 
to $ \hat{y}_{\ii, \pm } = ( \pm 1 / (2\pi ) , 0) $. This is of course related to the fact that the only critical point 
of the action functional
\[
 \mathcal{F} ( \hat{y}_{\ii, +} , \hat{y}_{\ii, -} ) 
 \assign 2\pi \Big( 2 \ln \lvert \hat{y}_{\ii, +} - \hat{y}_{\ii, -} \rvert - 2\pi \big[ ( \hat{y}_{\ii, +})_1 - ( \hat{y}_{\ii, -} )_1 \big] \Big)
\]
associated with the action of the Kirchhoff energy is $ ( \hat{y}_{\ii, +} , \hat{y}_{\ii, -} ) = ( 1 /(2\pi) , -1/(2\pi) ) \in \C^2 $ 
(up to translation).

Next, Step 1 of subsection \ref{sec:cvhatx} becomes, for any $ p \in [ 1, 2 [ $, and in $W^{1,p}_{\rm loc} (\R^2) $,
\[
 \hat{u}_\mathfrak{p} \rightharpoonup \ex^{i \Theta} \hat{u}_\ii .
\]
The term $ \Theta $ is somewhat the phase at infinity, even though we do not claim 
some uniformity at infinity in space. Next, for the local convergences, there are two 
phases $ \beta_{\pm } \in \R $ such that 
\begin{equation}
\label{zoozoom} 
\hat{u}_\mathfrak{p} ( \hat{z}_{\mathfrak{p},\pm} + \mathfrak{p} \cdot ) \to \ex^{i \beta_\pm } V_\pm 
\end{equation}
in $ C^k_{\rm loc} (\R^2 ) $ for any $k \in \N $. We are then simply able to show that $ \beta_\pm = \Theta $, 
but this is not enough for the uniqueness result. This follows from the arguments given 
in \cite{ShafrirLinfini}, as we explain.

We work for the $+ $ sign. Integrating \eqref{accuphase} over the disk $ B( 0, R ) $ yields
\[
 \int_{ \p B ( 0, R ) } \hat{A}_\mathfrak{p}^2 \frac{ \p \hat{\varphi}_\mathfrak{p}}{\p \nu } \, d \ell 
 + c_\mathfrak{p} \mathfrak{p} \int_{ \p B( 0, R) } \nu_2 ( \hat{A}_\mathfrak{p}^2 - 1 ) \, d \ell 
 = 0 .
\]
We now consider the average
\[
 \beta_\mathfrak{p} (r) 
 \assign \frac{1}{2 \pi r} \int_{\p B (0,r ) } \hat{\varphi}_\mathfrak{p} \, d \ell 
\]
which satisfies, for $ 1/ \mathfrak{p} \ls r_0 \ls r_1 \ls 3/20 $,
\begin{align*}
 \beta_\mathfrak{p} ( r_0 ) - \beta_\mathfrak{p} (r_1) 
 & = \int_{r_0}^{r_1} \p_r \beta_\mathfrak{p} ( r ) \, d r 
 = \int_{r_0}^{r_1} \frac{1}{2 \pi r} \int_{\p B (0, r ) } \p_r \hat{\varphi}_\mathfrak{p} \, d \ell d r 
 \\ & 
 = \int_{r_0}^{r_1} \frac{1}{2 \pi r} \int_{\p B (0, r ) } ( 1 - \hat{A}_\mathfrak{p}^2 ) \p_r \hat{\varphi}_\mathfrak{p} \, d \ell d r 
 + c_\mathfrak{p} \mathfrak{p} 
 \int_{r_0}^{r_1} \frac{1}{2 \pi r} \int_{\p B (0, r ) } \nu_2 ( \hat{A}_\mathfrak{p}^2 - 1 ) \, d \ell d r .
\end{align*}
Therefore, by Step 5,
\begin{align*}
 \lvert \beta_\mathfrak{p} ( r_0 ) - \beta_\mathfrak{p} (r_1) \rvert
 & 
 \ls C \int_{r_0}^{r_1} \frac{dr}{\mathfrak{p}^2 r^3 } 
 + C \int_{r_0}^{r_1} \frac{d r }{ \mathfrak{p}^2 r^2 } 
 \ls \frac{C}{ ( r_0 \mathfrak{p} )^2 } + \frac{C}{ \mathfrak{p} } .
\end{align*}
We now fix $ \eta \in ] 0,1 ] $. Taking $ r_0 = 1 / ( \sqrt{\eta} \mathfrak{p} ) $ and $ r_1 = 3/20 $, we infer
\[
 \lvert \beta_\mathfrak{p} ( r_0 ) - \beta_\mathfrak{p} (r_1) \rvert
 \ls C \eta + \frac{C}{ \mathfrak{p} } . 
\]
Moreover, by \eqref{zoozoom}, we have
\[
 \beta_\mathfrak{p} ( r_0 ) 
 = \beta_\mathfrak{p} (1 / ( \sqrt{\eta} \mathfrak{p} ) ) 
 \to \beta_+ 
\]
as $ \mathfrak{p} \to + \ii $, and by Step 1 of subsection \ref{sec:cvhatx}, we deduce
\[
 \beta_\mathfrak{p} ( r_1 ) 
 \to \Theta . 
\]
As a consequence,
\[
 \lvert \beta_+ - \Theta \rvert \ls C \eta ,
\]
and the conclusion follows by letting $ \eta \to 0 $.

%%%%%%%%%%%%%%%%
\bibliographystyle{plain}
\bibliography{main.bib}

\end{document}